\documentclass[a4paper]{amsart}

\usepackage{mathtools}
\usepackage{amssymb}
\usepackage{eucal}
\usepackage{mathrsfs}
\usepackage[T1]{fontenc}
\usepackage{lmodern}
\usepackage{microtype}
\usepackage{xcolor}

\usepackage{booktabs}
\usepackage{thmtools}

\usepackage{graphicx}
\usepackage{tikz}
\usepackage{placeins}
\usetikzlibrary{arrows.meta,shapes.geometric}

\usepackage[
  backend=biber,
  style=alphabetic,
  sorting=nyt,
  isbn=false,
  doi=false,
  url=false,
  eprint=true,
  giveninits=true,
  backref=false,
  maxbibnames=99
]{biblatex}
\renewbibmacro{in:}{}

\usepackage{hyperref}
\hypersetup{
  colorlinks,
  linkcolor={red!80!black},
  citecolor={blue!80!black},
  urlcolor={blue!80!black}
}
\usepackage{cleveref}

\let\oldmarginpar\marginpar
\renewcommand\marginpar[1]{\-\oldmarginpar[\raggedleft\footnotesize #1]%
	        {\raggedright\footnotesize #1}}

\newskip\stdskip                      
\newcommand{\kk}{\mathbf{k}}
\newcommand{\ZZ}{\mathbb{Z}}
\newcommand{\C}{\mathbb{C}}
\newcommand{\R}{\mathbb{R}}
\newcommand{\PP}{\mathbb{P}}

\DeclareMathOperator{\CW}{CW}
\DeclareMathOperator{\CE}{CE}
\DeclareMathOperator{\HW}{HW}
\DeclareMathOperator{\Perf}{Perf}
\DeclareMathOperator{\Res}{Res}
\DeclareMathOperator{\Ree}{Re}
\renewcommand{\Im}{\operatorname{Im}}
\newcommand{\m}{\mathfrak{m}}

\numberwithin{equation}{section}
\declaretheorem[style=plain,sibling=equation]{theorem}
\declaretheorem[style=plain,sibling=equation]{proposition}
\declaretheorem[style=plain,sibling=equation]{lemma}
\declaretheorem[style=plain,sibling=equation]{corollary}
\declaretheorem[style=plain,sibling=equation]{conjecture}
\crefname{conjecture}{conjecture}{conjectures}
\Crefname{conjecture}{Conjecture}{Conjectures}
\declaretheorem[style=definition,sibling=equation]{definition}
\declaretheorem[style=remark,sibling=equation]{remark}

\declaretheorem[style=definition,numbered=no]{conventions}
\declaretheorem[style=definition,numbered=no,name=Acknowledgements]{acknowledgements}
\AddToHook{env/theorem/begin}{\crefalias{equation}{theorem}}
\AddToHook{env/proposition/begin}{\crefalias{equation}{proposition}}
\AddToHook{env/lemma/begin}{\crefalias{equation}{lemma}}
\AddToHook{env/corollary/begin}{\crefalias{equation}{corollary}}
\AddToHook{env/conjecture/begin}{\crefalias{equation}{conjecture}}
\AddToHook{env/definition/begin}{\crefalias{equation}{definition}}
\AddToHook{env/remark/begin}{\crefalias{equation}{remark}}
\AddToHook{env/example/begin}{\crefalias{equation}{example}}
\AddToHook{env/equation/begin}{\crefalias{equation}{equation}}
\AddToHook{env/align/begin}{\crefalias{equation}{equation}}
\AddToHook{env/gather/begin}{\crefalias{equation}{equation}}

\title{Kodaira fibres and wrapped Floer cohomology}
\author{Yank\i\ Lekili}
\date{}

\begin{document}
\begin{abstract}
Let $F$ be a singular fibre of a relatively minimal complex elliptic
fibration with smooth total space, and let $\Omega$ be a nonvanishing
holomorphic two-form near $F$. We show that a small neighbourhood of $F$
is a Weinstein domain for $\Ree\Omega$ whose completion is a Legendrian
surgery, with cocores obtained by completing holomorphic disks
transverse to the components of $F$. For any coefficient field and
every multiplicative bulk class, we compute the wrapped Floer cohomology
of these cocores and prove that it is concentrated in degree zero. The
cocores generate, so the bulk-deformed wrapped Fukaya category is
equivalent to the category of perfect modules over an explicit algebra:
a multiplicative preprojective algebra of affine type for the normal
crossing fibres, and a quiver algebra with relations for types $II$,
$III$ and $IV$. Applications include formality of the affine plumbing
dg algebras; mirrors given by resolved affine surfaces at every
classical bulk class for affine $D,E$ and types $II,III,IV$; and
mirrors given by quotient stacks of algebraic tori at root-of-unity
bulk classes for the four affine stars.
\end{abstract}
\maketitle

\section{Introduction}
\label{sec:introduction}

Let $\pi:X\to\Delta$ be a relatively minimal complex elliptic
fibration with smooth total space, proper over a disk, and with just
one singular fibre, \[ F=\sum_{i=1}^{r} m_iC_i,\] and let $\Omega$ be a
nonvanishing holomorphic two-form near $F$.
A holomorphic disk transverse to $C_i$ at a smooth point of
$F_{\mathrm{red}}$ is Lagrangian for $\Ree\Omega$, and its projection
to the base has degree $m_i$. Our goal is to determine the algebra of bulk-deformed wrapped
morphisms between the exact planes obtained by completing these disks.

Let $W=\{|\pi|\le\epsilon\}$ be a small neighbourhood of $F$. We show
in \Cref{sec:neighbourhoods} that $W$ is a Weinstein domain for
$\Ree\Omega$ whose completion $M=M_F$ is a Legendrian surgery determined by
the Kodaira type of $F$, and that these planes are the cocores $L_i$ of
its $2$-handles (\Cref{thm:neighbourhood-surgery}). For
the normal crossing types the attaching link is the link of the plumbing
of two-spheres according to the dual graph of $F$; for $II$, $III$ and
$IV$ it is the rainbow closure of $\sigma_1^3$, $\sigma_1^4$ and
$(\sigma_1\sigma_2)^3$.

Fix a coefficient field $\kk$. A multiplicative bulk class is a character
\[
 \mathfrak b:H_2(M;\ZZ)\longrightarrow\kk^\times,
\]
equivalently a class in $H^2(M;\kk^\times)$. It modifies Floer operations
by assigning multiplicative weights to holomorphic-curve counts, as
described in \Cref{subsec:bulk-data}. Since $M$ retracts onto $F$, the
component classes $[C_i]$ form a basis of $H_2(M;\ZZ)$, so
$\mathfrak b$ is specified by the independent parameters
\[
 q_i=\mathfrak b([C_i]),\qquad
 \boldsymbol q=(q_i)\in(\kk^\times)^r.
\]
We write $\mathcal W_{\mathfrak b}(M)$ for the split-closed bulk-deformed
wrapped Fukaya category. Its branes carry trivialisations of the bulk
class along their Lagrangians; the contractible planes $L_i$ therefore
define objects for every $\mathfrak b$. When all $q_i=1$, this is the
ordinary wrapped Fukaya category.

Write $\Lambda^{\boldsymbol q}(Q)$ for the multiplicative preprojective
algebra of Crawley-Boevey and Shaw \cite{CBShaw} with vertex parameters
$q_i$, recalled in \Cref{eq:vertex-products}, and let $A_F(\boldsymbol q)$
be the algebra in \Cref{tab:algebras}; the algebras $A_{III}$ and
$A_{IV}$ are the quotients of the path algebras of the quivers in
\Cref{fig:nontransverse-quivers} by the relations \eqref{eq:III-algebra}
and \eqref{eq:IV-algebra}.
Let $e_i\in A_F(\boldsymbol q)$ be the vertex idempotent corresponding
to $C_i$ and its cocore $L_i$. We give the cocores their holomorphic
gradings, which restrict to grading zero on the transverse holomorphic
disks (\Cref{subsec:holomorphic-gradings}).

\begin{table}[ht]
\centering
\begin{tabular}{@{}ll@{}}
\toprule
Fibre & $A_F(\boldsymbol q)$\\
\midrule
$I_n$, $n\ge1$ & $\Lambda^{\boldsymbol q}(\widetilde A_{n-1})$\\
$I_n^*$, $n\ge0$ & $\Lambda^{\boldsymbol q}(\widetilde D_{n+4})$\\
$IV^*,III^*,II^*$ &
 $\Lambda^{\boldsymbol q}(\widetilde E_6),
  \Lambda^{\boldsymbol q}(\widetilde E_7),
  \Lambda^{\boldsymbol q}(\widetilde E_8)$\\
$II$ & $\kk\langle x,y,z\rangle/(q+x+z+zyx,\;1+x+z+xyz)$\\
$III,IV$ & $A_{III}(q_1,q_2),\ A_{IV}(q_1,q_2,q_3)$\\
\bottomrule
\end{tabular}
\par\vspace{6pt}
\caption{The algebras associated to Kodaira fibres; see
\Cref{eq:vertex-products,eq:cusp-algebra,eq:III-algebra,eq:IV-algebra}.}
\label{tab:algebras}
\end{table}

\begin{theorem}\label{thm:main}
Let $F$ be a singular fibre, $M$ the completion of a
neighbourhood of $F$, and $L_1,\ldots,L_r$ the cocores obtained by
completing holomorphic disks transverse to the components
$C_1,\ldots,C_r$ at smooth points of $F_{\mathrm{red}}$, with their
holomorphic gradings. For any field $\kk$ and every
$\mathfrak{b}\in H^2(M;\kk^\times)$, the Lagrangians $L_i$ generate
$\mathcal W_{\mathfrak{b}}(M)$ and
\[
\HW_{\mathfrak{b}}^*(L_i,L_j;\kk)=
\begin{cases}
e_jA_F(\boldsymbol q)e_i,&*=0,\\
0,&*\ne0.
\end{cases}
\]
Their endomorphism $A_\infty$ algebra has minimal model
$A_F(\boldsymbol q)$. Consequently,
\[
 \mathcal W_{\mathfrak{b}}(M)
 \simeq\Perf\, A_F(\boldsymbol q).
\]
\end{theorem}

We first realise $M$ as a Legendrian surgery whose cocores are the
completions of the transverse holomorphic disks, and compute the
Chekanov-Eliashberg dg algebra of its attaching link with bulk
parameters. Holomorphic graph-index calculations
show that the wrapped Floer complex is supported in degrees zero and
one. On the other hand, the surgery algebra is concentrated in
nonpositive degrees, and the comparison of \Cref{prop:bulk-cocores}
identifies its cohomology with the wrapped endomorphisms. Therefore,
we conclude that cohomology is supported in degree $0$ and that the
canonical projection to degree-zero cohomology is a quasi-isomorphism.
Cocore generation then identifies $\mathcal W_{\mathfrak b}(M)$
with $\Perf A_F(\boldsymbol q)$.

\begin{corollary}\label{cor:affine-formality}
For any affine diagram $Q$ occurring in \Cref{tab:algebras}, any
field $\kk$, and any $\boldsymbol q\in(\kk^\times)^{Q_0}$, the
canonical projection
\[
 \mathcal B_Q^{\boldsymbol q}\longrightarrow
 H^0(\mathcal B_Q^{\boldsymbol q})\cong\Lambda^{\boldsymbol q}(Q)
\]
is a quasi-isomorphism. Here $\mathcal B_Q^{\boldsymbol q}$ is the
plumbing dg algebra defined in \Cref{def:plumbing-dga}.
\end{corollary}

\begin{remark}\label{rem:ks-conjectures}
Kaplan and Schedler conjecture that the dg multiplicative preprojective
algebra of every connected non-Dynkin quiver is quasi-isomorphic to its
degree-zero cohomology \cite[Conjecture~1.3]{KS}. They prove this for
quivers containing a cycle \cite[Theorem~3.7 and Proposition~4.4]{KS}.
\Cref{cor:affine-formality} establishes the conjecture for every extended
Dynkin quiver, over any field and at every parameter tuple; the new
cases are the affine $D$ and $E$ diagrams. The comparison with the
plumbing dg algebra is explained in \Cref{subsec:plumbing-dga}.
Together with Kaplan-Schedler's self-duality theorem, this also
proves the $2$-Calabi-Yau assertion of their Conjecture~1.1
\cite{KS} for affine $D$ and $E$ quivers, over any
field and at every nonzero parameter tuple; see \Cref{lem:star-cy}.
\end{remark}

For $F=I_n$ ($n\ge1$), $I_n^*$ ($n\ge0$), $IV^*$, $III^*$ and $II^*$, let
$Q=\widetilde{A}_{n-1}$, $\widetilde D_{n+4}$, $\widetilde E_6$, $\widetilde E_7$ and $\widetilde E_8$,
respectively, and write $M_Q$ for the completion of a neighbourhood of
$F$, which by \Cref{thm:neighbourhood-surgery} is a plumbing of cotangent
bundles of two-spheres according to $Q$.
For any fibre type, we call $\boldsymbol q$ classical if
$\prod_iq_i^{m_i}=\mathfrak b([F])=1$, and put
\[
 S_Q(\boldsymbol q)=\operatorname{Spec}Z\bigl(\Lambda^{\boldsymbol q}(Q)\bigr),
 \qquad S_Q=S_Q(\boldsymbol1).
\]
For affine $D$ and $E$ and classical parameters $\boldsymbol q$,
$S_Q(\boldsymbol q)$ is a normal Gorenstein surface over any field
by \Cref{prop:affine-de-nccr}.

\begin{corollary}\label{cor:unipotent-mirror}
Over any field, let $Q$ be affine of type $D$ or $E$ and let
$\mathfrak b$ have classical parameters $\boldsymbol q$.
Set $A=\Lambda^{\boldsymbol q}(Q)$.
For an extending vertex idempotent $e$, the Satake map
$Z(A)\to eAe$ is an isomorphism, and
$A\cong\operatorname{End}_{eAe}(Ae)$ is a noncommutative crepant
resolution. The surface $S_Q(\boldsymbol q)$ admits
a projective crepant minimal resolution
$\widetilde S_Q(\boldsymbol q)\to S_Q(\boldsymbol q)$ which is smooth, and
\[
 \mathcal W_{\mathfrak b}(M_Q)\simeq
 \Perf\bigl(\Lambda^{\boldsymbol q}(Q)\bigr)
 \simeq\Perf\bigl(\widetilde S_Q(\boldsymbol q)\bigr).
\]
For $F=II,III,IV$ and classical $\boldsymbol q$, put
$S_F(\boldsymbol q)=\operatorname{Spec}Z(A_F(\boldsymbol q))$.
The algebra $A_F(\boldsymbol q)$ is an NCCR of this normal Gorenstein
surface, whose projective crepant minimal resolution
$\widetilde S_F(\boldsymbol q)$ is smooth, and
\[
 \mathcal W_{\mathfrak b}(M_F)\simeq\Perf A_F(\boldsymbol q)
 \simeq\Perf\bigl(\widetilde S_F(\boldsymbol q)\bigr).
\]
\end{corollary}

Thus every classical bulk parameter gives a commutative mirror
in these cases. At the trivial bulk class, write
$\widetilde S_Q=\widetilde S_Q(\boldsymbol1)$; the equations of $S_Q$ are
\[
\begin{array}{c|l}
 \widetilde A_{n-1}&xyz+xy+z^n=0\\
 \widetilde D_{n+4}&xyz+xy^2-p_{n-1}x^2y-p_nxz-z^2=0\\
 \widetilde E_6&xyz+z^2+x^2z+y^3=0\\
 \widetilde E_7&xyz+z^2+y^3+x^3y=0\\
 \widetilde E_8&xyz+z^2+y^3+x^5=0.
\end{array}
\]
Here $p_{-1}=-1$, $p_0=0$, and $p_{j+1}=x(p_j+p_{j-1})$.
These equations come from Shaw's thesis \cite[Theorem~4.1.1]{ShawThesis}
(see also \cite[Theorem~6.4 and Corollary~6.7]{KS}), with the type $D$
sign corrected, as explained in \Cref{subsec:surface}.
For affine $D$ and $E$, the origin is the only geometric singular
point of $S_Q$, a $\kk$-rational du Val singularity of geometric
type $D_{n+4}$, $E_6$, $E_7$ or $E_8$, respectively.

For $III$ and $IV$, the classical conditions are $q_1q_2=1$ and
$q_1q_2q_3=1$, respectively. Put $q=q_1$ for $III$, and
$\alpha=q_1$, $\beta=q_1q_2=q_3^{-1}$ for $IV$.
Their surfaces $S_F(\boldsymbol q)$, together with the type $II$
surface (whose classical condition forces $q=1$), have equations
\[
\begin{array}{c|l}
 II&xyz+x+z+1=0,\\
 III&xyz+x+y-qz+1+q=0,\\
 IV&xyz+x^2+(1+\alpha+\beta)x+\beta y-\alpha z
       +\alpha\beta+\alpha+\beta=0.
\end{array}
\]
The type $II$ surface is smooth. The type $III$ surface is smooth
unless $q=1$, when it has one split $A_1$ singularity. The type
$IV$ surface is smooth if no $q_i=1$, has one split $A_1$
singularity if exactly one $q_i=1$, and one split $A_2$ singularity
if all $q_i=1$. Write $S_F=S_F(\boldsymbol1)$ and
$\widetilde S_F=\widetilde S_F(\boldsymbol1)$; at this parameter,
the singular point for $III$ and $IV$ is $(-1,-1,1)$.

\begin{remark}\label{rem:painleve}
For types $II$, $III$, $IV$ and $I_0^*$, the categorical mirror
descriptions over $\C$ at the trivial bulk class overlap with
Beimler, Hu, Olsen and Shende \cite[Theorem~1.1]{BHOS}.
These overlaps arise from Hitchin systems. The proofs are
different, via microlocal sheaves and
toric mirror symmetry there, and via the algebras
$\Lambda^{\boldsymbol1}(Q)$ and $A_F(\boldsymbol1)$ here.

At the trivial bulk class, \Cref{cor:unipotent-mirror} should also
be obtainable, for every $Q$,
from the general homological mirror symmetry theorems of
Hacking-Keating \cite{HK} and Varolg\"une\c{s} \cite{Var} for affine
log Calabi-Yau surfaces, once $M_Q$ is identified with the Liouville
completion of such a surface and their mirror with $\widetilde S_Q$;
some additional work is required.
\end{remark}

\begin{remark}\label{rem:cycles}
For the cycles $I_n$, put
\[
 R=\kk[x,y,z^{\pm1}]/(xy-z+1).
\]
Suppose that $n$ is invertible in $\kk$ and that $\kk$ contains a
primitive $n$th root of unity. \Cref{thm:main,prop:cycle-algebra} give
\[
 \mathcal W(M_{I_n})\simeq\Perf(R\rtimes\mu_n)
 \simeq\Perf\bigl([\operatorname{Spec}R/\mu_n]\bigr),
\]
where $\mu_n$ acts on $x,y,z$ with weights $1,-1,0$. The algebraic
properties of the cycle algebras are established in
\cite[Theorem~1.2 and Section~6.2]{KS}. For $n=1$, this mirror
equivalence is due to Pascaleff \cite{PascaleffBinodal}; the general
case follows from a covering argument (cf.\ \cite{ChanUeda}).
\end{remark}

For the four affine stars, let $(E,G)$ be a pair consisting of an
elliptic curve $E$ and a cyclic group $G\cong\mu_N$ of automorphisms
fixing its origin. Let $s$ be the number of cone points of $E/G$.
The group order and cone-point orders are:
\begin{center}
\begin{tabular}{@{}ccc@{}}
\toprule
$Q$ & $N$ & $(d_1,\ldots,d_s)$\\
\midrule
$\widetilde D_4$ & $2$ & $(2,2,2,2)$\\
$\widetilde E_6$ & $3$ & $(3,3,3)$\\
$\widetilde E_7$ & $4$ & $(2,4,4)$\\
$\widetilde E_8$ & $6$ & $(2,3,6)$\\
\bottomrule
\end{tabular}
\end{center}
Label the central vertex by $0$ and the vertices on the $j$th arm
by $(j,p)$, $1\le p<d_j$, starting next to the centre, and orient
the arrows towards the centre. If $\kk$ contains primitive
$d_j$th roots $\zeta_j$, define the bulk class $\mathfrak b_{\mathrm{grp}}$ by
\[
 q_0=1,\qquad q_{(j,p)}=\zeta_j,
\]
and denote this parameter tuple by $\boldsymbol q_{\mathrm{grp}}$.
Let $\mathbb{T}$ be the underlying real torus of $E$, and put
\[
 \Gamma=\pi_1(\mathbb{T})\cong\ZZ^2,\qquad
 \mathbb G_{m,\kk}^2=\operatorname{Spec}\kk[\Gamma],
\]
with the induced $G$-action.

\begin{corollary}\label{cor:orbifold-mirror}
Suppose that $\operatorname{char}\kk\nmid N$ and that the primitive
roots $\zeta_j$ lie in $\kk$. For any such choice,
\[
 \mathcal W_{\mathfrak b_{\mathrm{grp}}}(M_Q)
 \simeq\mathcal W_G(T^*\mathbb{T})
 \simeq\Perf\bigl([\mathbb G_{m,\kk}^2/G]\bigr).
\]
Here $\mathcal W_G$ denotes the equivariant wrapped Fukaya category,
with the grading and relative spin background specified in
\Cref{sec:equivariant}. All three categories have the algebraic model
$\Perf(\kk[\Gamma]\rtimes G)$.
\end{corollary}

The equivalence between the equivariant cotangent category and the
quotient-stack category requires only
$\operatorname{char}\kk\nmid N$; the roots are needed for the
comparison with the plumbing. Proofs of
\Cref{cor:unipotent-mirror,cor:orbifold-mirror} are given in
\Cref{sec:star-consequences}.

\subsection{Crepant resolutions and Hecke algebras}

For the four affine stars, we choose the complex model
\[
 M_Q=\operatorname{Res}(T^*E/G),
\]
the minimal resolution of $T^*E/G$, identified with the completion
above by an exact symplectomorphism (\Cref{prop:quotients}).
The exceptional curves over a
cone point of order $d_j$ form the $A_{d_j-1}$ arm, while the
strict transform of the zero section is the central component.
In the cotangent coordinate $v$, the function $v^N$ defines a
proper map $f:M_Q\to\C$.

Our choice of roots of unity for the bulk parameters on the
exceptional spheres parallels the crepant resolution conjecture
of Ruan \cite[Section~2]{RuanCRC} and Bryan-Graber
\cite[Conjecture~1.2]{BryanGraber}.
That conjecture relates the genus-zero Gromov-Witten theories
of a quotient orbifold and its crepant resolution: after analytic
continuation and a change of cohomological variables, the
parameters associated with exceptional curves are specialised to
roots of unity, while the remaining parameters stay free.
For the local $A_{d_j-1}$ singularities, this comparison is proved
in equivariant quantum cohomology, with exceptional parameters
$e^{-2\pi i/d_j}$ \cite[Theorem~A.1 and Corollary~A.2]{CCIT};
related ADE computations appear in \cite{BryanGholampour}.
Our wrapped Fukaya category calculation exhibits the same pattern
of parameter specialisation, but does not establish the
Gromov-Witten correspondence for $M_Q$.

Allowing the bulk parameter on the central sphere to vary deforms
the quotient-torus mirror to a quantum torus.
Write $q=q_0$. The parameters $q_{(j,p)}$
belong to the exceptional spheres, and $q$ belongs to the central
sphere. Set $q_{(j,p)}=\zeta_j$ and allow arbitrary $q\in\kk^\times$;
denote the resulting bulk class by $\mathfrak b_{\mathrm{grp},q}$.
Under the field hypotheses of \Cref{cor:orbifold-mirror}, we obtain
\begin{equation}\label{eq:intro-quantum-mirror}
 \mathcal W_{\mathfrak b_{\mathrm{grp},q}}(M_Q)
 \simeq\Perf\bigl(\mathcal A_{q^N}\rtimes G\bigr),\qquad
 \mathcal A_\eta=
 \kk\langle X^{\pm1},Y^{\pm1}\rangle/(YX-\eta XY).
\end{equation}
The monomial action is the one described in \cite[Section~4.1]{EOR}.
At $q=1$ this is \Cref{cor:orbifold-mirror}.

Near the group-algebra point, there is also a comparison over a
formal power series ring in the bulk parameters, with coefficients
in $\C$. In this setting, \cite[Theorem~1.6 and Example~7.4]{HKTY}
identifies the cohomology of the bulk-deformed regular cotangent
fibre for $[E/G]$ with Etingof's orbifold Hecke algebra
\cite{EtingofOrbifold}, with constant terms $-1$. It is concentrated
in degree zero since $E$ is a $K(\pi,1)$. This computation concerns
the fibre algebra; split-generation of that formal orbifold
category is a further question.

The central corner $H=e_0\Lambda^{\boldsymbol q}(Q)e_0$ is the rank-one
generalised double affine Hecke algebra of Etingof, Oblomkov and Rains
\cite{EOR}: with $\alpha_{ja}=\prod_{p=1}^{a}q_{(j,p)}^{-1}$, it is
generated by $U_1^{\pm1},\ldots,U_s^{\pm1}$ subject to
\[
 \prod_{a=0}^{d_j-1}(U_j-\alpha_{ja})=0\quad(1\le j\le s),\qquad
 U_1\cdots U_s=q,
\]
see \Cref{subsec:central-corner}. \Cref{sec:hecke-tensor} gives the
parameter dictionary relating this presentation to Etingof's formally
deformed orbifold Hecke algebra.

\subsection{Hilbert schemes and higher rank}

For $n\ge1$, put
\[
 Y_n=\operatorname{Hilb}^n(M_Q),\qquad
 G_n=G^{\times n}\rtimes S_n,\qquad \mathbb{T}_n=\mathbb{T}^{\times n}.
\]
The Hilbert scheme is smooth of complex dimension $2n$, with a
projective symplectic resolution
\[
 Y_n\longrightarrow\operatorname{Sym}^n(T^*E/G)
       \cong (T^*E)^n/G_n.
\]
The holomorphic symplectic form $\Omega_n$ induced from $M_Q$
has weight one under cotangent scaling. If $V_n$ generates positive
real scaling, then $\lambda_n=\iota_{V_n}\Ree\Omega_n$ is a
finite-type Liouville form. Its skeleton is the underlying set of
the zero fibre of the proper map
\[
 h_n:Y_n\longrightarrow\operatorname{Sym}^n(M_Q)
          \longrightarrow\operatorname{Sym}^n(\C)\cong\C^n.
\]
The first arrow is the Hilbert-Chow map, sending a length-$n$
subscheme to its support counted with multiplicities; the second is
induced by $f$. The skeleton includes nonreduced subschemes with
transverse directions when $n>1$. Groechenig's description
\cite[Theorems~4.1 and~5.1]{Groechenig} realises these complex
spaces as moduli of stable parabolic Higgs bundles of rank $nN$.

The higher-rank generalised double affine Hecke algebra
$H_n(\boldsymbol u,\tau)$ of Etingof-Gan-Oblomkov \cite{EGO}
has an interaction parameter $\tau$. Put
\[
 \Lambda_n^{\boldsymbol q}
   =\bigl(\Lambda^{\boldsymbol q}(Q)\bigr)^{\otimes n}\rtimes S_n.
\]
At $\tau=1$, over any coefficient field,
\begin{equation}\label{eq:higher-noninteracting}
 H_n(\boldsymbol u,1)
   \cong e_0^{\otimes n}\Lambda_n^{\boldsymbol q}e_0^{\otimes n}.
\end{equation}
\Cref{sec:hecke-tensor} gives the presentations, parameter dictionary
$(\boldsymbol u\leftrightarrow\boldsymbol q)$ and corresponding dg
formality statement.

When the roots $\zeta_j$ lie in $\kk$, the corner at
$\boldsymbol q_{\mathrm{grp}}$ is full and isomorphic to
$\kk[\Gamma^n]\rtimes G_n$. \Cref{thm:higher-orbifold} gives the
corresponding equivariant mirror equivalence when $|G_n|$ is
invertible in $\kk$.

\begin{theorem}\label{thm:higher-orbifold}
For every affine star and every $n\ge1$, suppose that
$N^n n!=|G_n|$ is invertible in $\kk$. Then
\[
 \mathcal W_{G_n}(T^*\mathbb{T}_n)
 \simeq\Perf\bigl(\kk[\Gamma^n]\rtimes G_n\bigr)
 \simeq\Perf([\mathbb G_{m,\kk}^{2n}/G_n]).
\]
The brane data are the canonical cotangent grading and the
equivariant relative spin background $\pi^*T\mathbb{T}_n$, where
$\pi:T^*\mathbb{T}_n\to\mathbb{T}_n$ is the cotangent projection.
If $\kk$ contains the roots $\zeta_j$, these categories are also
equivalent to $\Perf(\Lambda_n^{\boldsymbol q_{\mathrm{grp}}})$.
\end{theorem}

The theorem concerns the equivariant cotangent category. For the
smooth resolution $Y_n$ we formulate a different comparison.
Let $\mathcal V_n$ be the tautological rank-$n$ bundle, whose fibre
at $Z$ is $H^0(Z,\mathcal O_Z)$, and define
\begin{equation}\label{eq:collision-bulk}
 \mathfrak b_{\mathrm{coll}}(\beta)
    =(-1)^{\langle c_1(\det\mathcal V_n),\beta\rangle},
 \qquad \beta\in H_2(Y_n;\ZZ).
\end{equation}
This character is $1$ on classes coming from $H_2(M_Q;\ZZ)$ and $-1$
on a Hilbert-Chow exceptional line. In the conjecture below,
$\mathcal W_{\mathfrak b_{\mathrm{coll}}}(Y_n)$ denotes the wrapped
Fukaya category with relative spin background
$c_1(\det\mathcal V_n)\bmod2$; its brane data are specified in
\Cref{sec:hilbert-schemes}.

The collision weight $-1$ is motivated by the Hilbert-Chow crepant
resolution correspondence: in the local model, the undeformed orbifold
point corresponds to exceptional quantum parameter $-1$
\cite[Theorem~3.11]{BryanGraber}. The character
$\mathfrak b_{\mathrm{coll}}$ gives the corresponding sign on collision
classes.

\begin{conjecture}\label{conj:hilbert-mirror}
Assume that $n!$ is invertible in $\kk$.
Put $Y_n^\vee=\operatorname{Hilb}^n(\widetilde S_Q)$ over $\kk$.
Then
\[
 \mathcal W_{\mathfrak b_{\mathrm{coll}}}(Y_n)
       \simeq\Perf(Y_n^\vee).
\]
\end{conjecture}

The algebraic side of the conjecture is established in
\Cref{prop:hilbert-mirror-model}:
\[
 \Perf(Y_n^\vee)\simeq\Perf(\Lambda_n^{\boldsymbol1}).
\]
Here $\tau=1$ and all vertex parameters equal $1$. The idempotent
$e_0^{\otimes n}$ is not full, so the Hecke corner does not recover
the category of perfect $\Lambda_n^{\boldsymbol1}$-modules. The remaining
step is the symplectic comparison across the Hilbert-Chow
exceptional locus. For $n=1$, the collision class is trivial and
the conjecture reduces to the trivial-bulk specialisation of
\Cref{cor:unipotent-mirror}.

\Cref{sec:hilbert-schemes} also explains the collision weight $-1$
through the local Hecke relation. For general interaction parameter
$\tau$, the tensor-product algebra must itself be deformed;
constructing the corresponding dg model and computing the wrapped
category remain open.

\subsection{Organisation}

\Cref{sec:neighbourhoods} constructs the Weinstein neighbourhoods
and identifies their surgery descriptions. \Cref{sec:algebras}
computes the associated bulk-deformed algebras.
\Cref{sec:foundations} constructs cofinal Floer complexes in degrees
zero and one. \Cref{sec:surgery} compares their gradings with those
of the surgery algebras, deduces formality, and completes the proof
of \Cref{thm:main} using cocore generation.
\Cref{sec:star-consequences} proves the mirror corollaries, the
higher-rank equivariant theorem and the algebraic side of the
Hilbert-scheme conjecture.

\begin{conventions}
The geometric models are defined over $\C$; Floer complexes and algebras
have coefficients in $\kk$. The planes have their unique spin structures,
trivial rank-one local systems, and compatible bulk trivialisations.
Signed curve counts are interpreted in $\kk$, and wrapping is unstopped.
We fix an arbitrary $\mathfrak{b}$ throughout; the ordinary category is
its trivial specialisation. Products are read from right to left:
$vu$ means $u$ followed by $v$, and $e_jA_F(\boldsymbol q)e_i$ is the
space of morphisms from $L_i$ to $L_j$.

For Liouville and Weinstein domains we follow
\cite[Section~11.1 and Definition~11.10]{CE}. Our sign
convention is $\iota_Z\omega=\lambda$, where $d\lambda=\omega$.
The completion of a Liouville domain $W$ is
\[
 \widehat W=W\cup\bigl([0,\infty)\times\partial W\bigr),
 \qquad \lambda=e^\rho\lambda|_{\partial W}
\]
on the cylindrical end. A Lagrangian disk $D\subset W$ with
$\lambda|_D=0$ near its boundary extends along the Liouville flow
to an exact cylindrical Lagrangian plane $\widehat D\subset\widehat W$.
\end{conventions}

\begin{acknowledgements}
The author thanks Umut Varolg\"une\c{s} for many discussions, which
played an important role in the genesis of this paper, and acknowledges
the papers \cite{EOR}, \cite{KS} and \cite{Pascaleff}, which together
inspired the ideas here. The author also thanks Roman Krutowski for his
comments on an earlier version. The text was mostly written by ChatGPT (Astra)
for Academic Researchers and Claude Fable 5.1. The
statements of the main results, the overall strategy of proof and the
key ideas behind each step are due to the author, who provided the
general set-up, an outline of the proof of \Cref{thm:main}, and precise
statements to be proved in each section. The AI systems produced
detailed arguments from these inputs. The author checked the resulting
manuscript, expanded it with further corollaries, remarks, figures and
connections to the literature, and revised the exposition for
readability. The author accepts full responsibility for the contents of
the paper.
\end{acknowledgements}

\section{Neighbourhoods of singular fibres as Legendrian surgeries}
\label{sec:neighbourhoods}

We construct a Weinstein neighbourhood of the singular fibre and
identify its Legendrian surgery description, with transverse
holomorphic disks completing to the cocores.

Let $\pi:X\to\Delta$ be a proper relatively minimal complex elliptic
fibration with smooth total space and a single singular fibre
\[
 F=\pi^{-1}(0)=\sum_{i=1}^k m_iC_i.
\]
Fix a nowhere-vanishing holomorphic two-form $\Omega$ on $X$ and put
$\omega=\Ree\Omega$. The existence of $\Omega$ forces $F$ to be
nonmultiple, by adjunction and the fact that the normal bundle of
the reduced fibre of a multiple fibre $mD$ has exact order $m$
\cite[Section~15, Lemmas~1 and~3 and Corollary~4]{PetersSurfaces}.
Complex curves are Lagrangian for $\omega$, whose symplectic
orientation agrees with the complex orientation.

Write $F_{\mathrm{red}}=\bigcup_iC_i$, let $\Sigma$ be its singular
locus, and let $\nu_i:\widetilde C_i\to C_i$ be the normalisations.
Choose pairwise disjoint holomorphic disks $D_i$ meeting $F$ only at
points $x_i\in C_i\setminus\Sigma$, transversally to $C_i$. Put
\[
 W_\epsilon=\{|\pi|\le\epsilon\},
 \qquad
 \mu=\log|\pi|:X\setminus F\longrightarrow\R.
\]
Properness implies that the $W_\epsilon$ form a neighbourhood basis of
$F$. For sufficiently small $\epsilon>0$, each $W_\epsilon$ is a compact
manifold with boundary $\{|\pi|=\epsilon\}$, and each
$D_i\cap W_\epsilon$ is a disk.

\Cref{fig:schematic-D4} shows the situation schematically for a fibre of
type $I_0^*$.

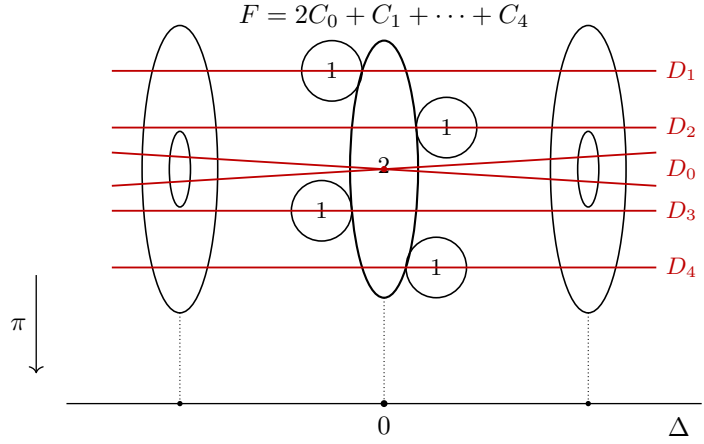
\begin{figure}[htbp]
\centering
\begin{tikzpicture}[x=1cm,y=1cm,line width=.6pt]
\draw[line width=.5pt] (-4.2,-3.1) -- (4.2,-3.1);
\node[below] at (0,-3.15) {$0$};
\node[below] at (3.9,-3.15) {$\Delta$};
\fill (0,-3.1) circle (1.3pt);
\draw[->,line width=.5pt] (-4.6,-1.4) -- (-4.6,-2.7);
\node[left] at (-4.6,-2.05) {$\pi$};
\foreach \x in {-2.7,2.7} {
  \draw (\x,0) ellipse (0.5cm and 1.9cm);
  \draw (\x,0) ellipse (0.14cm and 0.5cm);
  \fill (\x,-3.1) circle (1pt);
  \draw[densely dotted,line width=.4pt] (\x,-1.9) -- (\x,-3.1);
}
\draw[line width=.8pt] (0,0) ellipse (0.45cm and 1.7cm);
\node at (0,0.05) {\small $2$};
\draw (-0.69,1.3) circle (0.4cm);  \node at (-0.69,1.3) {\small $1$};
\draw (0.826,0.55) circle (0.4cm); \node at (0.826,0.55) {\small $1$};
\draw (-0.826,-0.55) circle (0.4cm); \node at (-0.826,-0.55) {\small $1$};
\draw (0.69,-1.3) circle (0.4cm);  \node at (0.69,-1.3) {\small $1$};
\draw[densely dotted,line width=.4pt] (0,-1.7) -- (0,-3.1);
\node[above] at (0,1.75) {$F=2C_0+C_1+\dots+C_4$};
\draw[red!75!black,line width=.7pt] (-3.6,1.3) -- (3.6,1.3);
\draw[red!75!black,line width=.7pt] (-3.6,0.55) -- (3.6,0.55);
\draw[red!75!black,line width=.7pt] (-3.6,-0.55) -- (3.6,-0.55);
\draw[red!75!black,line width=.7pt] (-3.6,-1.3) -- (3.6,-1.3);
\node[red!75!black,right] at (3.6,1.3) {\small $D_1$};
\node[red!75!black,right] at (3.6,0.55) {\small $D_2$};
\node[red!75!black,right] at (3.6,-0.55) {\small $D_3$};
\node[red!75!black,right] at (3.6,-1.3) {\small $D_4$};
\draw[red!75!black,line width=.7pt] (-3.6,-0.22) .. controls (-1.2,-0.07) and (1.2,0.07) .. (3.6,0.22);
\draw[red!75!black,line width=.7pt] (-3.6,0.22) .. controls (-1.2,0.07) and (1.2,-0.07) .. (3.6,-0.22);
\node[red!75!black,right] at (3.6,0.0) {\small $D_0$};
\fill[red!75!black] (0,0) circle (1.2pt);
\end{tikzpicture}
\caption{Schematic picture of the fibration near a fibre of type
$I_0^*$. The central component $C_0$ has multiplicity $2$ and the four
other components multiplicity $1$; nearby fibres are tori. The
holomorphic disks $D_1,\dots,D_4$ (red) through the multiplicity-one
components are local sections, whereas the disk $D_0$ through $C_0$
projects to $\Delta$ with degree $2$ and meets each nearby fibre
twice.}
\label{fig:schematic-D4}
\end{figure}

For the normal crossing fibres, the dual graph $\Gamma_F$ has one
vertex for each irreducible component and one edge for each node.
Kodaira's classification \cite{Kodaira63}, see also
\cite[Section~V.7]{BHPV}, gives:
\begin{center}
\begin{tabular}{@{}ll@{}}
\toprule
Fibre type & Dual graph $\Gamma_F$ \\
\midrule
$I_n$, $n\ge1$ & $\widetilde A_{n-1}$ \\
$I_n^*$, $n\ge0$ & $\widetilde D_{n+4}$ \\
$IV^*$ & $\widetilde E_6$ \\
$III^*$ & $\widetilde E_7$ \\
$II^*$ & $\widetilde E_8$ \\
\bottomrule
\end{tabular}
\end{center}
Here $I_1$ is a rational curve with one node, represented by a single
vertex with a loop; $I_2$ has two vertices joined by two edges. In the
remaining normal crossing cases, the components are smooth rational
curves meeting transversally.

The remaining types $II$, $III$, $IV$ have all multiplicities equal to
one and a single singular point $p$. They consist, respectively, of a
rational cuspidal curve, two smooth rational curves meeting at $p$ with
intersection multiplicity two, and three smooth rational curves meeting
pairwise transversally at $p$.

\begin{theorem}\label{thm:neighbourhood-surgery}
For all sufficiently small $\epsilon>0$, $W_\epsilon$ carries a
Weinstein structure $(\lambda,\phi)$ with $d\lambda=\omega$ and
skeleton $F_{\mathrm{red}}$. The form $\lambda$ restricts to zero on
$F_{\mathrm{red}}$ and on each $D_i\cap W_\epsilon$.
Up to Weinstein homotopy, this structure has the following surgery
descriptions, with the completed disks $\widehat D_i$ as the cocores
corresponding to the components $C_i$.
\begin{enumerate}
\item For the normal crossing types, $W_\epsilon$ is the positive
plumbing of disk cotangent bundles of two-spheres according to
$\Gamma_F$, with $D_i\cap W_\epsilon$ corresponding to the cotangent
fibre at $x_i$. Its completion is Legendrian surgery on
\[
 L_{\Gamma_F}\subset
 \#^{b_1(\Gamma_F)}(S^1\times S^2,\xi_{\mathrm{std}}),
\]
shown in \Cref{fig:surgery-In,fig:surgery-Dn,fig:surgery-stars}.
Here $b_1(\Gamma_F)=1$ for $I_n$ and $0$ otherwise.
\item For types $II$, $III$, $IV$, $W_\epsilon$ is obtained from a
Darboux ball by attaching one Weinstein $2$-handle for each component
$C_i$. The attaching link $\Lambda_F=\bigsqcup_i\Lambda_i$ is
Legendrian isotopic to the rainbow closure of
\[
 \sigma_1^3,\qquad \sigma_1^4,\qquad (\sigma_1\sigma_2)^3,
\]
respectively, shown in \Cref{fig:rainbow-fronts}. For $III$ and $IV$,
the isotopy can realise any prescribed correspondence between the
labelled components.
\end{enumerate}
\end{theorem}

The proof of this theorem occupies the rest of this section.

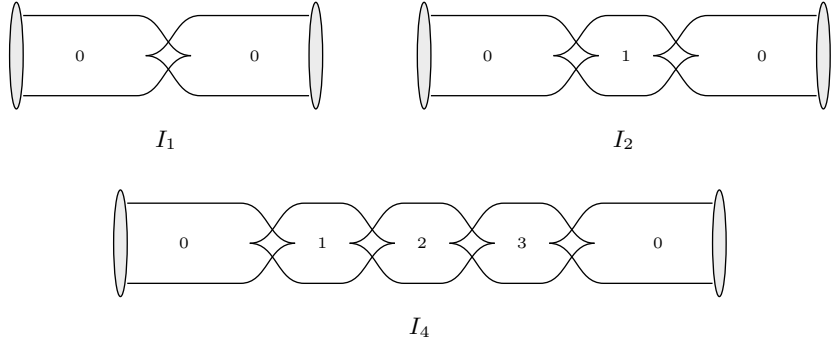
\begin{figure}[htbp]
\centering
\begin{tabular}{@{}c@{}}\begin{tikzpicture}[x=0.3cm,y=0.5cm,line width=.5pt]
\draw[fill=gray!15] (-1.300,0) ellipse (0.090cm and 0.705cm);\draw[fill=gray!15] (11.900,0) ellipse (0.090cm and 0.705cm);
\draw (-1.000,1.060) -- (4.000,1.060) .. controls (5.200,1.060) and (5.200,0.000) .. (6.400,0.000);
\draw (-1.000,-1.060) -- (4.000,-1.060) .. controls (5.200,-1.060) and (5.200,0.000) .. (6.400,0.000);
\draw (4.400,0.000) .. controls (5.600,0.000) and (5.600,1.060) .. (6.800,1.060) -- (11.600,1.060);
\draw (4.400,0.000) .. controls (5.600,0.000) and (5.600,-1.060) .. (6.800,-1.060) -- (11.600,-1.060);
\node[font=\tiny,inner sep=0pt] at (1.500,0.000) {$0$};
\node[font=\tiny,inner sep=0pt] at (9.200,0.000) {$0$};
\end{tikzpicture}\\[2pt]{\small $I_1$}\end{tabular}
\hspace{10mm}
\begin{tabular}{@{}c@{}}\begin{tikzpicture}[x=0.3cm,y=0.5cm,line width=.5pt]
\draw[fill=gray!15] (-1.300,0) ellipse (0.090cm and 0.705cm);\draw[fill=gray!15] (16.300,0) ellipse (0.090cm and 0.705cm);
\draw (-1.000,1.060) -- (4.000,1.060) .. controls (5.200,1.060) and (5.200,0.000) .. (6.400,0.000);
\draw (-1.000,-1.060) -- (4.000,-1.060) .. controls (5.200,-1.060) and (5.200,0.000) .. (6.400,0.000);
\draw (8.800,0.000) .. controls (10.000,0.000) and (10.000,1.060) .. (11.200,1.060) -- (16.000,1.060);
\draw (8.800,0.000) .. controls (10.000,0.000) and (10.000,-1.060) .. (11.200,-1.060) -- (16.000,-1.060);
\draw (4.400,0.000) .. controls (5.600,0.000) and (5.600,1.060) .. (6.800,1.060) -- (8.400,1.060) .. controls (9.600,1.060) and (9.600,0.000) .. (10.800,0.000);
\draw (4.400,0.000) .. controls (5.600,0.000) and (5.600,-1.060) .. (6.800,-1.060) -- (8.400,-1.060) .. controls (9.600,-1.060) and (9.600,0.000) .. (10.800,0.000);
\node[font=\tiny,inner sep=0pt] at (1.500,0.000) {$0$};
\node[font=\tiny,inner sep=0pt] at (7.600,0.000) {$1$};
\node[font=\tiny,inner sep=0pt] at (13.600,0.000) {$0$};
\end{tikzpicture}\\[2pt]{\small $I_2$}\end{tabular}

\vspace{4mm}

\begin{tabular}{@{}c@{}}\begin{tikzpicture}[x=0.3cm,y=0.5cm,line width=.5pt]
\draw[fill=gray!15] (-1.300,0) ellipse (0.090cm and 0.705cm);\draw[fill=gray!15] (25.100,0) ellipse (0.090cm and 0.705cm);
\draw (-1.000,1.060) -- (4.000,1.060) .. controls (5.200,1.060) and (5.200,0.000) .. (6.400,0.000);
\draw (-1.000,-1.060) -- (4.000,-1.060) .. controls (5.200,-1.060) and (5.200,0.000) .. (6.400,0.000);
\draw (17.600,0.000) .. controls (18.800,0.000) and (18.800,1.060) .. (20.000,1.060) -- (24.800,1.060);
\draw (17.600,0.000) .. controls (18.800,0.000) and (18.800,-1.060) .. (20.000,-1.060) -- (24.800,-1.060);
\draw (4.400,0.000) .. controls (5.600,0.000) and (5.600,1.060) .. (6.800,1.060) -- (8.400,1.060) .. controls (9.600,1.060) and (9.600,0.000) .. (10.800,0.000);
\draw (4.400,0.000) .. controls (5.600,0.000) and (5.600,-1.060) .. (6.800,-1.060) -- (8.400,-1.060) .. controls (9.600,-1.060) and (9.600,0.000) .. (10.800,0.000);
\draw (8.800,0.000) .. controls (10.000,0.000) and (10.000,1.060) .. (11.200,1.060) -- (12.800,1.060) .. controls (14.000,1.060) and (14.000,0.000) .. (15.200,0.000);
\draw (8.800,0.000) .. controls (10.000,0.000) and (10.000,-1.060) .. (11.200,-1.060) -- (12.800,-1.060) .. controls (14.000,-1.060) and (14.000,0.000) .. (15.200,0.000);
\draw (13.200,0.000) .. controls (14.400,0.000) and (14.400,1.060) .. (15.600,1.060) -- (17.200,1.060) .. controls (18.400,1.060) and (18.400,0.000) .. (19.600,0.000);
\draw (13.200,0.000) .. controls (14.400,0.000) and (14.400,-1.060) .. (15.600,-1.060) -- (17.200,-1.060) .. controls (18.400,-1.060) and (18.400,0.000) .. (19.600,0.000);
\node[font=\tiny,inner sep=0pt] at (1.500,0.000) {$0$};
\node[font=\tiny,inner sep=0pt] at (7.600,0.000) {$1$};
\node[font=\tiny,inner sep=0pt] at (12.000,0.000) {$2$};
\node[font=\tiny,inner sep=0pt] at (16.400,0.000) {$3$};
\node[font=\tiny,inner sep=0pt] at (22.400,0.000) {$0$};
\end{tikzpicture}\\[2pt]{\small $I_4$}\end{tabular}
\caption{Legendrian surgery diagrams for the cycle plumbings of type $I_n$, drawn for $n=1,2,4$.}
\label{fig:surgery-In}
\end{figure}

\begin{figure}[htbp]
\centering
\begin{tabular}{@{}c@{}}\begin{tikzpicture}[x=0.3cm,y=0.5cm,line width=.5pt]
\draw (0.000,0.000) .. controls (1.200,0.000) and (1.200,1.300) .. (2.400,1.300) -- (4.000,1.300) .. controls (5.200,1.300) and (5.200,0.000) .. (6.400,0.000);
\draw (0.000,0.000) .. controls (1.200,0.000) and (1.200,-1.300) .. (2.400,-1.300) -- (4.000,-1.300) .. controls (5.200,-1.300) and (5.200,0.000) .. (6.400,0.000);
\draw (4.400,0.000) .. controls (5.600,0.000) and (5.600,1.300) .. (6.800,1.300) -- (8.400,1.300) .. controls (9.600,1.300) and (9.600,0.000) .. (10.800,0.000);
\draw (4.400,0.000) .. controls (5.600,0.000) and (5.600,-1.300) .. (6.800,-1.300) -- (8.400,-1.300) .. controls (9.600,-1.300) and (9.600,0.000) .. (10.800,0.000);
\draw (8.800,0.000) .. controls (10.000,0.000) and (10.000,0.500) .. (11.200,0.500) -- (12.800,0.500) .. controls (14.000,0.500) and (14.000,0.000) .. (15.200,0.000);
\draw (8.800,0.000) .. controls (10.000,0.000) and (10.000,-0.500) .. (11.200,-0.500) -- (12.800,-0.500) .. controls (14.000,-0.500) and (14.000,0.000) .. (15.200,0.000);
\draw (8.000,0.000) .. controls (9.200,0.000) and (9.200,0.900) .. (10.400,0.900) -- (16.200,0.900) .. controls (17.400,0.900) and (17.400,0.000) .. (18.600,0.000);
\draw (8.000,0.000) .. controls (9.200,0.000) and (9.200,-0.900) .. (10.400,-0.900) -- (16.200,-0.900) .. controls (17.400,-0.900) and (17.400,0.000) .. (18.600,0.000);
\draw (-7.200,0.000) .. controls (-6.000,0.000) and (-6.000,0.900) .. (-4.800,0.900) -- (0.400,0.900) .. controls (1.600,0.900) and (1.600,0.000) .. (2.800,0.000);
\draw (-7.200,0.000) .. controls (-6.000,0.000) and (-6.000,-0.900) .. (-4.800,-0.900) -- (0.400,-0.900) .. controls (1.600,-0.900) and (1.600,0.000) .. (2.800,0.000);
\draw (-4.400,0.000) .. controls (-3.200,0.000) and (-3.200,0.500) .. (-2.000,0.500) -- (-0.400,0.500) .. controls (0.800,0.500) and (0.800,0.000) .. (2.000,0.000);
\draw (-4.400,0.000) .. controls (-3.200,0.000) and (-3.200,-0.500) .. (-2.000,-0.500) -- (-0.400,-0.500) .. controls (0.800,-0.500) and (0.800,0.000) .. (2.000,0.000);
\node[font=\tiny,inner sep=0pt] at (3.600,0.000) {$c_0$};
\node[font=\tiny,inner sep=0pt] at (7.200,0.000) {$c_1$};
\node[font=\tiny,inner sep=0pt] at (13.000,0.000) {$\ell_4$};
\node[font=\tiny,inner sep=0pt] at (16.900,0.000) {$\ell_3$};
\node[font=\tiny,inner sep=0pt] at (-2.200,0.000) {$\ell_2$};
\node[font=\tiny,inner sep=0pt] at (-5.800,0.000) {$\ell_1$};
\end{tikzpicture}\\[2pt]{\small $I_1^*$}\end{tabular}

\vspace{4mm}

\begin{tabular}{@{}c@{}}\begin{tikzpicture}[x=0.3cm,y=0.5cm,line width=.5pt]
\draw (0.000,0.000) .. controls (1.200,0.000) and (1.200,1.300) .. (2.400,1.300) -- (4.000,1.300) .. controls (5.200,1.300) and (5.200,0.000) .. (6.400,0.000);
\draw (0.000,0.000) .. controls (1.200,0.000) and (1.200,-1.300) .. (2.400,-1.300) -- (4.000,-1.300) .. controls (5.200,-1.300) and (5.200,0.000) .. (6.400,0.000);
\draw (4.400,0.000) .. controls (5.600,0.000) and (5.600,1.300) .. (6.800,1.300) -- (8.400,1.300) .. controls (9.600,1.300) and (9.600,0.000) .. (10.800,0.000);
\draw (4.400,0.000) .. controls (5.600,0.000) and (5.600,-1.300) .. (6.800,-1.300) -- (8.400,-1.300) .. controls (9.600,-1.300) and (9.600,0.000) .. (10.800,0.000);
\draw (8.800,0.000) .. controls (10.000,0.000) and (10.000,1.300) .. (11.200,1.300) -- (12.800,1.300) .. controls (14.000,1.300) and (14.000,0.000) .. (15.200,0.000);
\draw (8.800,0.000) .. controls (10.000,0.000) and (10.000,-1.300) .. (11.200,-1.300) -- (12.800,-1.300) .. controls (14.000,-1.300) and (14.000,0.000) .. (15.200,0.000);
\draw (13.200,0.000) .. controls (14.400,0.000) and (14.400,0.500) .. (15.600,0.500) -- (17.200,0.500) .. controls (18.400,0.500) and (18.400,0.000) .. (19.600,0.000);
\draw (13.200,0.000) .. controls (14.400,0.000) and (14.400,-0.500) .. (15.600,-0.500) -- (17.200,-0.500) .. controls (18.400,-0.500) and (18.400,0.000) .. (19.600,0.000);
\draw (12.400,0.000) .. controls (13.600,0.000) and (13.600,0.900) .. (14.800,0.900) -- (20.600,0.900) .. controls (21.800,0.900) and (21.800,0.000) .. (23.000,0.000);
\draw (12.400,0.000) .. controls (13.600,0.000) and (13.600,-0.900) .. (14.800,-0.900) -- (20.600,-0.900) .. controls (21.800,-0.900) and (21.800,0.000) .. (23.000,0.000);
\draw (-7.200,0.000) .. controls (-6.000,0.000) and (-6.000,0.900) .. (-4.800,0.900) -- (0.400,0.900) .. controls (1.600,0.900) and (1.600,0.000) .. (2.800,0.000);
\draw (-7.200,0.000) .. controls (-6.000,0.000) and (-6.000,-0.900) .. (-4.800,-0.900) -- (0.400,-0.900) .. controls (1.600,-0.900) and (1.600,0.000) .. (2.800,0.000);
\draw (-4.400,0.000) .. controls (-3.200,0.000) and (-3.200,0.500) .. (-2.000,0.500) -- (-0.400,0.500) .. controls (0.800,0.500) and (0.800,0.000) .. (2.000,0.000);
\draw (-4.400,0.000) .. controls (-3.200,0.000) and (-3.200,-0.500) .. (-2.000,-0.500) -- (-0.400,-0.500) .. controls (0.800,-0.500) and (0.800,0.000) .. (2.000,0.000);
\node[font=\tiny,inner sep=0pt] at (3.600,0.000) {$c_0$};
\node[font=\tiny,inner sep=0pt] at (7.600,0.000) {$c_1$};
\node[font=\tiny,inner sep=0pt] at (11.600,0.000) {$c_2$};
\node[font=\tiny,inner sep=0pt] at (17.400,0.000) {$\ell_4$};
\node[font=\tiny,inner sep=0pt] at (21.300,0.000) {$\ell_3$};
\node[font=\tiny,inner sep=0pt] at (-2.200,0.000) {$\ell_2$};
\node[font=\tiny,inner sep=0pt] at (-5.800,0.000) {$\ell_1$};
\end{tikzpicture}\\[2pt]{\small $I_2^*$}\end{tabular}
\caption{Legendrian surgery diagrams for the plumbings of type $I_n^*$ with $n>0$, drawn for $n=1,2$, with the labels of \Cref{fig:affine-d-quivers}.}
\label{fig:surgery-Dn}
\end{figure}
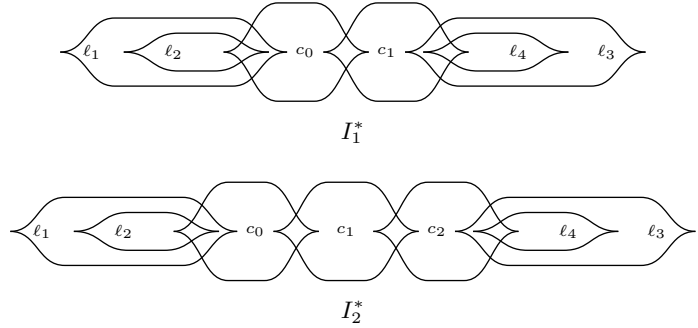

We construct the Liouville field using weighted radial models near the
singular points and cotangent models along the smooth parts of the
components. We then join these models while preserving
$Z\log|\pi|>0$. This ensures that the resulting field points outward
along every sufficiently small boundary $\{|\pi|=\epsilon\}$.

\begin{lemma}\label{lem:local-normal-forms}
Let $p\in\Sigma$. There are holomorphic coordinates $(x,y)$ centred at
$p$ in which $\Omega=dx\wedge dy$ and $F_{\mathrm{red}}$ is the zero set
of
\[
 g_p=
 \begin{cases}
 xy&\text{at a node},\\
 y^2-x^3&\text{for }II,\\
 y(y-x^2)&\text{for }III,\\
 xy(x-y)&\text{for }IV.
 \end{cases}
\]
Similarly, for each $i$ there are holomorphic coordinates $(u,v)$
centred at $x_i$ with $\Omega=du\wedge dv$, $C_i=\{v=0\}$ and
$D_i=\{u=0\}$.
\end{lemma}

\begin{proof}
Choose holomorphic coordinates giving the stated curve equations and
write $\Omega=f(X,Y)\,dX\wedge dY$. At a node and in types $II$,
$III$ and $IV$, respectively, put
\begin{equation}\label{eq:local-weights}
 (a_p,b_p)=(\tfrac12,\tfrac12),\ (\tfrac25,\tfrac35),\
 (\tfrac13,\tfrac23),\ (\tfrac12,\tfrac12).
\end{equation}
For $(a,b)=(a_p,b_p)$, set
\[
 h(X,Y)=\int_0^1 f(t^a X,t^b Y)\,dt,\qquad
 x=Xh^a,\quad y=Yh^b.
\]
Here $h(0)=f(0)\ne0$, so the powers are defined using a local
holomorphic logarithm. Since $(1+aX\partial_X+bY\partial_Y)h=f$, we have
$dx\wedge dy=f\,dX\wedge dY=\Omega$, while
$g_p(x,y)=h^d g_p(X,Y)$, where $d$ is the weighted degree of $g_p$.
Thus $(x,y)$ are the required coordinates. At $x_i$, apply the same
construction with $a=b=\tfrac12$ after straightening the transverse
curves $C_i$ and $D_i$ to the coordinate axes; the rescaling preserves
each axis.
\end{proof}

\FloatBarrier

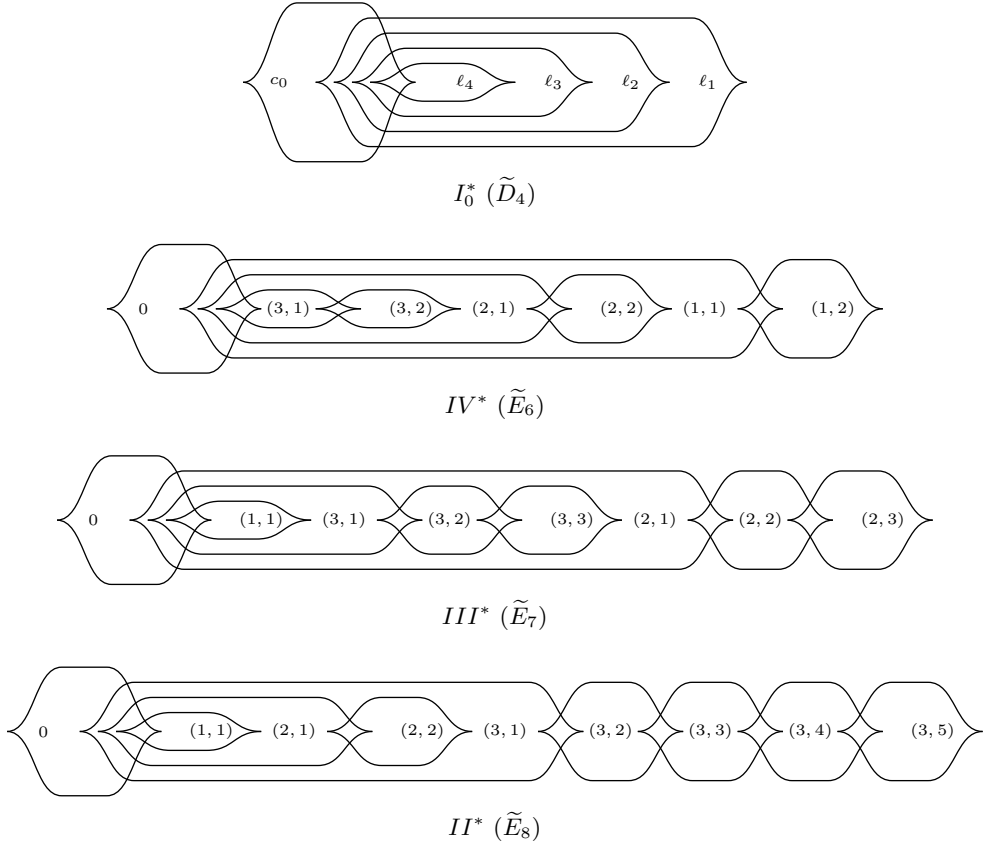
\begin{figure}[htbp]
\centering
\begin{tabular}{@{}c@{}}\begin{tikzpicture}[x=0.3cm,y=0.5cm,line width=.5pt]
\draw (0.000,0.000) .. controls (1.200,0.000) and (1.200,2.100) .. (2.400,2.100) -- (5.200,2.100) .. controls (6.400,2.100) and (6.400,0.000) .. (7.600,0.000);
\draw (0.000,0.000) .. controls (1.200,0.000) and (1.200,-2.100) .. (2.400,-2.100) -- (5.200,-2.100) .. controls (6.400,-2.100) and (6.400,0.000) .. (7.600,0.000);
\draw (5.600,0.000) .. controls (6.800,0.000) and (6.800,0.500) .. (8.000,0.500) -- (9.600,0.500) .. controls (10.800,0.500) and (10.800,0.000) .. (12.000,0.000);
\draw (5.600,0.000) .. controls (6.800,0.000) and (6.800,-0.500) .. (8.000,-0.500) -- (9.600,-0.500) .. controls (10.800,-0.500) and (10.800,0.000) .. (12.000,0.000);
\draw (4.800,0.000) .. controls (6.000,0.000) and (6.000,0.900) .. (7.200,0.900) -- (13.000,0.900) .. controls (14.200,0.900) and (14.200,0.000) .. (15.400,0.000);
\draw (4.800,0.000) .. controls (6.000,0.000) and (6.000,-0.900) .. (7.200,-0.900) -- (13.000,-0.900) .. controls (14.200,-0.900) and (14.200,0.000) .. (15.400,0.000);
\draw (4.000,0.000) .. controls (5.200,0.000) and (5.200,1.300) .. (6.400,1.300) -- (16.400,1.300) .. controls (17.600,1.300) and (17.600,0.000) .. (18.800,0.000);
\draw (4.000,0.000) .. controls (5.200,0.000) and (5.200,-1.300) .. (6.400,-1.300) -- (16.400,-1.300) .. controls (17.600,-1.300) and (17.600,0.000) .. (18.800,0.000);
\draw (3.200,0.000) .. controls (4.400,0.000) and (4.400,1.700) .. (5.600,1.700) -- (19.800,1.700) .. controls (21.000,1.700) and (21.000,0.000) .. (22.200,0.000);
\draw (3.200,0.000) .. controls (4.400,0.000) and (4.400,-1.700) .. (5.600,-1.700) -- (19.800,-1.700) .. controls (21.000,-1.700) and (21.000,0.000) .. (22.200,0.000);
\node[font=\tiny,inner sep=0pt] at (1.600,0.000) {$c_0$};
\node[font=\tiny,inner sep=0pt] at (9.800,0.000) {$\ell_4$};
\node[font=\tiny,inner sep=0pt] at (13.700,0.000) {$\ell_3$};
\node[font=\tiny,inner sep=0pt] at (17.100,0.000) {$\ell_2$};
\node[font=\tiny,inner sep=0pt] at (20.500,0.000) {$\ell_1$};
\end{tikzpicture}\\[2pt]{\small $I_0^*$ ($\widetilde D_4$)}\end{tabular}

\vspace{4mm}

\begin{tabular}{@{}c@{}}\begin{tikzpicture}[x=0.3cm,y=0.5cm,line width=.5pt]
\draw (0.000,0.000) .. controls (1.200,0.000) and (1.200,1.700) .. (2.400,1.700) -- (4.400,1.700) .. controls (5.600,1.700) and (5.600,0.000) .. (6.800,0.000);
\draw (0.000,0.000) .. controls (1.200,0.000) and (1.200,-1.700) .. (2.400,-1.700) -- (4.400,-1.700) .. controls (5.600,-1.700) and (5.600,0.000) .. (6.800,0.000);
\draw (4.800,0.000) .. controls (6.000,0.000) and (6.000,0.500) .. (7.200,0.500) -- (8.800,0.500) .. controls (10.000,0.500) and (10.000,0.000) .. (11.200,0.000);
\draw (4.800,0.000) .. controls (6.000,0.000) and (6.000,-0.500) .. (7.200,-0.500) -- (8.800,-0.500) .. controls (10.000,-0.500) and (10.000,0.000) .. (11.200,0.000);
\draw (9.200,0.000) .. controls (10.400,0.000) and (10.400,0.500) .. (11.600,0.500) -- (13.200,0.500) .. controls (14.400,0.500) and (14.400,0.000) .. (15.600,0.000);
\draw (9.200,0.000) .. controls (10.400,0.000) and (10.400,-0.500) .. (11.600,-0.500) -- (13.200,-0.500) .. controls (14.400,-0.500) and (14.400,0.000) .. (15.600,0.000);
\draw (4.000,0.000) .. controls (5.200,0.000) and (5.200,0.900) .. (6.400,0.900) -- (18.100,0.900) .. controls (19.300,0.900) and (19.300,0.000) .. (20.500,0.000);
\draw (4.000,0.000) .. controls (5.200,0.000) and (5.200,-0.900) .. (6.400,-0.900) -- (18.100,-0.900) .. controls (19.300,-0.900) and (19.300,0.000) .. (20.500,0.000);
\draw (18.500,0.000) .. controls (19.700,0.000) and (19.700,0.900) .. (20.900,0.900) -- (22.500,0.900) .. controls (23.700,0.900) and (23.700,0.000) .. (24.900,0.000);
\draw (18.500,0.000) .. controls (19.700,0.000) and (19.700,-0.900) .. (20.900,-0.900) -- (22.500,-0.900) .. controls (23.700,-0.900) and (23.700,0.000) .. (24.900,0.000);
\draw (3.200,0.000) .. controls (4.400,0.000) and (4.400,1.300) .. (5.600,1.300) -- (27.400,1.300) .. controls (28.600,1.300) and (28.600,0.000) .. (29.800,0.000);
\draw (3.200,0.000) .. controls (4.400,0.000) and (4.400,-1.300) .. (5.600,-1.300) -- (27.400,-1.300) .. controls (28.600,-1.300) and (28.600,0.000) .. (29.800,0.000);
\draw (27.800,0.000) .. controls (29.000,0.000) and (29.000,1.300) .. (30.200,1.300) -- (31.800,1.300) .. controls (33.000,1.300) and (33.000,0.000) .. (34.200,0.000);
\draw (27.800,0.000) .. controls (29.000,0.000) and (29.000,-1.300) .. (30.200,-1.300) -- (31.800,-1.300) .. controls (33.000,-1.300) and (33.000,0.000) .. (34.200,0.000);
\node[font=\tiny,inner sep=0pt] at (1.600,0.000) {$0$};
\node[font=\tiny,inner sep=0pt] at (8.000,0.000) {$(3,1)$};
\node[font=\tiny,inner sep=0pt] at (13.400,0.000) {$(3,2)$};
\node[font=\tiny,inner sep=0pt] at (17.050,0.000) {$(2,1)$};
\node[font=\tiny,inner sep=0pt] at (22.700,0.000) {$(2,2)$};
\node[font=\tiny,inner sep=0pt] at (26.350,0.000) {$(1,1)$};
\node[font=\tiny,inner sep=0pt] at (32.000,0.000) {$(1,2)$};
\end{tikzpicture}\\[2pt]{\small $IV^*$ ($\widetilde E_6$)}\end{tabular}

\vspace{4mm}

\begin{tabular}{@{}c@{}}\begin{tikzpicture}[x=0.3cm,y=0.5cm,line width=.5pt]
\draw (0.000,0.000) .. controls (1.200,0.000) and (1.200,1.700) .. (2.400,1.700) -- (4.400,1.700) .. controls (5.600,1.700) and (5.600,0.000) .. (6.800,0.000);
\draw (0.000,0.000) .. controls (1.200,0.000) and (1.200,-1.700) .. (2.400,-1.700) -- (4.400,-1.700) .. controls (5.600,-1.700) and (5.600,0.000) .. (6.800,0.000);
\draw (4.800,0.000) .. controls (6.000,0.000) and (6.000,0.500) .. (7.200,0.500) -- (8.800,0.500) .. controls (10.000,0.500) and (10.000,0.000) .. (11.200,0.000);
\draw (4.800,0.000) .. controls (6.000,0.000) and (6.000,-0.500) .. (7.200,-0.500) -- (8.800,-0.500) .. controls (10.000,-0.500) and (10.000,0.000) .. (11.200,0.000);
\draw (4.000,0.000) .. controls (5.200,0.000) and (5.200,0.900) .. (6.400,0.900) -- (13.700,0.900) .. controls (14.900,0.900) and (14.900,0.000) .. (16.100,0.000);
\draw (4.000,0.000) .. controls (5.200,0.000) and (5.200,-0.900) .. (6.400,-0.900) -- (13.700,-0.900) .. controls (14.900,-0.900) and (14.900,0.000) .. (16.100,0.000);
\draw (14.100,0.000) .. controls (15.300,0.000) and (15.300,0.900) .. (16.500,0.900) -- (18.100,0.900) .. controls (19.300,0.900) and (19.300,0.000) .. (20.500,0.000);
\draw (14.100,0.000) .. controls (15.300,0.000) and (15.300,-0.900) .. (16.500,-0.900) -- (18.100,-0.900) .. controls (19.300,-0.900) and (19.300,0.000) .. (20.500,0.000);
\draw (18.500,0.000) .. controls (19.700,0.000) and (19.700,0.900) .. (20.900,0.900) -- (22.500,0.900) .. controls (23.700,0.900) and (23.700,0.000) .. (24.900,0.000);
\draw (18.500,0.000) .. controls (19.700,0.000) and (19.700,-0.900) .. (20.900,-0.900) -- (22.500,-0.900) .. controls (23.700,-0.900) and (23.700,0.000) .. (24.900,0.000);
\draw (3.200,0.000) .. controls (4.400,0.000) and (4.400,1.300) .. (5.600,1.300) -- (27.400,1.300) .. controls (28.600,1.300) and (28.600,0.000) .. (29.800,0.000);
\draw (3.200,0.000) .. controls (4.400,0.000) and (4.400,-1.300) .. (5.600,-1.300) -- (27.400,-1.300) .. controls (28.600,-1.300) and (28.600,0.000) .. (29.800,0.000);
\draw (27.800,0.000) .. controls (29.000,0.000) and (29.000,1.300) .. (30.200,1.300) -- (31.800,1.300) .. controls (33.000,1.300) and (33.000,0.000) .. (34.200,0.000);
\draw (27.800,0.000) .. controls (29.000,0.000) and (29.000,-1.300) .. (30.200,-1.300) -- (31.800,-1.300) .. controls (33.000,-1.300) and (33.000,0.000) .. (34.200,0.000);
\draw (32.200,0.000) .. controls (33.400,0.000) and (33.400,1.300) .. (34.600,1.300) -- (36.200,1.300) .. controls (37.400,1.300) and (37.400,0.000) .. (38.600,0.000);
\draw (32.200,0.000) .. controls (33.400,0.000) and (33.400,-1.300) .. (34.600,-1.300) -- (36.200,-1.300) .. controls (37.400,-1.300) and (37.400,0.000) .. (38.600,0.000);
\node[font=\tiny,inner sep=0pt] at (1.600,0.000) {$0$};
\node[font=\tiny,inner sep=0pt] at (9.000,0.000) {$(1,1)$};
\node[font=\tiny,inner sep=0pt] at (12.650,0.000) {$(3,1)$};
\node[font=\tiny,inner sep=0pt] at (17.300,0.000) {$(3,2)$};
\node[font=\tiny,inner sep=0pt] at (22.700,0.000) {$(3,3)$};
\node[font=\tiny,inner sep=0pt] at (26.350,0.000) {$(2,1)$};
\node[font=\tiny,inner sep=0pt] at (31.000,0.000) {$(2,2)$};
\node[font=\tiny,inner sep=0pt] at (36.400,0.000) {$(2,3)$};
\end{tikzpicture}\\[2pt]{\small $III^*$ ($\widetilde E_7$)}\end{tabular}

\vspace{4mm}

\begin{tabular}{@{}c@{}}\begin{tikzpicture}[x=0.3cm,y=0.5cm,line width=.5pt]
\draw (0.000,0.000) .. controls (1.200,0.000) and (1.200,1.700) .. (2.400,1.700) -- (4.400,1.700) .. controls (5.600,1.700) and (5.600,0.000) .. (6.800,0.000);
\draw (0.000,0.000) .. controls (1.200,0.000) and (1.200,-1.700) .. (2.400,-1.700) -- (4.400,-1.700) .. controls (5.600,-1.700) and (5.600,0.000) .. (6.800,0.000);
\draw (4.800,0.000) .. controls (6.000,0.000) and (6.000,0.500) .. (7.200,0.500) -- (8.800,0.500) .. controls (10.000,0.500) and (10.000,0.000) .. (11.200,0.000);
\draw (4.800,0.000) .. controls (6.000,0.000) and (6.000,-0.500) .. (7.200,-0.500) -- (8.800,-0.500) .. controls (10.000,-0.500) and (10.000,0.000) .. (11.200,0.000);
\draw (4.000,0.000) .. controls (5.200,0.000) and (5.200,0.900) .. (6.400,0.900) -- (13.700,0.900) .. controls (14.900,0.900) and (14.900,0.000) .. (16.100,0.000);
\draw (4.000,0.000) .. controls (5.200,0.000) and (5.200,-0.900) .. (6.400,-0.900) -- (13.700,-0.900) .. controls (14.900,-0.900) and (14.900,0.000) .. (16.100,0.000);
\draw (14.100,0.000) .. controls (15.300,0.000) and (15.300,0.900) .. (16.500,0.900) -- (18.100,0.900) .. controls (19.300,0.900) and (19.300,0.000) .. (20.500,0.000);
\draw (14.100,0.000) .. controls (15.300,0.000) and (15.300,-0.900) .. (16.500,-0.900) -- (18.100,-0.900) .. controls (19.300,-0.900) and (19.300,0.000) .. (20.500,0.000);
\draw (3.200,0.000) .. controls (4.400,0.000) and (4.400,1.300) .. (5.600,1.300) -- (23.000,1.300) .. controls (24.200,1.300) and (24.200,0.000) .. (25.400,0.000);
\draw (3.200,0.000) .. controls (4.400,0.000) and (4.400,-1.300) .. (5.600,-1.300) -- (23.000,-1.300) .. controls (24.200,-1.300) and (24.200,0.000) .. (25.400,0.000);
\draw (23.400,0.000) .. controls (24.600,0.000) and (24.600,1.300) .. (25.800,1.300) -- (27.400,1.300) .. controls (28.600,1.300) and (28.600,0.000) .. (29.800,0.000);
\draw (23.400,0.000) .. controls (24.600,0.000) and (24.600,-1.300) .. (25.800,-1.300) -- (27.400,-1.300) .. controls (28.600,-1.300) and (28.600,0.000) .. (29.800,0.000);
\draw (27.800,0.000) .. controls (29.000,0.000) and (29.000,1.300) .. (30.200,1.300) -- (31.800,1.300) .. controls (33.000,1.300) and (33.000,0.000) .. (34.200,0.000);
\draw (27.800,0.000) .. controls (29.000,0.000) and (29.000,-1.300) .. (30.200,-1.300) -- (31.800,-1.300) .. controls (33.000,-1.300) and (33.000,0.000) .. (34.200,0.000);
\draw (32.200,0.000) .. controls (33.400,0.000) and (33.400,1.300) .. (34.600,1.300) -- (36.200,1.300) .. controls (37.400,1.300) and (37.400,0.000) .. (38.600,0.000);
\draw (32.200,0.000) .. controls (33.400,0.000) and (33.400,-1.300) .. (34.600,-1.300) -- (36.200,-1.300) .. controls (37.400,-1.300) and (37.400,0.000) .. (38.600,0.000);
\draw (36.600,0.000) .. controls (37.800,0.000) and (37.800,1.300) .. (39.000,1.300) -- (40.600,1.300) .. controls (41.800,1.300) and (41.800,0.000) .. (43.000,0.000);
\draw (36.600,0.000) .. controls (37.800,0.000) and (37.800,-1.300) .. (39.000,-1.300) -- (40.600,-1.300) .. controls (41.800,-1.300) and (41.800,0.000) .. (43.000,0.000);
\node[font=\tiny,inner sep=0pt] at (1.600,0.000) {$0$};
\node[font=\tiny,inner sep=0pt] at (9.000,0.000) {$(1,1)$};
\node[font=\tiny,inner sep=0pt] at (12.650,0.000) {$(2,1)$};
\node[font=\tiny,inner sep=0pt] at (18.300,0.000) {$(2,2)$};
\node[font=\tiny,inner sep=0pt] at (21.950,0.000) {$(3,1)$};
\node[font=\tiny,inner sep=0pt] at (26.600,0.000) {$(3,2)$};
\node[font=\tiny,inner sep=0pt] at (31.000,0.000) {$(3,3)$};
\node[font=\tiny,inner sep=0pt] at (35.400,0.000) {$(3,4)$};
\node[font=\tiny,inner sep=0pt] at (40.800,0.000) {$(3,5)$};
\end{tikzpicture}\\[2pt]{\small $II^*$ ($\widetilde E_8$)}\end{tabular}
\caption{Legendrian surgery diagrams for the star-shaped plumbings of types $I_0^*$, $IV^*$, $III^*$ and $II^*$.}
\label{fig:surgery-stars}
\end{figure}

\begin{figure}[ht]
\centering
\begin{tikzpicture}[x=0.62cm,y=1cm,line width=.6pt]
\draw[black] (0,0.550) .. controls (0.5,0.550) and (0.5,1.100) .. (1,1.100);
\draw[black] (0,1.100) .. controls (0.5,1.100) and (0.5,0.550) .. (1,0.550);
\draw[black] (1,0.550) .. controls (1.5,0.550) and (1.5,1.100) .. (2,1.100);
\draw[black] (1,1.100) .. controls (1.5,1.100) and (1.5,0.550) .. (2,0.550);
\draw[black] (2,0.550) .. controls (2.5,0.550) and (2.5,1.100) .. (3,1.100);
\draw[black] (2,1.100) .. controls (2.5,1.100) and (2.5,0.550) .. (3,0.550);
\draw[black] (0,1.650) -- (3,1.650);
\draw[black] (0,2.200) -- (3,2.200);
\draw[black] (3,0.550) .. controls (4.265,0.550) and (4.610,1.375) .. (5.300,1.375) .. controls (4.610,1.375) and (4.265,2.200) .. (3,2.200);
\draw[black] (0,0.550) .. controls (-1.265,0.550) and (-1.610,1.375) .. (-2.300,1.375) .. controls (-1.610,1.375) and (-1.265,2.200) .. (0,2.200);
\draw[black] (3,1.100) .. controls (3.770,1.100) and (3.980,1.375) .. (4.400,1.375) .. controls (3.980,1.375) and (3.770,1.650) .. (3,1.650);
\draw[black] (0,1.100) .. controls (-0.770,1.100) and (-0.980,1.375) .. (-1.400,1.375) .. controls (-0.980,1.375) and (-0.770,1.650) .. (0,1.650);
\node[font=\small] at (1.5,2.650) {$II$: $\sigma_1^3$};
\end{tikzpicture}
\hspace{8mm}
\begin{tikzpicture}[x=0.62cm,y=1cm,line width=.6pt]
\draw[red!75!black] (0,0.550) .. controls (0.5,0.550) and (0.5,1.100) .. (1,1.100);
\draw[blue!75!black] (0,1.100) .. controls (0.5,1.100) and (0.5,0.550) .. (1,0.550);
\draw[blue!75!black] (1,0.550) .. controls (1.5,0.550) and (1.5,1.100) .. (2,1.100);
\draw[red!75!black] (1,1.100) .. controls (1.5,1.100) and (1.5,0.550) .. (2,0.550);
\draw[red!75!black] (2,0.550) .. controls (2.5,0.550) and (2.5,1.100) .. (3,1.100);
\draw[blue!75!black] (2,1.100) .. controls (2.5,1.100) and (2.5,0.550) .. (3,0.550);
\draw[blue!75!black] (3,0.550) .. controls (3.5,0.550) and (3.5,1.100) .. (4,1.100);
\draw[red!75!black] (3,1.100) .. controls (3.5,1.100) and (3.5,0.550) .. (4,0.550);
\draw[blue!75!black] (0,1.650) -- (4,1.650);
\draw[red!75!black] (0,2.200) -- (4,2.200);
\draw[red!75!black] (4,0.550) .. controls (5.265,0.550) and (5.610,1.375) .. (6.300,1.375) .. controls (5.610,1.375) and (5.265,2.200) .. (4,2.200);
\draw[red!75!black] (0,0.550) .. controls (-1.265,0.550) and (-1.610,1.375) .. (-2.300,1.375) .. controls (-1.610,1.375) and (-1.265,2.200) .. (0,2.200);
\draw[blue!75!black] (4,1.100) .. controls (4.770,1.100) and (4.980,1.375) .. (5.400,1.375) .. controls (4.980,1.375) and (4.770,1.650) .. (4,1.650);
\draw[blue!75!black] (0,1.100) .. controls (-0.770,1.100) and (-0.980,1.375) .. (-1.400,1.375) .. controls (-0.980,1.375) and (-0.770,1.650) .. (0,1.650);
\node[font=\small] at (2.0,2.650) {$III$: $\sigma_1^4$};
\end{tikzpicture}

\vspace{5mm}

\begin{tikzpicture}[x=0.62cm,y=1cm,line width=.6pt]
\draw[black] (0,0.550) .. controls (0.5,0.550) and (0.5,1.100) .. (1,1.100);
\draw[red!75!black] (0,1.100) .. controls (0.5,1.100) and (0.5,0.550) .. (1,0.550);
\draw[blue!75!black] (0,1.650) -- (1,1.650);
\draw[red!75!black] (1,0.550) -- (2,0.550);
\draw[black] (1,1.100) .. controls (1.5,1.100) and (1.5,1.650) .. (2,1.650);
\draw[blue!75!black] (1,1.650) .. controls (1.5,1.650) and (1.5,1.100) .. (2,1.100);
\draw[red!75!black] (2,0.550) .. controls (2.5,0.550) and (2.5,1.100) .. (3,1.100);
\draw[blue!75!black] (2,1.100) .. controls (2.5,1.100) and (2.5,0.550) .. (3,0.550);
\draw[black] (2,1.650) -- (3,1.650);
\draw[blue!75!black] (3,0.550) -- (4,0.550);
\draw[red!75!black] (3,1.100) .. controls (3.5,1.100) and (3.5,1.650) .. (4,1.650);
\draw[black] (3,1.650) .. controls (3.5,1.650) and (3.5,1.100) .. (4,1.100);
\draw[blue!75!black] (4,0.550) .. controls (4.5,0.550) and (4.5,1.100) .. (5,1.100);
\draw[black] (4,1.100) .. controls (4.5,1.100) and (4.5,0.550) .. (5,0.550);
\draw[red!75!black] (4,1.650) -- (5,1.650);
\draw[black] (5,0.550) -- (6,0.550);
\draw[blue!75!black] (5,1.100) .. controls (5.5,1.100) and (5.5,1.650) .. (6,1.650);
\draw[red!75!black] (5,1.650) .. controls (5.5,1.650) and (5.5,1.100) .. (6,1.100);
\draw[blue!75!black] (0,2.200) -- (6,2.200);
\draw[red!75!black] (0,2.750) -- (6,2.750);
\draw[black] (0,3.300) -- (6,3.300);
\draw[black] (6,0.550) .. controls (7.760,0.550) and (8.240,1.925) .. (9.200,1.925) .. controls (8.240,1.925) and (7.760,3.300) .. (6,3.300);
\draw[black] (0,0.550) .. controls (-1.760,0.550) and (-2.240,1.925) .. (-3.200,1.925) .. controls (-2.240,1.925) and (-1.760,3.300) .. (0,3.300);
\draw[red!75!black] (6,1.100) .. controls (7.265,1.100) and (7.610,1.925) .. (8.300,1.925) .. controls (7.610,1.925) and (7.265,2.750) .. (6,2.750);
\draw[red!75!black] (0,1.100) .. controls (-1.265,1.100) and (-1.610,1.925) .. (-2.300,1.925) .. controls (-1.610,1.925) and (-1.265,2.750) .. (0,2.750);
\draw[blue!75!black] (6,1.650) .. controls (6.770,1.650) and (6.980,1.925) .. (7.400,1.925) .. controls (6.980,1.925) and (6.770,2.200) .. (6,2.200);
\draw[blue!75!black] (0,1.650) .. controls (-0.770,1.650) and (-0.980,1.925) .. (-1.400,1.925) .. controls (-0.980,1.925) and (-0.770,2.200) .. (0,2.200);
\node[font=\small] at (3.0,3.750) {$IV$: $(\sigma_1\sigma_2)^3$};
\end{tikzpicture}
\caption{Front projections of the attaching links $\Lambda_{II}$, $\Lambda_{III}$ and $\Lambda_{IV}$: the rainbow closures of $\sigma_1^3$, $\sigma_1^4$ and $(\sigma_1\sigma_2)^3$.}
\label{fig:rainbow-fronts}
\end{figure}
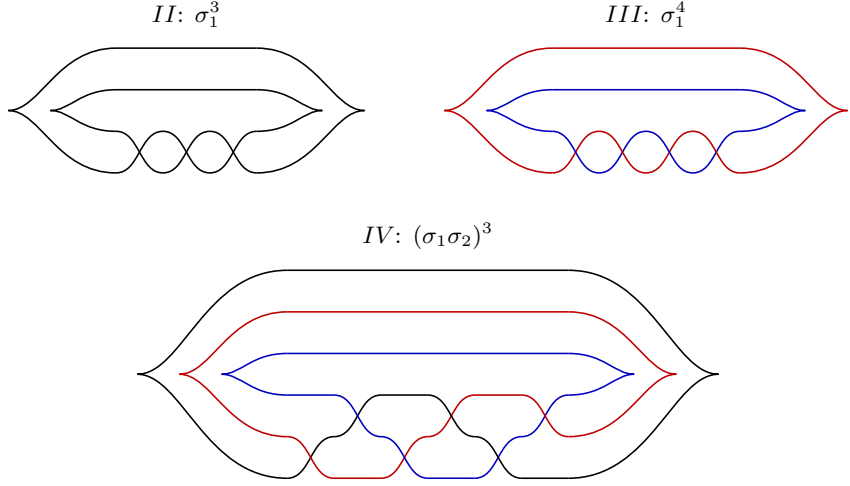

Each branch $B$ is defined by an irreducible factor $g_B$ of $g_p$,
which is quasi-homogeneous of positive degree $d_B$ for the weights
\eqref{eq:local-weights}.

For real numbers $a,b$ with $a+b=1$ put
\[
 \xi_{a,b}=ax\,\partial_x+by\,\partial_y .
\]
Consider the vector field $Z_{a,b}=\xi_{a,b}+\overline{\xi_{a,b}}$
on $\C^2$ generating the flow $(x,y)\mapsto(e^{as}x,e^{bs}y)$. Since
$\iota_{\bar\xi}\Omega=0$,
\begin{equation}\label{eq:linear-liouville}
 \lambda_{a,b}:=\iota_{Z_{a,b}}\omega=\Ree(ax\,dy-by\,dx),
 \qquad d\lambda_{a,b}=\omega .
\end{equation}
In the real coordinates
\begin{equation}\label{eq:real-cotangent-coordinates}
 q=(\Ree x,\Im x),\qquad p=(-\Ree y,\Im y)
\end{equation}
we have $\omega=dp\wedge dq$, and a direct computation gives
\begin{equation}\label{eq:linear-cotangent-form}
 \lambda_{a,b}=p\,dq-d\bigl(a\langle p,q\rangle\bigr),
 \qquad Z_{a,b}=a\,q\,\partial_q+b\,p\,\partial_p .
\end{equation}

This is a special case of the following cotangent model. Let $S$ be a
surface and $V$ a vector field on $S$. In local cotangent coordinates
$(q,p)$, let $E=p\,\partial_p$ be the fibrewise Euler field and $L_V$
the cotangent lift of $V$. With $\omega=dp\wedge dq$, put
\[
 \lambda_V=p\,dq-d\langle p,V(q)\rangle,\qquad
 Z_V=E+L_V.
\]

Writing $V=V^1(q)\,\partial_{q_1}+V^2(q)\,\partial_{q_2}$, we have
\[
 L_V=V^1(q)\frac{\partial}{\partial q_1}
       +V^2(q)\frac{\partial}{\partial q_2}
 -\left(p_1\frac{\partial V^1}{\partial q_1}
         +p_2\frac{\partial V^2}{\partial q_1}\right)
         \frac{\partial}{\partial p_1}
  -\left(p_1\frac{\partial V^1}{\partial q_2}
         +p_2\frac{\partial V^2}{\partial q_2}\right)
         \frac{\partial}{\partial p_2}.
\]
One checks $\iota_E\omega=p\,dq$ and
$\iota_{L_V}\omega=-d\langle p,V\rangle$, so $Z_V$ is the Liouville
field of $\lambda_V$. Along the zero section $Z_V=V$, and $\lambda_V$
vanishes on the zero section and on the cotangent fibres over zeros of
$V$. By \eqref{eq:linear-cotangent-form},
$\lambda_{a,b}=\lambda_V$ for $V=a\,q\,\partial_q$. For a function
$\beta$ on $S$,
\begin{equation}\label{eq:cutoff-liouville}
 Z_{\beta V}=\beta Z_V+(1-\beta)E-\langle p,V\rangle\,
 \nabla\beta\cdot\partial_p ,
\end{equation}
where $\nabla\beta\cdot\partial_p=\sum_j(\partial_{q_j}\beta)\,
\partial_{p_j}$.

We next straighten each branch at a singular point equivariantly. For
$0<\eta\le\frac12$ and $k>0$ write
\[
 T_k(\eta)=\{(u,v)\in\C^2:0<|u|<1,\ |v|<\eta|u|^k\}.
\]

\begin{lemma}\label{lem:branch-coordinates}
Let $B$ be a branch of $\{g_p=0\}$ at $p$, in the coordinates of
\Cref{lem:local-normal-forms}, and let $\Phi_B$, $c_B$ and $k_B$ be
given by the table below. Then $\Phi_B$ is injective on
$T_{k_B}(\eta)$, $\Phi_B^*(dx\wedge dy)=du\wedge dv$, the curve
$u\mapsto\Phi_B(u,0)$ parametrises $B\setminus\{p\}$ near $p$ and is the
normalisation of $B$, and
\[
 (\Phi_B)^*\xi_{a_p,b_p}=c_Bu\,\partial_u+(1-c_B)\,v\,\partial_v .
\]
\[
\begin{array}{c|c|c|c|c}
 \text{type}&B&\Phi_B(u,v)&c_B&k_B\\\hline
 \text{node}&\{y=0\}&(u,v)&\tfrac12&1\\
 \text{node}&\{x=0\}&(-v,u)&\tfrac12&1\\
 II&\{y^2=x^3\}&\bigl(u^2,\ u^3+v/2u\bigr)&\tfrac15&4\\
 III&\{y=0\}&(u,v)&\tfrac13&2\\
 III&\{y=x^2\}&(u,\ u^2+v)&\tfrac13&2\\
 IV&\{y=0\}&(u,v)&\tfrac12&1\\
 IV&\{x=0\}&(-v,u)&\tfrac12&1\\
 IV&\{y=x\}&(u,\ u+v)&\tfrac12&1
\end{array}
\]
Moreover, on $T_{k_B}(\eta)$,
\begin{equation}\label{eq:branch-factor-bounds}
 \bigl|v\,\partial_v\log(g_B\circ\Phi_B)-1\bigr|\le2\eta,
 \qquad
 \bigl|v\,\partial_v\log(g_{B'}\circ\Phi_B)\bigr|\le2\eta
 \quad(B'\ne B),
\end{equation}
and each $g_{B'}\circ\Phi_B$ with $B'\ne B$ is nonvanishing.
The expression $v\,\partial_v\log(g_B\circ\Phi_B)$ in the first
estimate is interpreted by its removable extension at $v=0$.
\end{lemma}

\begin{proof}
The Jacobian determinants are one, and the weight identities follow by
substitution. Injectivity is immediate except for the cusp chart. There,
equality of the images of two distinct points forces
\[
 u'=-u,\qquad v+v'=-4u^4,
\]
contradicting $|v|,|v'|<\eta|u|^4$.

For the estimates, $g_B\circ\Phi_B=\pm v$, except in the cusp case,
where
\[
 (y^2-x^3)\circ\Phi_B=u^2v\left(1+\frac{v}{4u^4}\right).
\]
The remaining factors are, up to sign, $u$ or $u^j\pm v$, with $j=1,2$.
Their nonvanishing and the bounds follow from these expressions and
\[
 \frac{|v|}{|u^j\pm v|}\le\frac{\eta}{1-\eta}\le2\eta.
\]
\end{proof}

In the real coordinates \eqref{eq:real-cotangent-coordinates} attached to
$(u,v)$, \eqref{eq:linear-cotangent-form} and \Cref{lem:branch-coordinates}
give
\begin{equation}\label{eq:branch-cotangent-model}
 \Phi_B^*\omega=dp\wedge dq,\qquad
 \Phi_B^*\lambda_{a_p,b_p}=\lambda_{V_B},\qquad
 V_B=c_B\,q\,\partial_q .
\end{equation}
Thus near each branch the linear Liouville structure at $p$ is the
cotangent model of a radial vector field on the branch.

\begin{proposition}\label{prop:neighbourhood-weinstein}
For all sufficiently small $\epsilon>0$ there is a Weinstein structure
$(\lambda,\phi)$ on $W_\epsilon$ with $d\lambda=\omega$ such that:
\begin{enumerate}
\item $\lambda$ vanishes on $F_{\mathrm{red}}$ and on
each $D_i\cap W_\epsilon$, and $Z\mu>0$ on $W_\epsilon\setminus F$;
\item the zeros of $Z$ are the points of $\Sigma$, of index zero, the
points $x_i$, of index two, and $e_i-1$ points of index one on each
$C_i\setminus\Sigma$, where $e_i$ is the number of points of
$\widetilde C_i$ lying over $\Sigma$;
\item near $p\in\Sigma$, in the coordinates of
\Cref{lem:local-normal-forms}, $\lambda=\lambda_{a_p,b_p}$ and
$\phi=|x|^2+|y|^2$;
\item the skeleton is $F_{\mathrm{red}}$, the unstable manifold of $x_i$
is $D_i\cap W_\epsilon$, and all stable manifolds are contained in
$F_{\mathrm{red}}$.
\end{enumerate}
\end{proposition}

The count in (2) comes from a vector field on each normalisation
$\widetilde C_i$ with $e_i$ sources at the preimages of $\Sigma$, one
sink at $x_i$, and $e_i-1$ saddles. These counts describe the auxiliary
handle decomposition used in the construction.

\begin{proof}
Near $p\in\Sigma$ use $\lambda_p=\lambda_{a_p,b_p}$, and near
$x_i$ use $\lambda_{x_i}=\lambda_{-1/2,3/2}$ in the coordinates
$(u,v)$ with $C_i=\{v=0\}$ and $D_i=\{u=0\}$. Writing
$\pi=w_p\prod_Bg_B^{m_B}$ with $w_p$ a unit gives
\[
 Z_p\mu=\sum_Bm_Bd_B+\Ree\frac{\xi_{a_p,b_p}w_p}{w_p}
 \ge\kappa_p>0
\]
after shrinking the charts. Similarly, $\pi=v^{m_i}w_i$ near $x_i$
gives $Z_{x_i}\mu\ge m_i$.

Each normalisation $\widetilde C_i$ is a sphere. On it, choose a
smooth function $f_i$ and a vector field $Y_i$, with $Y_if_i>0$
away from the zeros of $Y_i$, sources at the $e_i$ points above
$\Sigma$, a sink at $x_i$, and $e_i-1$ saddles. Use the local models
\[
 (f_i,Y_i)=
 \begin{cases}
 \bigl((|x|^2+|y|^2)\circ\nu_i,\ c_Bq\,\partial_q\bigr),
     &\text{near a source on branch }B,\\
 \bigl(c_i-|q|^2,\ -\tfrac12q\,\partial_q\bigr),
     &\text{near }x_i,\\
 \bigl(c_s-q_1^2+q_2^2,\ -q_1\partial_{q_1}+q_2\partial_{q_2}\bigr),
     &\text{near a saddle }s,
 \end{cases}
\]
with suitable constants $c_i,c_s$.

Choose $R>0$ so that the source and sink models hold on disjoint
disks $\{|q|<R\}$ around the marked points
$\nu_i^{-1}(\Sigma)\cup\{x_i\}$, identifying $x_i$ with its unique
lift to $\widetilde C_i$. Set
\[
 \mathcal A=\{Re^{-L}<|q|<R\},\qquad
 A_i=\widetilde C_i\setminus
 \bigcup_{\text{marked points}}\{|q|<R\}.
\]
Each collar $\mathcal A$ has logarithmic length $L$, while $A_i$
and $Y_i$ remain fixed as $L$ increases.

To ensure $Z\mu>0$ away from the prescribed local models, we
rescale $Y_i$. For $0<\delta<1$, choose smooth functions
$\beta_i:\widetilde C_i\to[\delta,1]$ and put
\[
 V_i=\beta_iY_i,\qquad
 \beta_i=
 \begin{cases}
 1&\text{near the marked points},\\
 \delta&\text{near }A_i,
 \end{cases}
 \qquad |q|\,|\nabla\beta_i|\le C/L\quad\text{on }\mathcal A.
\]
This preserves $V_if_i>0$ away from the zeros. Place the overlaps
with the marked-point neighbourhoods where $\beta_i=1$ for every $\delta$.

The relative Lagrangian neighbourhood theorem gives cotangent charts
$\Theta_i$
along the normalisations with small disks about the marked points
removed, agreeing with the branch charts
of \Cref{lem:branch-coordinates} and the charts at $x_i$ on their
collars. Use $\Theta_{i*}\lambda_{V_i}$ on the tubes and
$\lambda_p,\lambda_{x_i}$ near the marked points. These forms agree
on every overlap by the choice of $\beta_i$ and equations
\eqref{eq:branch-cotangent-model} and \eqref{eq:linear-cotangent-form},
defining a primitive $\lambda$ of $\omega$ near $F_{\mathrm{red}}$. Its restriction
vanishes on $F_{\mathrm{red}}$ and on the prescribed
transverse disks near $x_i$. Measure the tube width by $\eta$, using
$|v|<\eta|u|^{k_B}$ on the branch collars and $|p|<\eta$ over $A_i$.

We verify $Z\mu>0$ first over $A_i$, where $Z=E+\delta L_{Y_i}$.
A transverse holomorphic coordinate $\tau$ satisfies
$\tau\circ\Theta_i(q,p)=\ell_q(p)+O(|p|^2)$ with $\ell_q$ a real
linear isomorphism. Since $\pi=\tau^{m_i}\cdot(\text{unit})$, we have
$E\mu=m_i+O(|p|)$ and $L_{Y_i}\mu=O(1)$, uniformly over $A_i$.
Thus
\[
 Z\mu=m_i+O(\delta+\eta)>0
\]
for sufficiently small $\delta,\eta$ once the charts are fixed.

On each collar,
\eqref{eq:cutoff-liouville} gives
\[
 Z=\beta_i Z_{Y_i}+(1-\beta_i)E
   -\langle p,Y_i\rangle\nabla\beta_i\cdot\partial_p .
\]
The first two fields satisfy $Z_{Y_i}\mu\ge c>0$ and
$E\mu\ge m_B/2$ for small $\eta$, by
\eqref{eq:branch-factor-bounds}; at $x_i$ use
$\pi=v^{m_i}w_i$. The same bounds give
$|p|\,|\partial_p\mu|\le C$, so the last term contributes at most
$C/L$ in absolute value. Thus, for small $\eta$,
$Z\mu\ge c_0-C/L$, with $c_0,C$ independent of $L,\delta$.
Choose $L$ large, fix the
cotangent charts, and then choose $\delta,\eta$ successively small
so that both lower bounds are positive.

Extend the $f_i$ to a smooth function $f$ near $F_{\mathrm{red}}$,
using $f_i+|p|^2$ near the sinks and saddles and $|x|^2+|y|^2$ near
$\Sigma$. At a saddle the normal eigenvalues of $Z$ are $1\pm\delta>0$,
giving index one; the other models give index zero at $\Sigma$ and
index two at $x_i$.
The local models and the inequality $V_if_i>0$ on the compact
remainder of $F_{\mathrm{red}}$ give $Zf>0$ away from the zeros in
a neighbourhood of $F_{\mathrm{red}}$. Shrink $\epsilon$ so that $W_\epsilon$ lies there.
Near the boundary, set
$\phi=(1-\chi(\mu))f+\chi(\mu)(K+\mu)$, where $\chi$ increases
from $0$ to $1$ and $K+\mu>f$ on the transition region. Then
$Z\phi>0$ away from the zeros, $\phi$ is constant on the boundary,
and its local Morse models are unchanged.

Finally, $F_{\mathrm{red}}$ is compact and invariant, whereas every
forward trajectory off $F$ exits $W_\epsilon$: until it exits, it
lies in a compact annulus where $Z\mu$ has a positive minimum.
Thus the skeleton is $F_{\mathrm{red}}$, which contains every stable
manifold. For small $\epsilon$, $D_i\cap W_\epsilon$ lies in the
linear model at $x_i$ and is its unstable disk.
\end{proof}

\begin{remark}
Since $\mu=\log|\pi|$, the hypersurface
$\pi^{-1}(\{|t|=r\})$ is the level set $\mu^{-1}(\log r)$.
Thus $Z\mu=d\mu(Z)>0$ means that $Z$ is transverse to this
hypersurface and points outward from $W_r=\{|\pi|\le r\}$,
for every sufficiently small $r>0$.
\end{remark}

\begin{proof}[Proof of \Cref{thm:neighbourhood-surgery}]
Let $(\lambda,\phi)$ be as in \Cref{prop:neighbourhood-weinstein}. The
first assertions are parts (1), (2) and (4) of that proposition.

For the normal crossing types, \Cref{prop:neighbourhood-weinstein}
constructs the positive cotangent plumbing along $\Gamma_F$,
including self-plumbing for $I_1$
(see \cite[Proposition~7.3]{SmithThomas} for trees).

Choose disjoint balls $B_p=\{|x|^2+|y|^2\le\rho_p^2\}$ in the
node charts. Let $\Delta_{i,a}\subset\widetilde C_i$, $1\le a\le e_i$,
be the disk components of $\nu_i^{-1}(\bigcup_p B_p)$, and put
$K_{i,a}=\nu_i(\partial\Delta_{i,a})$. Each $\partial B_p$ contains
a standard Legendrian Hopf pair. Choose Morse data joining these
disks by $e_i-1$ bands into a disk $\Delta_i$ with complementary
disk containing $x_i$. Attach the corresponding Weinstein
$1$-handles to the balls.
The band sum $\Lambda_i=\nu_i(\partial\Delta_i)$ of the $K_{i,a}$
is the Legendrian attaching circle of a $2$-handle whose core is
the complementary disk, with its Liouville collar, and whose
cocore is $D_i\cap W_\epsilon$.
Let $G$ have the balls $B_p$ as vertices and the $1$-handles as edges.
Cancelling $0$-$1$ pairs along a spanning tree of $G$
\cite[Proposition~12.22]{CE} leaves the handle counts
$(h_0,h_1,h_2)=(1,b_1(\Gamma_F),k)$.
Bands through cancelled handles give Legendrian connected sums,
preserving the Hopf clasps.
Ordering the bands as in \cite[Section~3]{ELplumbing} gives
$L_{\Gamma_F}$. The $2$-handle charts remain fixed, so their
cocores complete to $\widehat D_i$.

For $II$, $III$ and $IV$, \Cref{prop:neighbourhood-weinstein} gives
one minimum $p$ and the index-two points $x_i$. Since
$\phi=|x|^2+|y|^2$ near $p$, a sufficiently small sublevel is the ball
\[
 B=\{\phi\le r^2\}=\{|x|^2+|y|^2\le r^2\}.
\]
By Weinstein's handle theory \cite{Weinstein91}
(see \cite[Section~11.4 and Proposition~12.12]{CE}), $W_\epsilon$ is obtained
from $(B,\lambda_p)$ by attaching Weinstein $2$-handles along the
Legendrian circles
\[
 \Lambda_i=W^s(x_i)\cap\partial B=C_i\cap\partial B,
\]
with cocores $W^u(x_i)=D_i\cap W_\epsilon$.

The path $\lambda_t=\lambda_{a_t,b_t}$ with
$(a_t,b_t)=(1-t)(a_p,b_p)+t(\frac12,\frac12)$ consists of Liouville
forms whose Liouville fields are transverse to the round sphere
$\partial B$. At $t=1$ it is the standard radial Liouville form of the
linear symplectic form $\omega$, so $(B,\lambda_1)$ is a Darboux ball.
By \cite[Lemma~12.7(i)]{CE}, this deformation extends across a
collar of $\partial B$ with no critical points, fixing the structure
beyond the collar up to positive scaling. Comparing the initial
and final collar holonomies \cite[Lemma~11.4]{CE} gives a contactomorphism
from $(\partial B,\ker\lambda_p)$ to
$(\partial B,\ker\lambda_1)\cong(S^3,\xi_{\mathrm{std}})$ isotopic to the
identity. Thus $W_\epsilon$ is Weinstein homotopic to the Legendrian
surgery on the image $\Lambda_F$ of $\bigsqcup_i\Lambda_i$, with the
same cocores since scaling leaves the Liouville field unchanged.
As a smooth oriented link $\Lambda_F$ is the link of the plane curve
singularity $\{g_p=0\}$, oriented by the complex
orientation of its branches: the positive torus link $T(2,3)$, $T(2,4)$
or $T(3,3)$.

The handle classes are $[C_i]$, so the Weinstein framing convention
\cite{Gompf98} gives
\[
 \operatorname{tb}(\Lambda_i)=C_i^2+1,\qquad
 \operatorname{lk}(\Lambda_i,\Lambda_j)=C_i\cdot C_j\quad(i\ne j).
\]
Since $F\cdot C_i=0$, the component self-intersections are $0$ in type
$II$ and $-2$ in types $III$, $IV$. Thus the links $T(2,3)$, $T(2,4)$
and $T(3,3)$ have component Thurston-Bennequin invariants $1$, $-1$
and $-1$, respectively.
The rainbow closures in \Cref{fig:rainbow-fronts} have these oriented
link types and component Thurston-Bennequin invariants. The Legendrian
classifications \cite[Section~4.1]{EtnyreHonda} and
\cite[Theorem~1.1(1)]{DET} therefore identify them with $\Lambda_F$.
In particular, all component rotation numbers vanish. For $III$ and
$IV$, the ordered classification \cite[Theorem~1.2(1)]{DET} allows any
prescribed correspondence between components.
\end{proof}

\begin{remark}\label{rem:neighbourhood-independence}
Let $\lambda_0,\lambda_1$ be primitives of $\omega$ on $W_\epsilon$,
vanishing on $F_{\mathrm{red}}$ and each $D_i\cap W_\epsilon$, with
Liouville fields pointing outward along $\partial W_\epsilon$.
Their completions are related by an exact symplectomorphism carrying
the completed planes to one another up to exact cylindrical Lagrangian
isotopy. Indeed, $\lambda_1-\lambda_0=dg$ since $W_\epsilon$ retracts
onto $F_{\mathrm{red}}$. The path
$\lambda_t=(1-t)\lambda_0+t\lambda_1$ retains the vanishing and
outward-pointing conditions, so
\cite[Lemma~11.6 and Proposition~11.8]{CE} give exact symplectic
identifications of the completions under which the completed disks
form an exact cylindrical Lagrangian isotopy.
\end{remark}

\section{Chekanov-Eliashberg algebras associated with Kodaira fibres}
\label{sec:algebras}

We compute the Chekanov-Eliashberg dg algebras with bulk parameters
and identify their degree-zero cohomology with $A_F(\boldsymbol q)$.
For the normal crossing types we use the plumbing calculation of
\cite{ELplumbing}; for $II$, $III$, $IV$, the rainbow-closure model of
\cite{CSrainbow}. The comparison with wrapped Floer cohomology and the
vanishing in negative degrees are established in
\Cref{sec:foundations,sec:surgery}.

\subsection{Legendrian dg algebras and basepoint variables}
\label{subsec:ce-conventions}

Let $\Lambda=\bigsqcup_{v}\Lambda_v$ be a Legendrian link in
$\#^k(S^1\times S^2,\xi_{\mathrm{std}})$, presented in Gompf's standard
form, with one basepoint on each component. We write
$\CE^*(\Lambda)$ for its fully noncommutative Chekanov-Eliashberg dg
algebra with cohomological grading, in the form given in
\cite[Section~2]{ELplumbing} following Ekholm and Ng. It is an
algebra over $S=\bigoplus_v\kk e_v$, freely generated over $S$ by the
Reeb chords of $\Lambda$ (the crossings of the resolved diagram), the
internal generators of the $1$-handles, and invertible basepoint
variables $t_v^{\pm1}$ with $e_vt_ve_v=t_v$ and $dt_v=0$. The
differential of a chord counts holomorphic disks, each weighted by the
sign and by the time-ordered product of the basepoint variables
encountered along its boundary. When $k=0$ the internal generators are
absent and $\CE^*(\Lambda)$ is the Chekanov-Eliashberg algebra of a link
in $(S^3,\xi_{\mathrm{std}})$ with basepoints \cite[Section~2.1]{CSrainbow}.
The components of $\Lambda$ carry their bounding spin structures, which
are those of the cocores; relative to the formulas of Ekholm and Ng
this amounts to the substitution $t_v\mapsto-t_v$
\cite[footnote~2]{ELplumbing}, already incorporated below. All
components of the links in \Cref{thm:neighbourhood-surgery} have
rotation number zero, so $\CE^*(\Lambda)$ is $\ZZ$-graded once a Maslov
potential is chosen on each component. We use the gradings of
\cite{ELplumbing,CSrainbow}. The explicit models below have generators
in degrees $0$ and $-1$. Their comparison with the holomorphic gradings
of the planes is made in the proof of \Cref{prop:bulk-cocores}.

Since $t_v$ is closed of degree zero, the two-sided ideal generated by
the elements $t_v-q_ve_v$ is a dg ideal. For
$\boldsymbol q\in(\kk^\times)^r$ we put
\begin{equation}\label{eq:ce-specialisation}
 \CE^*_{\boldsymbol q}(\Lambda)=\CE^*(\Lambda)/(t_v-q_ve_v:v).
\end{equation}
Its differential is obtained by substituting $q_ve_v$ for $t_v$;
the trivial bulk class corresponds to $\boldsymbol q=\boldsymbol1$.
The identification $q_v=\mathfrak b([C_v])$ is justified by the
capping argument in the proof of \Cref{prop:bulk-cocores}.
For plumbings, this bulk interpretation was proposed in
\cite[Remark~2]{ELplumbing}.

A stable tame isomorphism fixing the basepoint variables descends under
$t_v=q_ve_v$: both the changes of generators and the cancelling generator
pairs remain valid after substitution. For cycle plumbings, we also
need the reduction of the internal $1$-handle algebra described in the
proof of \Cref{thm:plumbing-ce}.

\subsection{Normal crossing types: the plumbing dg algebra}
\label{subsec:plumbing-dga}

Let $F$ be of normal crossing type and let $Q$ be the graph
$\Gamma_F$, so that $Q$ is $\widetilde A_{n-1}$, $\widetilde D_{n+4}$,
$\widetilde E_6$, $\widetilde E_7$ or $\widetilde E_8$. The graph for
$I_1$ has one vertex and one loop, and that for $I_2$ has two vertices
and two distinct edges. Orient the edges and order them, and write
$\overline Q$ for the double of $Q$: for each arrow $a$ of $Q$ it has
the reverse arrow $a^*$. For the cycles we use the orientation
$a_i:i\to i+1$ and the spanning tree $T$ obtained by deleting
$a_{n-1}$; for the trees $T=Q$. The conventions for each type are drawn
in \Cref{fig:affine-e-quivers,fig:cyclic-quivers,fig:affine-d-quivers}.
Set
\[
 S=\bigoplus_{i\in Q_0}\kk e_i,\qquad
 \mathscr L_Q=\kk\overline Q
 \big[(1+aa^*)^{-1},(1+a^*a)^{-1}:a\in Q_1\big].
\]
Here $1=\sum_i e_i$, and the inverses are algebraic universal
localisations. Put
\begin{equation}\label{eq:vertex-products}
\begin{gathered}
 A_i=\prod_{t(a)=i}^{\longrightarrow}(e_i+aa^*),\qquad
 B_i=\prod_{s(a)=i}^{\longrightarrow}(e_i+a^*a),\\
 \Lambda^{\boldsymbol q}(Q)=\mathscr L_Q/(A_i-q_iB_i:i\in Q_0).
\end{gathered}
\end{equation}
An empty product is $e_i$. Both products follow the chosen arrow order.
This is the multiplicative preprojective algebra of Crawley-Boevey and
Shaw with parameters $q_i$; up to isomorphism it does not depend on the
orientation or the order
\cite[Theorem~1.4]{CBShaw}. For the affine fibres set
$A_F(\boldsymbol q)=\Lambda^{\boldsymbol q}(Q)$.

The Chekanov-Eliashberg algebra of the link $L_Q$ of
\Cref{thm:neighbourhood-surgery}(1) was computed in
\cite[Theorem~17]{ELplumbing}. In the notation of
\Cref{subsec:ce-conventions}, the result is a quasi-isomorphism
between $\CE^*(L_Q)$ and the dg algebra $\mathcal L_Q$ over
$S\langle t_i^{\pm1}\rangle$ generated by
\[
 c_a,\ c_a^*\ (a\in Q_1),\qquad
 z_a^{\pm1},\ \zeta_a\ (a\in Q_1\setminus T),\qquad
 \tau_i\ (i\in Q_0),
\]
with $c_a=e_{t(a)}c_ae_{s(a)}$, $c_a^*=e_{s(a)}c_a^*e_{t(a)}$,
$z_a=e_{s(a)}z_ae_{s(a)}$, $\zeta_a=e_{s(a)}\zeta_ae_{s(a)}$ and
$\tau_i=e_i\tau_ie_i$, in degrees
$|c_a|=|c_a^*|=|z_a|=0$ and $|\zeta_a|=|\tau_i|=-1$, and with
differential
\begin{equation}\label{eq:el-differential}
\begin{gathered}
 d\zeta_a=e_{s(a)}+c_a^*c_a-z_a,\\
 d\tau_i=\prod_{t(a)=i}^{\longrightarrow}(e_i+c_ac_a^*)
  -t_i\prod_{s(a)=i,\ a\notin T}^{\longrightarrow}z_a
   \prod_{s(a)=i,\ a\in T}^{\longrightarrow}(e_i+c_a^*c_a).
\end{gathered}
\end{equation}
The generators $c_a,c_a^*$ are the two crossings between the components
of $L_Q$ attached to the ends of $a$; $\tau_i$ is the self-crossing of
the $i$th component; and $z_a,\zeta_a$ come from the $1$-handle of the
edge $a\notin T$, which is present only for the cycles.

\begin{definition}\label{def:plumbing-dga}
For $\boldsymbol q\in(\kk^\times)^{Q_0}$, let
$\mathcal B_Q^{\boldsymbol q}$ be the dg algebra obtained from
$\mathcal L_Q$ by setting $t_i=q_ie_i$ in
\Cref{eq:el-differential}.
\end{definition}

\begin{theorem}\label{thm:plumbing-ce}
For any normal crossing type $F$, any field $\kk$, and any
nonzero parameter tuple $\boldsymbol q$, the dg algebra
$\CE^*_{\boldsymbol q}(L_Q)$ is quasi-isomorphic to
$\mathcal B_Q^{\boldsymbol q}$.
\end{theorem}

\begin{proof}
For cycle plumbings, \cite[Theorem~12 and Proposition~14]{ELplumbing}
replace the internal two-strand algebra by its quasi-isomorphic Laurent
model $\kk[z^{\pm1}]$. This reduction does not involve the basepoint
variables. After specialising $t_i=q_ie_i$, the remaining chord
generators form a semifree dg extension of the internal algebra, so the
reduction remains a quasi-isomorphism. The subsequent changes of
generators and destabilisations in
\cite[Lemma~16 and Theorem~17]{ELplumbing} fix the basepoint variables
and give the displayed differential. For trees, no internal-algebra
reduction is needed.
\end{proof}

The dg algebra $\mathcal B_Q^{\boldsymbol q}$ is concentrated in
nonpositive degrees. A word of degree $-1$ contains exactly one of the
generators $\zeta_a,\tau_i$, so its differential is a two-sided multiple
of $d\zeta_a$ or $d\tau_i$. Hence
\[
 H^0(\mathcal B_Q^{\boldsymbol q})
 =S\langle c_a,c_a^*,z_a^{\pm1}\rangle/(d\zeta_a,\ d\tau_i),
 \qquad H^k(\mathcal B_Q^{\boldsymbol q})=0\quad(k>0).
\]

\begin{lemma}\label{lem:factors-invertible}
In $H^0(\mathcal B_Q^{\boldsymbol q})$ the elements
$e_{t(a)}+c_ac_a^*$ and $e_{s(a)}+c_a^*c_a$, $a\in Q_1$, are
invertible in their respective corners. Consequently
\[
 H^0(\mathcal B_Q^{\boldsymbol q})\cong\Lambda^{\boldsymbol q}(Q),
\]
by the map $c_a\mapsto a$, $c_a^*\mapsto a^*$, $z_a\mapsto e_{s(a)}+a^*a$.
\end{lemma}

\begin{proof}
For $a\notin T$, the relation $d\zeta_a=0$ identifies
$e_{s(a)}+c_a^*c_a$ with the invertible element $z_a$. The identity
\[
 (e_{t(a)}+c_ac_a^*)^{-1}
 =e_{t(a)}-c_a(e_{s(a)}+c_a^*c_a)^{-1}c_a^*
\]
gives invertibility at the other endpoint. Now remove the edges of
$T$ successively at leaves. At each leaf, the vertex relation contains
just one factor whose invertibility is not yet known; all other
factors are invertible, so this factor is invertible as well. The
displayed identity, with the arrows exchanged if necessary, gives
invertibility at its other endpoint. This proves the assertion for
every edge. Eliminating the $z_a$ then gives the defining presentation
of $\Lambda^{\boldsymbol q}(Q)$.
\end{proof}

To compare with the dg multiplicative preprojective algebra, adjoin
$z_a^{\pm1},\zeta_a$ also for $a\in T$, with
\[
 d\zeta_a=e_{s(a)}+c_a^*c_a-z_a.
\]
This models derived localisation at the remaining factors. Since
their classes are already invertible in $H^0$, the localisation map
is a quasi-isomorphism \cite[Section~3]{BCL18}. The argument of
\cite[Remark~4.7]{KS} now applies to this model with localisation
variables for every arrow and gives
\[
 \mathcal B_Q^{\boldsymbol q}\simeq
 \Lambda^{\mathrm{dg},\boldsymbol q}(Q).
\]

In the presentations below, the idempotent $e_i$ corresponds to the
component $C_i$, hence to the plane $\widehat D_i$; the parameter $q_i$
is the value of $\mathfrak{b}$ on the class $\gamma_i$ of $C_i$.
The doubled quivers of each type are drawn with their presentations
below.

\subsection{Affine \texorpdfstring{$A$}{A} diagrams}

For $I_n$, orient the cycle by $a_i:i\to i+1$, with indices in
$\ZZ/n$. When $n=2$, the arrow $a_1$ and the formal reverse $a_0^*$
remain distinct. The conventions are drawn in \Cref{fig:cyclic-quivers}.
The relation at vertex $i$ is
\begin{equation}\label{eq:cycle-relations}
 e_i+a_{i-1}a_{i-1}^*=q_i(e_i+a_i^*a_i).
\end{equation}

\begin{figure}[htbp]
\centering
\scalebox{0.7}{%
\begin{tikzpicture}[
  x=1cm,y=1cm,
  cqvertex/.style={circle,draw,fill=black,inner sep=0pt,minimum size=3.5pt,
    outer sep=1pt},
  cqarrow/.style={-{Stealth[length=1.6mm,width=1.2mm]},line width=.45pt,shorten >=2pt,shorten <=2pt},
  cqlabel/.style={font=\small,fill=white,inner sep=1.5pt},
  cqtitle/.style={font=\small}
]
  \node[cqtitle] at (-8,1.4) {$I_1$};
  \node[cqvertex,label={[font=\small,inner sep=1pt,yshift=3pt]above:$0$}] (cqloop) at (-8,0) {};
  \path[cqarrow]
    (cqloop) edge[loop left,min distance=13mm]
      node[cqlabel,left] {$a_0$} (cqloop)
    (cqloop) edge[loop right,min distance=13mm]
      node[cqlabel,right] {$a_0^*$} (cqloop);

  \node[cqtitle] at (-2.5,1.4) {$I_2$};
  \node[cqvertex,label={[font=\small,inner sep=1pt,yshift=3pt]above:$0$}] (cqtwozero) at (-4.5,0) {};
  \node[cqvertex,label={[font=\small,inner sep=1pt,yshift=3pt]above:$1$}] (cqtwoone) at (-.5,0) {};
  \draw[cqarrow] (cqtwozero) to[bend left=35]
    node[cqlabel,above] {$a_0$} (cqtwoone);
  \draw[cqarrow] (cqtwoone) to[bend right=12]
    node[cqlabel,above] {$a_0^*$} (cqtwozero);
  \draw[cqarrow] (cqtwoone) to[bend left=35]
    node[cqlabel,below] {$a_1$} (cqtwozero);
  \draw[cqarrow] (cqtwozero) to[bend right=12]
    node[cqlabel,below] {$a_1^*$} (cqtwoone);

  \node[cqtitle] at (5.2,1.4) {$I_n$, $n\geq3$};
  \node[cqvertex,label={[font=\small,inner sep=1pt,yshift=3pt]above:$0$}] (cqzero) at (1.5,0) {};
  \node[cqvertex,label={[font=\small,inner sep=1pt,yshift=3pt]above:$1$}] (cqone) at (5.2,0) {};
  \node[cqvertex,label={[font=\small,inner sep=1pt,yshift=3pt]above:$n-1$}] (cqlast) at (8.9,0) {};
  \draw[cqarrow] (cqzero) to[bend left=13]
    node[cqlabel,above] {$a_0$} (cqone);
  \draw[cqarrow] (cqone) to[bend left=13]
    node[cqlabel,below] {$a_0^*$} (cqzero);
  \draw[cqarrow,densely dotted] (cqone) to[bend left=13]
    node[cqlabel,above] {$a_1,\ldots,a_{n-2}$} (cqlast);
  \draw[cqarrow,densely dotted] (cqlast) to[bend left=13]
    node[cqlabel,below] {$a_{n-2}^*,\ldots,a_1^*$} (cqone);
  \draw[cqarrow] (cqlast) to[bend left=55]
    node[cqlabel,below] {$a_{n-1}$} (cqzero);
  \draw[cqarrow] (cqzero) to[bend right=35]
    node[cqlabel,below] {$a_{n-1}^*$} (cqlast);
\end{tikzpicture}%
}
\caption{Doubled quivers for the cyclic fibres.}
\label{fig:cyclic-quivers}
\end{figure}
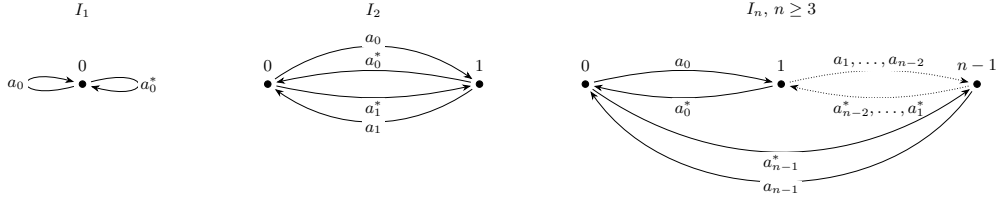

Write $x=\sum_i a_i$, $y=\sum_i a_i^*$, and $z=1+yx$. Put
\[
 D=\bigoplus_{i\in\ZZ/n}\kk[z^{\pm1}]e_i,\qquad
 \boldsymbol q=\sum_iq_ie_i,
 \qquad \sigma(e_i)=e_{i+1},\quad \sigma(z)=\boldsymbol qz.
\]
These formulas define an automorphism of $D$. Its iterates account
for the individual vertex parameters; for example,
$\sigma^n(z)=(\prod_iq_i)z$. Set $t=z-1$.

\begin{proposition}\label{prop:bulk-cycle}
The algebra $\Lambda^{\boldsymbol q}(\widetilde A_{n-1})$ has generators $D,x,y$ and relations
\begin{equation}\label{eq:bulk-gwa}
 xd=\sigma(d)x,\qquad yd=\sigma^{-1}(d)y\quad(d\in D),
 \qquad yx=t,\quad xy=\sigma(t).
\end{equation}
Set $U_p=x^p$ for $p\ge0$ and $U_p=y^{-p}$ for $p<0$.
Each corner $e_j\Lambda^{\boldsymbol q}(\widetilde A_{n-1})e_i$ has basis
\begin{equation}\label{eq:bulk-cycle-basis}
 \bigl\{U_pz^re_i:
 p,r\in\ZZ,\ p\equiv j-i\pmod n\bigr\}.
\end{equation}
\end{proposition}

\begin{proof}
Summing \Cref{eq:cycle-relations} gives $yx=z-1$,
$xy=\boldsymbol qz-1$, and
$xz=\boldsymbol qzx$. Together with the endpoints of the arrows,
these are \Cref{eq:bulk-gwa}. Conversely, in that presentation
$a_i=xe_i$ and $a_i^*=ye_{i+1}$ satisfy the vertex relations and
all localisation factors are invertible: their nontrivial corner
entries are $ze_i$ or $q_ize_i$.

The basis assertion is \cite[Proposition~7.1]{KS}. Kaplan and Schedler
concatenate paths from left to right, so our algebra is the opposite
of theirs with parameters $q_i^{-1}$, with $x$, $y$, $z$ corresponding
to their $a$, $a^*$, $1+aa^*$. Their left $S$-basis
$(1+aa^*)^ma^\ell$, $(1+aa^*)^m(a^*)^\ell$, valid over every field and
for every nonzero parameter tuple, is our set of elements $U_pz^re_i$.
Since $U_pz^re_i$ lies in $e_j\Lambda^{\boldsymbol q}(\widetilde A_{n-1})e_i$
with $j\equiv i+p$, this is \Cref{eq:bulk-cycle-basis}.
\end{proof}

For a single vertex,
\[
 \Lambda^q(\widetilde A_0)=
 \kk\langle x,y\rangle[(1+yx)^{-1}]/(xy-qyx-q+1).
\]

At $q=1$, this is the commutative algebra
$\kk[x,y,(1+xy)^{-1}]$.

\begin{proposition}\label{prop:cycle-algebra}
At $\boldsymbol q=\boldsymbol1$, give
\[
 R=\kk[x,y,z^{\pm1}]/(xy-z+1)
\]
the $\ZZ/n$-grading $\deg x=1$, $\deg y=-1$, and $\deg z=0$.
Then, over any field,
\[
 \Lambda^1(\widetilde A_{n-1})\cong
 \bigl\{(r_{ji})\in M_n(R):r_{ji}\in R_{j-i}\bigr\}.
\]
The isomorphism sends $e_i$ to $E_{ii}$, $a_i$ to $xE_{i+1,i}$,
and $a_i^*$ to $yE_{i,i+1}$, where $E_{ji}$ are the matrix units.
\end{proposition}

\begin{proof}
The assignments satisfy the vertex relations and send every localisation
factor to an invertible diagonal matrix with entries $1$ or $z$.
They therefore define a homomorphism from
$\Lambda^1(\widetilde A_{n-1})$ to the displayed matrix algebra.

Set $V_p=x^p$ for $p\geq0$ and $V_p=y^{-p}$ for $p<0$.
By the single-vertex case of \Cref{prop:bulk-cycle}, the elements
$V_pz^r$, $p,r\in\ZZ$, form a basis of $R$, with degree $p$ modulo $n$.
The homomorphism sends the basis from \Cref{prop:bulk-cycle}
bijectively to
\[
 V_pz^rE_{ji},\qquad p\equiv j-i\pmod n.
\]
Hence it is an isomorphism.
\end{proof}

If $n$ is invertible in $\kk$ and $\kk$ contains a primitive $n$-th
root of unity $\zeta$, the algebra at $\boldsymbol q=\boldsymbol1$ also has the skew group
description
\[
 \Lambda^1(\widetilde A_{n-1})\cong R\rtimes\mu_n(\kk),
\]
where $\mu_n(\kk)=\langle\zeta\rangle$ acts on $x,y,z$ with weights
$1,-1,0$. Indeed, $g=\sum_i\zeta^ie_i$ satisfies
$gx=\zeta xg$, $gy=\zeta^{-1}yg$, and $g^n=1$. Conversely, in the
skew group algebra the formulas
\[
 e_i=\frac1n\sum_{r=0}^{n-1}\zeta^{-ir}g^r,\qquad
 a_i=xe_i,\qquad a_i^*=ye_{i+1}
\]
recover the quiver generators. These constructions are inverse.

Kaplan and Schedler prove that
$\Lambda^{\boldsymbol q}(\widetilde A_{n-1})$ is $2$-Calabi-Yau
and prime over any field and at every parameter tuple. At the
trivial bulk class it is a noncommutative crepant resolution of its
centre, whose spectrum is the affine surface $z^n+xy+xyz=0$
\cite[Theorem~1.2 and Section~6.2]{KS}. For $n>1$, this surface has
an $A_{n-1}$ singularity; for $n=1$ it is smooth.

\subsection{Affine \texorpdfstring{$D$}{D} diagrams}

For $I_n^*$ with $n>0$, label the chain by $c_0,\ldots,c_n$ and orient
its arrows as $b_k:c_{k-1}\to c_k$. Orient the four leaf arrows
$a_1,a_2$ towards $c_0$ and $a_3,a_4$ towards $c_n$, and order the
arrows as $a_1,a_2,b_1,\ldots,b_n,a_3,a_4$. Write
$u_i=a_ia_i^*$, using $1$ for the identity in the relevant vertex corner.
The arrow labels in \Cref{fig:affine-d-quivers} follow this convention.

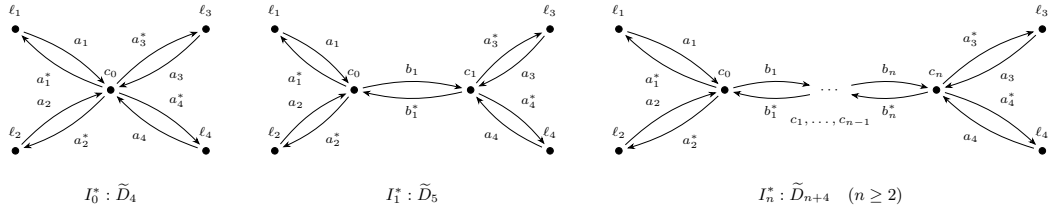
\begin{figure}[htbp]
\centering
\begingroup
\newcommand{\Dquiveredge}[3]{%
  \draw[->,shorten >=2pt,shorten <=2pt] (#1) to[bend left=13] node[auto] {$#3$} (#2);
  \draw[->,shorten >=2pt,shorten <=2pt] (#2) to[bend left=13] node[auto] {$#3^*$} (#1);%
}
\scalebox{0.7}{%
\begin{tikzpicture}[
  >={Stealth[length=4pt]},
  every node/.style={font=\scriptsize},
  v/.style={circle,draw,fill=black,inner sep=0pt,minimum size=3.5pt,outer sep=1pt},
  every edge/.style={draw}]
\begin{scope}[xshift=-8cm]
  \node[v,label={[font=\scriptsize,inner sep=1pt,yshift=3pt]above:$c_0$}] (d0c) at (0,0) {};
  \node[v,label={[font=\scriptsize,inner sep=1pt,yshift=3pt]above:$\ell_1$}] (d0l1) at (-1.8,1.15) {};
  \node[v,label={[font=\scriptsize,inner sep=1pt,yshift=3pt]above:$\ell_2$}] (d0l2) at (-1.8,-1.15) {};
  \node[v,label={[font=\scriptsize,inner sep=1pt,yshift=3pt]above:$\ell_3$}] (d0l3) at (1.8,1.15) {};
  \node[v,label={[font=\scriptsize,inner sep=1pt,yshift=3pt]above:$\ell_4$}] (d0l4) at (1.8,-1.15) {};
  \Dquiveredge{d0l1}{d0c}{a_1}
  \Dquiveredge{d0l2}{d0c}{a_2}
  \Dquiveredge{d0l3}{d0c}{a_3}
  \Dquiveredge{d0l4}{d0c}{a_4}
  \node[font=\small] at (0,-1.95) {$I_0^*: \widetilde D_4$};
\end{scope}
\begin{scope}[xshift=-3.4cm]
  \node[v,label={[font=\scriptsize,inner sep=1pt,yshift=3pt]above:$c_0$}] (d1c0) at (0,0) {};
  \node[v,label={[font=\scriptsize,inner sep=1pt,yshift=3pt]above:$c_1$}] (d1c1) at (2.2,0) {};
  \node[v,label={[font=\scriptsize,inner sep=1pt,yshift=3pt]above:$\ell_1$}] (d1l1) at (-1.5,1.15) {};
  \node[v,label={[font=\scriptsize,inner sep=1pt,yshift=3pt]above:$\ell_2$}] (d1l2) at (-1.5,-1.15) {};
  \node[v,label={[font=\scriptsize,inner sep=1pt,yshift=3pt]above:$\ell_3$}] (d1l3) at (3.7,1.15) {};
  \node[v,label={[font=\scriptsize,inner sep=1pt,yshift=3pt]above:$\ell_4$}] (d1l4) at (3.7,-1.15) {};
  \Dquiveredge{d1l1}{d1c0}{a_1}
  \Dquiveredge{d1l2}{d1c0}{a_2}
  \Dquiveredge{d1c0}{d1c1}{b_1}
  \Dquiveredge{d1l3}{d1c1}{a_3}
  \Dquiveredge{d1l4}{d1c1}{a_4}
  \node[font=\small] at (1.1,-1.95) {$I_1^*: \widetilde D_5$};
\end{scope}
\begin{scope}[xshift=5.6cm]
  \node[v,label={[font=\scriptsize,inner sep=1pt,yshift=3pt]above:$c_0$}] (dnc0) at (-2,0) {};
  \node[v,label={[font=\scriptsize,inner sep=1pt,yshift=3pt]above:$c_n$}] (dncn) at (2,0) {};
  \node (dnmid) at (0,0) {$\cdots$};
  \node at (0,-.6) {$c_1,\ldots,c_{n-1}$};
  \node[v,label={[font=\scriptsize,inner sep=1pt,yshift=3pt]above:$\ell_1$}] (dnl1) at (-4,1.15) {};
  \node[v,label={[font=\scriptsize,inner sep=1pt,yshift=3pt]above:$\ell_2$}] (dnl2) at (-4,-1.15) {};
  \node[v,label={[font=\scriptsize,inner sep=1pt,yshift=3pt]above:$\ell_3$}] (dnl3) at (4,1.15) {};
  \node[v,label={[font=\scriptsize,inner sep=1pt,yshift=3pt]above:$\ell_4$}] (dnl4) at (4,-1.15) {};
  \Dquiveredge{dnl1}{dnc0}{a_1}
  \Dquiveredge{dnl2}{dnc0}{a_2}
  \Dquiveredge{dnc0}{dnmid}{b_1}
  \Dquiveredge{dnmid}{dncn}{b_n}
  \Dquiveredge{dnl3}{dncn}{a_3}
  \Dquiveredge{dnl4}{dncn}{a_4}
  \node[font=\small] at (0,-1.95)
    {$I_n^*: \widetilde D_{n+4}\quad(n\geq2)$};
\end{scope}
\end{tikzpicture}%
}
\endgroup
\caption{The doubled affine $D$ quivers.}
\label{fig:affine-d-quivers}
\end{figure}

Write $\lambda_i$ for the parameter at leaf $i$ and $\nu_k$ for
that at $c_k$.

\begin{proposition}\label{prop:d-algebra}
For every nonzero parameter tuple, the algebra
$\Lambda^{\boldsymbol q}(\widetilde D_{n+4})$ has the following
presentation. For $n>0$ its relations are
\begin{equation}\label{eq:affine-d-polynomial}
\begin{gathered}
 a_i^*a_i=(\lambda_i^{-1}-1)e_i\quad(1\leq i\leq4),\\
 (1+u_1)(1+u_2)=\nu_0(1+b_1^*b_1),\\
 1+b_kb_k^*=\nu_k(1+b_{k+1}^*b_{k+1})\quad(1\leq k<n),\\
 (1+b_nb_n^*)(1+u_3)(1+u_4)=\nu_n.
\end{gathered}
\end{equation}
For $n=0$ the last three relations are replaced by
\begin{equation}\label{eq:d4-polynomial}
 (1+u_1)(1+u_2)(1+u_3)(1+u_4)=\nu_0.
\end{equation}
\end{proposition}

\begin{proof}
The displayed relations are the vertex equations for the indicated
orientation and order. Since the graph is a tree,
\Cref{lem:factors-invertible} shows that all localisation factors are
already invertible in this polynomial quotient, giving the claimed
presentation.
\end{proof}

\subsection{The affine stars}
\label{subsec:stars}\label{subsec:central-corner}

For $I_0^*$, $IV^*$, $III^*$ and $II^*$ the graph is a star with
arms of lengths $d_j-1$, where $(d_1,\ldots,d_s)$ is $(2,2,2,2)$,
$(3,3,3)$, $(2,4,4)$ or $(2,3,6)$. Label the central vertex by $0$ and
the vertices of the $j$th arm by $(j,p)$, $1\le p<d_j$, starting next to
the centre, orient the arrows towards the centre,
$c_{jp}:(j,p)\to(j,p-1)$ with $(j,0)=0$, and order them by $j$ and
then by $p$. Write $q=q_0$. The doubled quivers of $IV^*$, $III^*$ and
$II^*$ are drawn in \Cref{fig:affine-e-quivers}; that of $I_0^*$ is in
\Cref{fig:affine-d-quivers}.

\begin{figure}[htbp]
\centering
\begingroup
\tikzset{
  kodaira e vertex/.style={
    circle,draw,fill=black,inner sep=0pt,minimum size=3.5pt,
    outer sep=1pt
  },
  kodaira e arrow/.style={
    -{Stealth[length=1.5mm,width=1mm]},line width=.4pt,shorten >=2pt,shorten <=2pt
  }
}
\newcommand{\KodairaEEdge}[3]{%
  \path[kodaira e arrow]
    (#1) edge[bend left=14]
      node[auto,font=\scriptsize,inner sep=1pt] {$c_{#3}^{*}$} (#2);
  \path[kodaira e arrow]
    (#2) edge[bend left=14]
      node[auto,font=\scriptsize,inner sep=1pt] {$c_{#3}$} (#1);
}

\begin{tabular}{@{}c@{\hspace{8mm}}c@{}}
\scalebox{0.7}{%
\begin{tikzpicture}[x=1cm,y=1cm,baseline=(current bounding box.north)]
  \node[kodaira e vertex,label={[font=\small,inner sep=1pt,yshift=3pt]left:$0$}] (e0) at (0,0) {};
  \node[kodaira e vertex,label={[font=\scriptsize,inner sep=1pt,xshift=-2pt]left:$(1,1)$}] (e1) at (0,1.5) {};
  \node[kodaira e vertex,label={[font=\scriptsize,inner sep=1pt,yshift=3pt]above:$(1,2)$}] (e2) at (0,3) {};
  \node[kodaira e vertex,label={[font=\scriptsize,inner sep=1pt,yshift=3pt,xshift=-10pt]above:$(2,1)$}] (e3) at (-1.30,-.75) {};
  \node[kodaira e vertex,label={[font=\scriptsize,inner sep=1pt,yshift=3pt,xshift=-10pt]above:$(2,2)$}] (e4) at (-2.60,-1.5) {};
  \node[kodaira e vertex,label={[font=\scriptsize,inner sep=1pt,yshift=3pt,xshift=10pt]above:$(3,1)$}] (e5) at (1.30,-.75) {};
  \node[kodaira e vertex,label={[font=\scriptsize,inner sep=1pt,yshift=3pt,xshift=10pt]above:$(3,2)$}] (e6) at (2.60,-1.5) {};
  \KodairaEEdge{e0}{e1}{11}
  \KodairaEEdge{e1}{e2}{12}
  \KodairaEEdge{e0}{e3}{21}
  \KodairaEEdge{e3}{e4}{22}
  \KodairaEEdge{e0}{e5}{31}
  \KodairaEEdge{e5}{e6}{32}
  \node at (0,-2.1) {$\widetilde E_6\quad(IV^*)$};
\end{tikzpicture}%
}
&
\begin{tabular}[t]{@{}c@{}}
\scalebox{0.7}{%
\begin{tikzpicture}[x=1cm,y=1cm,baseline=(current bounding box.north)]
  \node[kodaira e vertex,label={[font=\small,inner sep=1pt,yshift=3pt,xshift=-10pt]above:$0$}] (e0) at (0,0) {};
  \node[kodaira e vertex,label={[font=\scriptsize,inner sep=1pt,yshift=3pt]above:$(2,1)$}] (e1) at (-1.45,0) {};
  \node[kodaira e vertex,label={[font=\scriptsize,inner sep=1pt,yshift=3pt]above:$(2,2)$}] (e2) at (-2.90,0) {};
  \node[kodaira e vertex,label={[font=\scriptsize,inner sep=1pt,yshift=3pt]above:$(2,3)$}] (e3) at (-4.35,0) {};
  \node[kodaira e vertex,label={[font=\scriptsize,inner sep=1pt,yshift=3pt]above:$(3,1)$}] (e4) at (1.45,0) {};
  \node[kodaira e vertex,label={[font=\scriptsize,inner sep=1pt,yshift=3pt]above:$(3,2)$}] (e5) at (2.90,0) {};
  \node[kodaira e vertex,label={[font=\scriptsize,inner sep=1pt,yshift=3pt]above:$(3,3)$}] (e6) at (4.35,0) {};
  \node[kodaira e vertex,label={[font=\scriptsize,inner sep=1pt,yshift=3pt]above:$(1,1)$}] (e7) at (0,1.5) {};
  \KodairaEEdge{e0}{e1}{21}
  \KodairaEEdge{e1}{e2}{22}
  \KodairaEEdge{e2}{e3}{23}
  \KodairaEEdge{e0}{e4}{31}
  \KodairaEEdge{e4}{e5}{32}
  \KodairaEEdge{e5}{e6}{33}
  \KodairaEEdge{e0}{e7}{11}
  \node at (0,-.85) {$\widetilde E_7\quad(III^*)$};
\end{tikzpicture}%
}
\\[6mm]
\scalebox{0.7}{%
\begin{tikzpicture}[x=1cm,y=1cm]
  \node[kodaira e vertex,label={[font=\small,inner sep=1pt,yshift=3pt,xshift=-10pt]above:$0$}] (e0) at (0,0) {};
  \node[kodaira e vertex,label={[font=\scriptsize,inner sep=1pt,yshift=3pt]above:$(3,1)$}] (e1) at (-1.40,0) {};
  \node[kodaira e vertex,label={[font=\scriptsize,inner sep=1pt,yshift=3pt]above:$(3,2)$}] (e2) at (-2.80,0) {};
  \node[kodaira e vertex,label={[font=\scriptsize,inner sep=1pt,yshift=3pt]above:$(3,3)$}] (e3) at (-4.20,0) {};
  \node[kodaira e vertex,label={[font=\scriptsize,inner sep=1pt,yshift=3pt]above:$(3,4)$}] (e4) at (-5.60,0) {};
  \node[kodaira e vertex,label={[font=\scriptsize,inner sep=1pt,yshift=3pt]above:$(3,5)$}] (e5) at (-7,0) {};
  \node[kodaira e vertex,label={[font=\scriptsize,inner sep=1pt,yshift=3pt]above:$(2,1)$}] (e6) at (1.40,0) {};
  \node[kodaira e vertex,label={[font=\scriptsize,inner sep=1pt,yshift=3pt]above:$(2,2)$}] (e7) at (2.80,0) {};
  \node[kodaira e vertex,label={[font=\scriptsize,inner sep=1pt,yshift=3pt]above:$(1,1)$}] (e8) at (0,1.5) {};
  \KodairaEEdge{e0}{e1}{31}
  \KodairaEEdge{e1}{e2}{32}
  \KodairaEEdge{e2}{e3}{33}
  \KodairaEEdge{e3}{e4}{34}
  \KodairaEEdge{e4}{e5}{35}
  \KodairaEEdge{e0}{e6}{21}
  \KodairaEEdge{e6}{e7}{22}
  \KodairaEEdge{e0}{e8}{11}
  \node at (-2.1,-.85) {$\widetilde E_8\quad(II^*)$};
\end{tikzpicture}%
}
\end{tabular}
\end{tabular}
\endgroup
\caption{The doubled affine quivers for $IV^*$, $III^*$ and $II^*$.}
\label{fig:affine-e-quivers}
\end{figure}
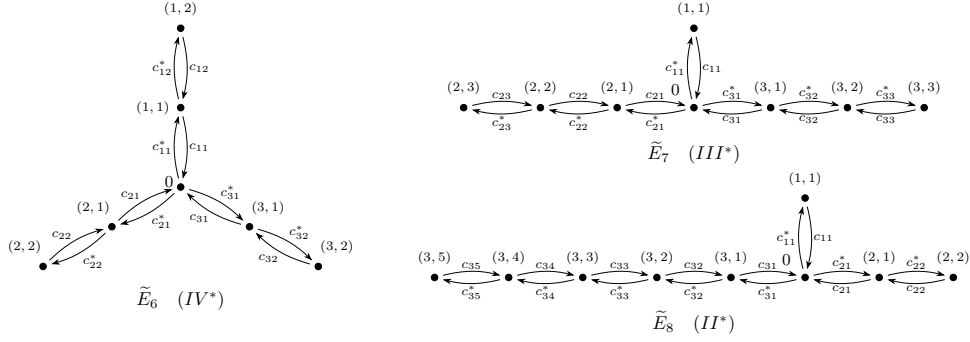

\begin{proposition}\label{prop:star-algebra}
For every nonzero parameter tuple, $\Lambda^{\boldsymbol q}(Q)$ is the
quotient of the path algebra $\kk\overline Q$ by the relations
\[
\begin{gathered}
 c_{j,d_j-1}^*c_{j,d_j-1}=(q_{(j,d_j-1)}^{-1}-1)\,e_{(j,d_j-1)},\\
 e_{(j,p)}+c_{j,p+1}c_{j,p+1}^*
 =q_{(j,p)}\bigl(e_{(j,p)}+c_{jp}^*c_{jp}\bigr)\quad(1\le p\le d_j-2),\\
 \prod_{j=1}^{s}\bigl(e_0+c_{j1}c_{j1}^*\bigr)=q\,e_0 .
\end{gathered}
\]
No localisation is needed. For $I_0^*$ this is
\Cref{eq:d4-polynomial} with $u_j=c_{j1}c_{j1}^*$.
\end{proposition}

\begin{proof}
This is \cite[Lemma~8.1]{CBShaw}, with the stated vertex and arrow labels.
\end{proof}

Put $A=\Lambda^{\boldsymbol q}(Q)$ and $H=e_0Ae_0$.
We identify $H$ with a generalised double affine Hecke algebra and
give a criterion for $e_0$ to be full, so that $A$ and $H$ are
Morita equivalent.

Set
\[
 B_{jp}=e_{(j,p)}+c_{jp}^*c_{jp},\qquad
 U_j=e_0+c_{j1}c_{j1}^*,\qquad
 \alpha_{ja}=\prod_{p=1}^a q_{(j,p)}^{-1},\quad \alpha_{j0}=1.
\]
The tip and interior relations are
\[
 B_{j,d_j-1}=q_{(j,d_j-1)}^{-1}e_{(j,d_j-1)},\qquad
 e_{(j,p)}+c_{j,p+1}c_{j,p+1}^*=q_{(j,p)}B_{jp},
\]
and the central relation is $U_1\cdots U_s=qe_0$, with this order
fixed. If $C=1+cc^*$ and $B=1+c^*c$, then
$(C-1)f(C)=cf(B)c^*$ for every polynomial $f$. Induction from
the tip therefore gives
\begin{equation}\label{eq:star-leg-polynomial}
 \prod_{a=p}^{d_j-1}
 \left(B_{jp}-(q_{(j,p)}\cdots q_{(j,a)})^{-1}e_{(j,p)}\right)=0,
 \qquad
 \prod_{a=0}^{d_j-1}(U_j-\alpha_{ja}e_0)=0.
\end{equation}

The central-corner calculation of Crawley-Boevey and Shaw
\cite[Appendix~1, Lemma~10.1]{EOR} gives
\begin{equation}\label{eq:star-central-corner}
 H\cong
 \frac{\kk\langle U_1^{\pm1},\ldots,U_s^{\pm1}\rangle}
 {\left(\prod_{a=0}^{d_j-1}(U_j-\alpha_{ja})\ (1\le j\le s),
           \ U_1\cdots U_s-q\right)}.
\end{equation}
Thus $H$ is a rank-one generalised double affine Hecke algebra.
The calculation is over $\kk$ and requires no splitting field.

\begin{lemma}\label{lem:star-fullness}
If no nonempty consecutive product $q_{(j,p)}\cdots q_{(j,a)}$ on an
arm equals $1$, then $e_0$ is full in $A$.
\end{lemma}

\begin{proof}
This is \cite[Appendix~2, Proposition~11.2]{EOR} over $\C$; the
following argument works over any field.
In $A/Ae_0A$, the first arrow pair on each arm vanishes, so
$B_{j1}=e_{(j,1)}$. The first equation in
\Cref{eq:star-leg-polynomial} then gives a nonzero scalar multiple
of $e_{(j,1)}$ equal to zero. Repeat along the arm. Every vertex
idempotent vanishes in the quotient, proving $Ae_0A=A$.
\end{proof}

At $\mathfrak b_{\mathrm{grp}}$, the numbers $\alpha_{ja}$ run
through the $d_j$th roots of unity, and all consecutive products
in \Cref{lem:star-fullness} differ from $1$. Consequently,
\begin{equation}\label{eq:star-group-corner}
 H\cong
 \kk\langle U_1,\ldots,U_s\rangle/
 (U_j^{d_j}-1,\ U_1\cdots U_s-1)
 \cong\kk[\Gamma\rtimes G].
\end{equation}
The last presentation is that of the orbifold fundamental group
of $E/G$. The orbifold fundamental group is an extension of $G$
by $\Gamma=\pi_1(E)$, split by the fixed origin of $E$.
Hence $A$ is Morita equivalent to
$\kk[\Gamma]\rtimes G$ at these parameters.

\subsection{The three nontransverse fibres}
\label{subsec:nontransverse}

By \Cref{thm:neighbourhood-surgery}(2), the attaching links for
$II$, $III$, $IV$ are the rainbow closures of $\sigma_1^3$,
$\sigma_1^4$, $(\sigma_1\sigma_2)^3$, respectively. We use the model
$G(\beta)$ of \cite[Definition~2.2 and Lemma~2.3]{CSrainbow},
retaining one idempotent and one bulk parameter per component.

For $\beta=\sigma_{i_1}\cdots\sigma_{i_m}$, let
\[
 P_\beta=P_{i_1}(z_1)\cdots P_{i_m}(z_m),
\]
where $P_i(z)$ is the identity except for the block
$\bigl(\begin{smallmatrix}0&1\\1&z\end{smallmatrix}\bigr)$
in positions $i,i+1$. The model has differential
\begin{equation}\label{eq:rainbow-differential}
 dB=P_\beta^{-1}+DLU,
\end{equation}
where $L,U$ are lower and upper unitriangular matrices, $B$ has degree
$-1$, and the remaining generators have degree zero.

The comparison in \cite[Lemma~2.3]{CSrainbow} fixes the basepoint
variables. We use
\mbox{$D=-\operatorname{diag}(q_ie_i)$} for $III$, $IV$. For $II$, we
use $D=\operatorname{diag}(-q,1)$, retaining one basepoint as in
\cite[Section~4]{CSrainbow}.

By \cite[Lemma~2.4]{CSrainbow}, eliminating the entries of $L,U$ from
$P_\beta^{-1}+DLU=0$ gives the following presentations of
$H^0(G(\beta)_{\boldsymbol q})$. The crossing generators are labelled
alphabetically in braid order, as in \Cref{fig:nontransverse-quivers}.

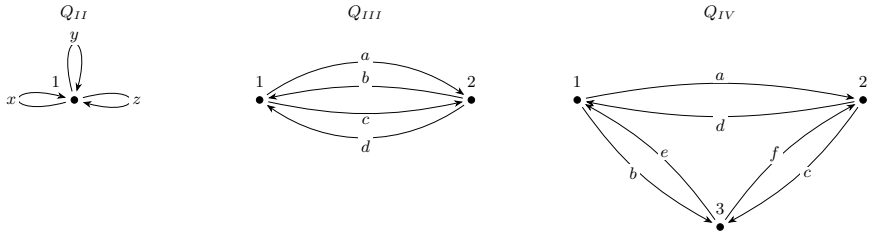
\begin{figure}[htbp]
\centering
\scalebox{0.7}{%
\begin{tikzpicture}[
  x=1cm,y=1cm,
  nqvertex/.style={circle,draw,fill=black,inner sep=0pt,minimum size=3.5pt,
    outer sep=1pt},
  nqarrow/.style={-{Stealth[length=1.6mm,width=1.2mm]},line width=.45pt,shorten >=2pt,shorten <=2pt},
  nqlabel/.style={font=\small,fill=white,inner sep=1.5pt},
  nqtitle/.style={font=\small}
]
  \node[nqtitle] at (-7,1.65) {$Q_{II}$};
  \node[nqvertex,label={[font=\small,inner sep=1pt,yshift=3pt,xshift=-10pt]above:$1$}] (nqcusp) at (-7,0) {};
  \path[nqarrow]
    (nqcusp) edge[loop left,min distance=13mm]
      node[nqlabel,left] {$x$} (nqcusp)
    (nqcusp) edge[loop above,min distance=13mm]
      node[nqlabel,above] {$y$} (nqcusp)
    (nqcusp) edge[loop right,min distance=13mm]
      node[nqlabel,right] {$z$} (nqcusp);

  \node[nqtitle] at (-1.5,1.65) {$Q_{III}$};
  \node[nqvertex,label={[font=\small,inner sep=1pt,yshift=3pt]above:$1$}] (nqiiione) at (-3.5,0) {};
  \node[nqvertex,label={[font=\small,inner sep=1pt,yshift=3pt]above:$2$}] (nqiiitwo) at (.5,0) {};
  \draw[nqarrow] (nqiiione) to[bend left=35]
    node[nqlabel,above] {$a$} (nqiiitwo);
  \draw[nqarrow] (nqiiitwo) to[bend right=12]
    node[nqlabel,above] {$b$} (nqiiione);
  \draw[nqarrow] (nqiiione) to[bend right=12]
    node[nqlabel,below] {$c$} (nqiiitwo);
  \draw[nqarrow] (nqiiitwo) to[bend left=35]
    node[nqlabel,below] {$d$} (nqiiione);

  \node[nqtitle] at (5.2,1.65) {$Q_{IV}$};
  \node[nqvertex,label={[font=\small,inner sep=1pt,yshift=3pt]above:$1$}] (nqivone) at (2.5,0) {};
  \node[nqvertex,label={[font=\small,inner sep=1pt,yshift=3pt]above:$2$}] (nqivtwo) at (7.9,0) {};
  \node[nqvertex,label={[font=\small,inner sep=1pt,yshift=3pt]above:$3$}] (nqivthree) at (5.2,-2.4) {};
  \draw[nqarrow] (nqivone) to[bend left=11]
    node[nqlabel,above] {$a$} (nqivtwo);
  \draw[nqarrow] (nqivtwo) to[bend left=11]
    node[nqlabel,below] {$d$} (nqivone);
  \draw[nqarrow] (nqivone) to[bend right=13]
    node[nqlabel,left] {$b$} (nqivthree);
  \draw[nqarrow] (nqivthree) to[bend right=13]
    node[nqlabel,right] {$e$} (nqivone);
  \draw[nqarrow] (nqivtwo) to[bend left=13]
    node[nqlabel,right] {$c$} (nqivthree);
  \draw[nqarrow] (nqivthree) to[bend left=13]
    node[nqlabel,left] {$f$} (nqivtwo);
\end{tikzpicture}%
}
\caption{Quivers for the nontransverse fibres. The arrows are the
degree-zero generators of $A_{II}(q)$, $A_{III}(q_1,q_2)$, and
$A_{IV}(q_1,q_2,q_3)$, with
relations \eqref{eq:cusp-algebra}, \eqref{eq:III-algebra}, and
\eqref{eq:IV-algebra}. The cusp loops commute at $q=1$.}
\label{fig:nontransverse-quivers}
\end{figure}

For type $II$, the two diagonal relations give
\begin{equation}\label{eq:cusp-algebra}
 A_{II}(q)=\kk\langle x,y,z\rangle/
 (q+x+z+zyx,\;1+x+z+xyz).
\end{equation}
At $q=1$, this algebra is commutative and equals
$\kk[x,y,z]/(xyz+x+z+1)$ \cite[Section~6.1.5]{ELduality}.
For type $III$, with $a,c:1\to2$ and $b,d:2\to1$,
\begin{equation}\label{eq:III-algebra}
\begin{split}
 A_{III}(q_1,q_2)=\kk Q_{III}/\big(&
 (1-q_1)e_1+ba+da+dc+dcba,\\
 &(1-q_2^{-1})e_2+ab+ad+cd+abcd\big).
\end{split}
\end{equation}
For type $IV$, with $a:1\to2$, $b:1\to3$, $c:2\to3$, $d:2\to1$,
$e:3\to1$, $f:3\to2$, put $A=ec-d$, $B=fb-a-fca$ and $C=ca-b$. Then
\begin{equation}\label{eq:IV-algebra}
\begin{gathered}
 A_{IV}(q_1,q_2,q_3)=\kk Q_{IV}/(R_1^{\boldsymbol q},
                    R_2^{\boldsymbol q},R_3^{\boldsymbol q}),\\
 R_1^{\boldsymbol q}=(1-q_1)e_1+da+eb-eca,\\
 R_2^{\boldsymbol q}=(1-q_2)e_2+fc-Bq_1^{-1}A,\\
 R_3^{\boldsymbol q}=(1-q_3)e_3+Cq_1^{-1}e
       -(-c-Cq_1^{-1}A)q_2^{-1}(-f+Bq_1^{-1}e).
\end{gathered}
\end{equation}
Each $R_i^{\boldsymbol q}$ lies in $e_i\kk Q_{IV}e_i$, and the only
inverses used in the elimination are the scalars $q_i^{-1}$. At
$\boldsymbol q=\boldsymbol1$ the relations reduce to $da+eb-eca$,
$fc-BA$ and $Ce-(-c-CA)(-f+Be)$.

For each Kodaira type $F$, let $\mathcal B_F^{\boldsymbol q}$ denote
the dg algebra obtained above: the plumbing algebra
$\mathcal B_Q^{\boldsymbol q}$ of \Cref{def:plumbing-dga} for the normal
crossing types, and the rainbow model $G(\beta)$ of
\Cref{eq:rainbow-differential} specialised at
$D=\operatorname{diag}(-q,1)$ for $II$ and $D=-\operatorname{diag}(q_ie_i)$ for
$III$, $IV$. In all cases
\begin{equation}\label{eq:dg-model-summary}
 \CE^*_{\boldsymbol q}(\Lambda_F)\simeq\mathcal B_F^{\boldsymbol q},
 \qquad
 H^0(\mathcal B_F^{\boldsymbol q})=A_F(\boldsymbol q),
 \qquad
 H^k(\mathcal B_F^{\boldsymbol q})=0\quad(k>0).
\end{equation}
Here $A_F(\boldsymbol q)$ is the algebra listed in \Cref{tab:algebras}.
The dg algebra $\mathcal B_F^{\boldsymbol q}$ is concentrated in
nonpositive degrees.

\begin{remark}\label{rem:rainbow-specialisation}
For the rainbow closures, Casals-Simons establish vanishing with the
basepoint variables retained \cite[Main Theorem]{CSrainbow}. Here we
establish vanishing in negative degrees after every bulk specialisation
by combining the Floer degree bound of \Cref{sec:foundations} with the
graded surgery comparison of \Cref{sec:surgery}.
\end{remark}

\section{Wrapped Floer cohomology}
\label{sec:foundations}

We construct cofinal Floer complexes between the completed transverse
disks in degrees zero and one. A generic choice of the boundary
radius makes all Reeb chords nondegenerate, and interpolation
to a Hamiltonian depending on the base identifies their degrees with
Morse indices. \Cref{sec:surgery} combines this bound with the
nonpositive surgery algebra to prove vanishing outside degree zero.
Throughout, the disks carry the bulk trivialisations of
\Cref{subsec:bulk-data}; these change the coefficients of Floer
operations but not the chords or their degrees.

\subsection{Holomorphic gradings and relative periods}
\label{subsec:holomorphic-gradings}

We use $\iota_{X_H}\omega=-dH$. In coordinates with
$\omega=dp\wedge dq$, the time-one map of $H(p)$ is
$(q,p)\mapsto(q+\nabla H(p),p)$. Gradings follow
\cite[Section~2d]{SeidelGraded}: the degree of a chord from $L_i$ to
$L_j$ is the absolute index of $d\phi_H^1(TL_i)$ and $TL_j$, with
the first grading transported along the Hamiltonian flow. A small
Hamiltonian with a nondegenerate minimum on a Lagrangian gives a
degree-zero constant chord.

\begin{lemma}\label{lem:phase}
A complex symplectic surface $(X,I,\Omega)$ with symplectic form
$\Ree\Omega$ has a grading structure in which every holomorphic curve
has phase zero. In holomorphic coordinates $\Omega=dw\wedge dt$,
let $F_i,F_j$ be holomorphic functions and $L_i,L_j$ their graphs
$w=dF_i(t),w=dF_j(t)$, both graded by zero.
For a real Hamiltonian $H(t)$ and a nondegenerate chord
from $L_i$ to $L_j$ lying over $t=t_0$, the point $t_0$ is a
nondegenerate critical point of
\[
 H(t)+\Ree\bigl(F_j(t)-F_i(t)\bigr),
\]
and the degree of the chord is the Morse index of this function at $t_0$.
If $H$ is subharmonic, the degree is zero or one.
\end{lemma}

\begin{proof}
Choose an $I$-Hermitian metric whose fundamental form $\alpha_I$ satisfies
\[
 \alpha_I^2=(\Ree\Omega)^2.
\]
These forms determine an almost complex structure compatible with
$\Ree\Omega$, with complex volume form
\[
 \Psi=\alpha_I-i\Im\Omega.
\]
On every holomorphic curve, $\Psi$ restricts to the positive area form
$\alpha_I$, so the curve has phase zero. The space of normalised metrics
is contractible, making the grading structure independent of this choice
up to homotopy.

Put $q=(\Ree w,\Im w)$ and $p=(-\Ree t,\Im t)$, so that
$\omega=dp\wedge dq$, and put $G_k(p)=\Ree F_k(t(p))$ for $k=i,j$.
The functions $G_k$ are harmonic, and $L_k$ is the graph
$q=-\nabla G_k(p)$. The homotopy $w=u\,dF_k(t)$, $0\le u\le1$,
consists of holomorphic graphs, hence has phase zero throughout.
Thus the grading transported from $w=0$ agrees with the holomorphic
grading.
At a chord, the tangent planes to $\phi_H^1(L_i)$ and $L_j$ are
the graphs of $D^2H-D^2G_i$ and $-D^2G_j$. The linear symplectic shear
\[
 (\delta q,\delta p)\longmapsto
 (\delta q+D^2G_j\,\delta p,\delta p)
\]
subtracts the target graph, sending these planes to the graph of
$D^2(H+G_j-G_i)$ and the vertical plane. Transporting both gradings
along this shear, the graph-index rule
\cite[Section~2d, (iii)--(v)]{SeidelGraded} identifies the chord degree
with the Morse index of $H+G_j-G_i$.
If $H$ is subharmonic, then
$\Delta(H+G_j-G_i)=\Delta H\ge0$, so its Hessian cannot be
negative definite.
\end{proof}

Extend the grading structure of \Cref{lem:phase} to
$\widehat W_\epsilon$ by Liouville transport and to the surgery through a fixed
identification in \Cref{thm:neighbourhood-surgery}. For $I_n$, where
$H^1(M;\ZZ)=\ZZ$, this identification is part of our grading convention.
The unique grading on $\widehat D_i$ restricting to zero on
$D_i\cap W_\epsilon$ is its \textbf{holomorphic grading}.

Let $t$ be the base coordinate, also pulled back to $X$, and define
the fibrewise differential $\eta_t$ by $\Omega=\eta_t\wedge dt$.
On a simply connected sector of $\Delta^*$, each $D_i$ has $m_i$
holomorphic branches. For two such branches and a path $\gamma_t$
in the fibre from the first to the second, with its homotopy class
transported over the sector, their relative period is the holomorphic function
\[
 P(t)=\int_{\gamma_t}\eta_t.
\]
There are countably many such path classes.

Integrating $\eta_t$ on the universal covers of the fibres gives
a holomorphic coordinate $w$ with $\Omega=dw\wedge dt$.
For a chord of $H(t)$ in the class of $\gamma_t$, choose lifts of
the endpoint branches joined by the lifted chord. Their $w$-coordinates
differ by $P(t)$, so \Cref{lem:phase} identifies the degree of a
nondegenerate chord with the Morse index of
\[
 H(t)+\Ree A(t),\qquad A'(t)=P(t).
\]

Use the primitive of \Cref{prop:neighbourhood-weinstein}, for which
$Z\mu>0$ near $F_{\mathrm{red}}$, where $\mu=\log|t|$. For small
$\epsilon>0$, put $Y_\epsilon=\partial W_\epsilon=\{|t|=\epsilon\}$
and $\Gamma_i=D_i\cap Y_\epsilon$. Since $\lambda|_{D_i}=0$,
the field $Z$ is tangent to $D_i$, so $\Gamma_i$ is Legendrian
and its Liouville collar lies in the holomorphic disk.

\begin{lemma}\label{lem:generic-boundary}
For almost every sufficiently small $\epsilon>0$, all positive Reeb
chords between the Legendrians $\Gamma_i\subset Y_\epsilon$ are
nondegenerate.
\end{lemma}

\begin{proof}
In the fibrewise holomorphic coordinates,
\begin{equation}\label{eq:boundary-reeb}
 R_{\lambda|_{Y_\epsilon}}=\frac{X_\mu}{Z\mu},\qquad
 X_\mu:\quad \dot w=-\frac1t,\quad \dot t=0,\qquad
 \iota_{X_\mu}\Omega=-\frac{dt}{t}.
\end{equation}
Fix a relative period $P$ on a simply connected sector and put
$B(t)=tP(t)$. The endpoint equation for a positive chord of
$X_\mu$-time $c$ is
\begin{equation}\label{eq:boundary-chord}
 B(\epsilon e^{i\theta})=-c,\qquad c>0.
\end{equation}
Its differential in $(\theta,c)$ is invertible precisely when
\begin{equation}\label{eq:boundary-transversality}
 \partial_\theta\Im B(\epsilon e^{i\theta})
 =\Ree(tB'(t))\ne0.
\end{equation}
This is transversality of the surface swept out from $\Gamma_i$
to $\Gamma_j$, hence nondegeneracy of the Reeb chord.

We first exclude $B\equiv-c<0$. Since $W_\epsilon$ retracts onto
$F_{\mathrm{red}}$, its second homology is generated by the component
classes. The form $\Omega$ restricts to zero on their normalisations,
so $\Omega=d\beta$ for a smooth complex-valued one-form $\beta$.
Its restriction to each holomorphic disk $D_i\cap W_\epsilon$ is exact.
The restrictions $\pi|_{D_i}$ and $\pi|_{D_j}$ have degrees $m_i$ and
$m_j$. Choose $N\ge1$ divisible by both, so the chosen branches return
to themselves after $N$ turns around $t=0$. Write them as
$\sigma_i(\theta),\sigma_j(\theta)$ over $t=\epsilon e^{i\theta}$,
with $\theta\in\R/(2\pi N\ZZ)$. If $B\equiv-c$, then
$\phi_{X_\mu}^c(\sigma_i(\theta))=\sigma_j(\theta)$ on the sector
where $P$ is defined. The flow is holomorphic on the regular locus and complete on
each compact fibre, so this equality holds for every $\theta$.
The cylinder
\[
 S(u,\theta)=\phi_{X_\mu}^u(\sigma_i(\theta)),\qquad 0\le u\le c,
\]
has boundary in the two disks. By \eqref{eq:boundary-reeb} and
Stokes' theorem,
\[
 -2\pi iNc=\int_S\Omega=\int_{\partial S}\beta=0,
\]
a contradiction.

For nonconstant $B$, remove the isolated zeros of $B'$. The locus
$\{B\in\R_{<0}\}$ is then a smooth real curve. Sard's theorem
applied to its radius function shows that almost every circle
meets it transversally, giving \eqref{eq:boundary-transversality}.
The radii of the removed points form a countable set; constant $B$
outside $\R_{<0}$ gives no positive chords. Taking the union of
these null exceptional sets over a countable sector cover and all
branch pairs and relative path classes proves the claim.

If $X$ is any positive multiple of the Reeb field and $\gamma$ has
endpoints $p_-,p_+$ and $X$-time $\tau$, this transversality says
\begin{equation}\label{eq:boundary-spanning}
 T_{p_+}Y_\epsilon
 =D\phi_X^\tau(T_{p_-}\Gamma_i)\oplus\R R_\alpha(p_+)
 \oplus T_{p_+}\Gamma_j,
 \qquad \alpha=\lambda|_{Y_\epsilon}.
\end{equation}
\end{proof}

Fix $\epsilon$ supplied by \Cref{lem:generic-boundary}, and write
$W=W_\epsilon$, $Y=Y_\epsilon$, $f=\mu-\log\epsilon$ and
$\alpha=\lambda|_Y$.
Let $r$ be the Liouville radial coordinate, normalised by $r=1$ on
$Y$, so that $\lambda=r\alpha$ on its collar and completed end.

\begin{proposition}\label{prop:boundary-indices}
Every realisation of a positive Reeb chord between the $\Gamma_i$
by a radial Hamiltonian $H=h(r)$ with $h''>0$ at the chord is
nondegenerate and has degree zero or one with the holomorphic gradings.
\end{proposition}

\begin{proof}
Fix a chord $\gamma$ from $p_-\in\Gamma_i$ to
$p_+\in\Gamma_j$. To apply \Cref{lem:phase}, we interpolate between
a Hamiltonian depending only on $t$ and a radial Hamiltonian. For
$0\le s\le1$, put
\[
 f_s=(1-s)f+s(r-1),\qquad g_s=(1-s)Zf+s>0.
\]
On $Y$, $f_s=0$ and $df_s=g_s\,dr$, so $X_{f_s}=g_sR_\alpha$.
Thus $\gamma$ is a flow line of $X_{f_s}$ of a positive time $c_s$
depending smoothly on $s$. For $\sigma>0$, fixed below, put
\[
 H_s=c_sf_s+\frac\sigma2 f_s^2.
\]
Since $f_s$ is constant along its Hamiltonian flow,
$\phi_{H_s}^1(x)=\phi_{f_s}^{\,c_s+\sigma f_s(x)}(x)$.
Hence $\gamma$, suitably parametrised, is a time-one chord of every
$H_s$; $H_0$ is a function of $t$ and $H_1$ is radial.

Let $B_s=D\phi_{f_s}^{c_s}$ at $p_-$ and choose
$0\ne e\in T_{p_-}\Gamma_i$. Since the disks are Liouville invariant,
\[
 T_{p_-}\widehat D_i=\R e\oplus\R Z(p_-),\qquad
 T_{p_+}\widehat D_j=T_{p_+}\Gamma_j\oplus\R Z(p_+).
\]
Differentiating the flow gives
\[
 D\phi_{H_s}^1(e)=B_se,\qquad
 D\phi_{H_s}^1(Z(p_-))
 =B_sZ(p_-)+\sigma g_s(p_-)X_{f_s}(p_+).
\]
As $\sigma\to\infty$, the image tangent plane converges to
$\operatorname{span}\{B_se,X_{f_s}(p_+)\}$, the tangent plane to
the swept surface. By \eqref{eq:boundary-spanning}, this plane is
transverse to $T_{p_+}\widehat D_j$. Continuity and compactness of
$[0,1]$ make the convergence and transversality uniform in $s$.
Consequently $\gamma$ is nondegenerate for every $H_s$ once
$\sigma$ is sufficiently large, and its degree is independent of $s$.

At $s=1$, the Reeb and Liouville flows commute, so
$B_1Z(p_-)=Z(p_+)$. The image tangent plane is then
\[
 \operatorname{span}\{B_1e,\ Z(p_+)+\sigma R_\alpha(p_+)\}.
\]
Equation \eqref{eq:boundary-spanning} makes this transverse to
$T_{p_+}\widehat D_j$ for every $\sigma>0$. The radial chord's degree
is therefore constant over all positive radial curvatures.

Since $f=\log|t|-\log\epsilon$ is harmonic,
\[
 \Delta H_0=\sigma|\nabla f|^2=\frac{\sigma}{|t|^2}>0.
\]
Thus \Cref{lem:phase}, applied to the lifted chord, gives degree
zero or one.
The interpolation and constancy over positive radial curvatures give
the assertion for every $h$ with $h'(1)=c_1$ and $h''(1)>0$.

Inverse Liouville flow carries a chord at $r=r_0>1$ to a chord at
$r=1$ of $g(r)=h(r_0r)/r_0$, preserving its degree.
Since $g''(1)=r_0h''(r_0)>0$, the same conclusion holds at every
end radius.
\end{proof}

\begin{remark}\label{rem:kodaira-degree-zero}
In the Kodaira setting, an explicit computation of the relative periods
shows that, for sufficiently small $\epsilon$, all the chords in
\Cref{prop:boundary-indices} have degree zero. We will not need this
stronger statement.
\end{remark}

\begin{proposition}\label{prop:nonnegative-wrapping}
Let $M$ be the completed surgery of \Cref{thm:neighbourhood-surgery},
and give its cocores $\widehat D_i$ the holomorphic gradings.
There are cofinal wrapping Hamiltonians whose Floer cochain complexes
between these cocores are supported in degrees zero and one.
For any field $\kk$, every $\mathfrak b\in H^2(M;\kk^\times)$,
and all $i,j$,
\[
 \HW^k_{\mathfrak b}(\widehat D_i,\widehat D_j;\kk)=0
 \qquad(k\notin\{0,1\}).
\]
\end{proposition}

\begin{proof}
Use \Cref{thm:neighbourhood-surgery} to work on $\widehat W$.
Compactness and \Cref{lem:generic-boundary} make the positive
Reeb-chord spectrum locally finite and bounded away from zero.
Choose increasing cofinal Hamiltonians as in
\cite[Appendix~B.1.1 and B.2]{ELduality}, fixed and sufficiently small
on $W$, with convex radial profiles $h_k(r)$ on the end, eventually
linear with slopes tending to infinity outside the spectrum.
Require $h_k''>0$ at every chord radius.

On each disk, choose the compact Morse perturbation with exactly one
critical point, a minimum, matching near $Y$ a common radial
Hamiltonian of positive slope below the spectrum. The disks are
disjoint and the radial collar is invariant, so the only compact
chords are these degree-zero minima. All end chords have degree
zero or one by \Cref{prop:boundary-indices}.

Bulk weights change neither generators nor degrees. Thus all these
complexes are supported in degrees zero and one for every
$\mathfrak b$. Taking the degreewise direct limit of Floer cohomology,
as in \Cref{subsec:bulk-data}, gives the asserted vanishing.
\end{proof}

\section{Surgery and generation}
\label{sec:surgery}

We now identify the algebras of \Cref{sec:algebras} with the
bulk-deformed wrapped endomorphisms of the cocores, using the surgery
formula of Bourgeois, Ekholm and Eliashberg \cite{BEE} in the form of
\cite{ELduality}, and we prove that the cocores generate the
bulk-deformed category, following \cite{CDRGG}. Both arguments are
proved in the literature for the ordinary wrapped Fukaya category.
These arguments extend to the bulk-deformed categories of
\Cref{subsec:bulk-data}: bulk weights respect gluing, so the algebraic
identities remain valid, and we check below that the comparison maps
retain invertible leading coefficients.

\subsection{Bulk-deformed branes and weights}
\label{subsec:bulk-data}

Let $M$ be the completion of a neighbourhood of $F$, identified with the
Legendrian surgery of \Cref{thm:neighbourhood-surgery}, and let $r$ be
the number of components. Since $M$ retracts onto $F_{\mathrm{red}}$,
\[
 H_2(M;\ZZ)=\bigoplus_{i=1}^r\ZZ\gamma_i,\qquad \gamma_i=[C_i],
 \qquad
 H_1(M;\ZZ)=
 \begin{cases}\ZZ,&F=I_n,\\0,&\text{otherwise},\end{cases}
\]
where $C_i$ carries its complex orientation. By the proof of
\Cref{thm:neighbourhood-surgery}, $\gamma_i$ is the class of the $i$th
handle: the core sphere for the plumbings, and the core capped in the
$0$-handle for $II$, $III$ and $IV$. Reversing the orientation of $C_i$
inverts $q_i$.

View the bulk class as a character and write its coordinates
\begin{equation}\label{eq:bulk-parameters}
 \mathfrak{b}:H_2(M;\ZZ)\longrightarrow\kk^\times,
 \qquad q_i=\mathfrak{b}(\gamma_i).
\end{equation}
The parameter space is the full
torus $(\kk^\times)^r$, including roots of unity and the trivial
character. Since $H_1(M;\ZZ)$ is free, the universal coefficient
theorem identifies these characters with $H^2(M;\kk^\times)$.

Without giving full details, we briefly indicate what
$\mathcal W_{\mathfrak b}(M)$ is. Choose a multiplicative singular
two-cocycle $\eta$ representing $\mathfrak b$. A bulk-deformed brane is
a properly embedded exact graded spin Lagrangian, conical outside a
compact set, equipped with a trivialisation of $\eta$ along it and a
finite-rank local system. Each chord carries the Hom space between the endpoint
fibres of the local systems, tensored with a one-dimensional bulk
coefficient line. After choosing a basis vector in each bulk
coefficient line, a polygon contributes a scalar bulk weight
multiplying the usual composition and parallel-transport maps. This weight is the product of the cocycle
values on a triangulation, corrected by the boundary trivialisations;
it is independent of subdivision and multiplicative under gluing.
Changes of trivialisation are absorbed by tensoring the local systems
with rank-one local systems, and changing the representative $\eta$
gives an equivalent category. We write $\mathcal W_{\mathfrak b}(M)$
for the split-closed triangulated envelope of the resulting wrapped
$A_\infty$ category. The contractible cocores admit the required
trivialisations, unique up to isomorphism.

We use the Hamiltonian model of wrapped Floer cohomology:
\[
 \HW^*_{\mathfrak b}(L_0,L_1;\kk)
 =\varinjlim_k \operatorname{HF}^*_{\mathfrak b}(L_0,L_1;H_k;\kk).
\]
Here $H_k$ is an increasing cofinal sequence of Hamiltonians that are
eventually linear in the Liouville radial coordinate, with slopes tending
to infinity outside the Reeb chord spectrum; the maps are monotone
continuation maps. The comparison with the geometric wrapped model is
\cite[Appendix~B.1.1, Lemmas~79--80]{ELduality}. Its bulk-deformed
version follows from the gluing rule for weights, as in
\Cref{lem:coefficient-principle}.

The trivialisation of $\eta$ is part of the brane data. For a closed
two-form, the analogous datum is a primitive of its restriction
\cite[Section~3.1, Remark~3.2]{SiegelTwist}. Here bulk deformation means
a multiplicative degree-two twist, rather than a series of operations
with arbitrary interior insertions.

For the cocores, a capped polygon of class $\beta$ carries the universal
weight $[\beta]\in\ZZ[H_2(M;\ZZ)]$. Evaluating
$[\beta]\mapsto\mathfrak b(\beta)$ gives its bulk weight. For cyclic
plumbings, choose reference paths recording chord winding, compatibly
with gluing. We cap the polygon boundary as a whole; individual
noncontractible chords need not admit capping disks.

Weights do not change the Floer equations, degrees, energy bounds,
or maximum principle. Exactness excludes nonconstant sphere bubbles and bubbles
on an embedded exact boundary component. Fixed asymptotic data give
finite zero-dimensional counts, so the coefficients are finite sums
in this group algebra. This is the bulk-deformed
version of \cite[Sections~3.1 and 4.3.2]{SiegelTwist}; its construction
uses only multiplication of weights, not that paper's particular
choice of an exponential coefficient field.

\subsection{Bulk weights in Floer-theoretic arguments}
\label{subsec:coefficient-principle}

Write $w(u)$ for the scalar bulk weight of a holomorphic curve $u$,
defined in \Cref{subsec:bulk-data}. Bulk-deformed operations count rigid
curves with their signed weights $\pm w(u)$.

\begin{lemma}\label{lem:coefficient-principle}
Let $\mathcal M$ be a compactified one-dimensional moduli space of
holomorphic curves with boundary on exact branes carrying bulk
trivialisations, in a Liouville manifold or cobordism. Suppose its
boundary consists of broken configurations of rigid curves. The bulk
weight is constant on each connected component and equals the product
of the weights of the pieces at every breaking. Consequently, the
signed boundary-count identities hold with bulk weights. The signed
scalar weight of each rigid curve is a unit.
\end{lemma}

\begin{proof}
The cocycle identity and the boundary trivialisations make the weight
invariant under deformation and multiplicative under gluing. The signed
weighted boundary count therefore vanishes on each compact
one-dimensional component. Each signed scalar weight is a unit in the
coefficient ring.
\end{proof}

\subsection{The surgery comparison}
\label{subsec:surgery-comparison}

Let $F$ be one of the Kodaira types, $M$ the completion of the surgery
of \Cref{thm:neighbourhood-surgery}, $\Lambda_F$ its attaching link and
$\widehat D_i$ its cocores with their holomorphic gradings. The grading
of the cocores induces a Maslov potential on each component of
$\Lambda_F$, and hence a $\ZZ$-grading on $\CE^*_{\boldsymbol q}(\Lambda_F)$;
for the cycles it also fixes the trivialisation along the $1$-handle.
We call this the induced grading. The geometric gradings compared below
differ by an integer shift on each component, which changes the degree of a chord from $\Lambda_i$ to
$\Lambda_j$ by the difference of the shifts, and, for the cycles, by an
element $m\in H^1(M;\ZZ)=\ZZ$, which changes the degree of a word by $m$
times its winding number around the cycle.

\begin{proposition}\label{prop:bulk-cocores}
For every bulk class $\mathfrak b$ as in \Cref{eq:bulk-parameters},
there is an $A_\infty$ quasi-isomorphism respecting the idempotents
and gradings,
\begin{equation}\label{eq:wrapped-derived-mpa}
 \CW^*_{\mathfrak{b}}\Bigl(\bigoplus_i\widehat D_i,\bigoplus_i\widehat D_i\Bigr)
 \simeq\mathcal B_F^{\boldsymbol q},
\end{equation}
where the induced grading on $\mathcal B_F^{\boldsymbol q}$ agrees
with that of \Cref{sec:algebras}.

The projection $\mathcal B_F^{\boldsymbol q}\to A_F(\boldsymbol q)$
is a quasi-isomorphism. Thus $A_F(\boldsymbol q)$ is a minimal model
for the cocore endomorphism algebra, and
\begin{equation}\label{eq:bulk-hw}
 \HW_{\mathfrak{b}}^a(\widehat D_i,\widehat D_j;\kk)=
 \begin{cases}e_jA_F(\boldsymbol q)e_i,&a=0,\\0,&a\ne0.\end{cases}
\end{equation}
\end{proposition}

\begin{proof}
By \cite[Theorem~2 and Theorem~83]{ELduality},
following \cite[Section~5.4 and Remark~5.9]{BEE}, there is an
$A_\infty$ quasi-isomorphism
$\Phi:\CW^*(\bigoplus_i\widehat D_i)\to\CE^*(\Lambda_F)$ for the
ordinary wrapped Floer complex, whose components count holomorphic
disks in the surgery cobordism with boundary on the cocores and on the
cores. It is defined for links in the boundary of any subcritical
Weinstein manifold. For the cyclic plumbings, we use the
fixed integer grading of \Cref{subsec:holomorphic-gradings} and
compatible basepoint paths as in
\cite[Appendix~A, Remark~72]{ELduality}, together with the coefficient
choices of \Cref{subsec:bulk-data}. The map preserves gradings when
$\CE^*(\Lambda_F)$ carries the induced grading. Equip every cocore with
its bulk trivialisation and count every disk with its weight. The
target differential counts disks in the symplectisation of the
boundary before surgery.

Choose reference arcs on each attaching circle cut at its basepoint,
with compatible reference paths and capping chains as in
\Cref{subsec:bulk-data}. Closing a disk $u$ through the handle cores,
a positive passage through the $i$th basepoint contributes the oriented
capped core class $\gamma_i=[C_i]$, and a negative passage contributes
$-\gamma_i$. Thus, if $n_i$ is the signed number of passages, then
\[
 [u_{\mathrm{cap}}]=\sum_i n_i[C_i],\qquad
 w(u)=\mathfrak b([u_{\mathrm{cap}}])=\prod_i q_i^{n_i}.
\]
Here the coefficient-line bases are chosen compatibly with the cappings;
changing chord cappings rescales these bases. The weight is therefore
exactly that obtained by substituting $t_i=q_ie_i$ in the basepoint
word. Hence the weighted target is $\CE^*_{\boldsymbol q}(\Lambda_F)$.
By \Cref{lem:coefficient-principle}, the weighted surgery map
$\Phi_{\mathfrak b}$ satisfies the $A_\infty$ relations.

We show that $\Phi_{\mathfrak b}$ is a quasi-isomorphism following
\cite[Theorem~83]{ELduality}, using the chord correspondence and disk
counts of \cite{EkholmCurves}. As noted in
\cite[Introduction]{EkholmCurves}, these results also apply without
the simple-connectivity assumption; see \cite[Section~7.3]{EkholmOancea}.
Choose an action cutoff $A>0$ outside the action set of $\Lambda_F$,
including the total actions of composable words of Reeb chords.
For sufficiently thin surgery handles, the cocore chords of action
less than $A$ correspond to such words of total action less than $A$
\cite[Theorem~1.2]{EkholmCurves}. Their actions differ by an error
tending to zero with the handle width; we choose the width so that
this error is smaller than the gap to $A$. The unweighted linear term
sends each chord to its corresponding word with coefficient $\pm1$,
plus terms of smaller action \cite[Theorem~7.5(a)]{EkholmCurves}.
Let $d_w$ denote this diagonal coefficient with weights. For a
one-letter word, the local disk is unique
\cite[Lemma~7.4]{EkholmCurves}; the deformation in
\cite[Lemma~7.6]{EkholmCurves} has a compact one-dimensional moduli
space with no broken ends, so it preserves the weighted count and
$d_w$ is a unit. For a concatenation $w=uv$, let $p_{u,v}$ be the
weighted two-input disk count of
\cite[Theorem~7.5(c)]{EkholmCurves}. By \Cref{lem:coefficient-principle},
the boundary-count identity in the proof of that theorem remains valid
with bulk weights and gives
\[
 p_{u,v}d_{uv}=\pm d_ud_v.
\]
Induction on word length makes the right-hand side a unit, hence
$d_{uv}$ is a unit. The constant generators have the local unit-disk
coefficients of \cite[Remark~82]{ELduality}, which are units as well.
Thus, for each cutoff and sufficiently thin handles,
$\Phi_{\mathfrak b,1}$ is triangular with invertible diagonal on the
finitely many generators below the cutoff, hence is an isomorphism
there. As the cutoff increases, shrink the handles and use the
compatible cobordism maps of \cite[Theorem~83]{ELduality}. These are
continuation quasi-isomorphisms, and their compatibility with the
surgery maps persists with bulk weights by
\Cref{lem:coefficient-principle}. Passing to the direct limit therefore
proves that $\Phi_{\mathfrak b}$ is a quasi-isomorphism.
The identification of the wrapped
complexes with their Hamiltonian versions is
\cite[Appendix~B.1.1, Lemmas~79--80]{ELduality}, which is again a
comparison of counts of rigid disks and continuation maps and holds
with weights.

Composing $\Phi_{\mathfrak b}$ with the comparison of
\Cref{thm:plumbing-ce} and \Cref{subsec:nontransverse} gives
\Cref{eq:wrapped-derived-mpa} with the induced grading on
$\mathcal B_F^{\boldsymbol q}$. The changes of generators and
destabilisations respect the endpoints and homotopy classes of words,
so they remain graded under shifts of the Maslov potentials and of the
ambient grading structure. The internal Laurent reduction is also
graded for an arbitrary degree of its Laurent generator
\cite[Proposition~14]{ELplumbing}. Thus this comparison does not require
agreement with the grading of \Cref{sec:algebras}. By
\Cref{prop:nonnegative-wrapping}, its cohomology vanishes in negative
induced degrees.

Regrading changes the degrees of the cohomology classes but preserves
the underlying cohomology vector spaces. Let $d$ denote the degree
in \Cref{sec:algebras}.
The induced degree of a homogeneous class in the $(j,i)$ corner is
\[
 d+s_j-s_i+m\nu,
\]
where $\nu$ is its winding number and the term $m\nu$ occurs only for
$I_n$. These grading differences depend on the geometric brane data,
not on the bulk weights. We determine them at
$\boldsymbol q=\boldsymbol1$.

For $I_n$, \Cref{prop:bulk-cycle} gives nonzero classes
$x^ne_0$ and $y^ne_0$ of standard degree zero and winding numbers
$1$ and $-1$. Their induced degrees are $m$ and $-m$, respectively,
so nonnegativity gives $m=0$. This includes $I_1$.

For every type except $II$, let $Q_F$ be its displayed quiver,
including both arrows over every edge in the affine cases, and let
$J$ be the arrow ideal. At $\boldsymbol q=\boldsymbol1$ there is a
surjection
\[
 A_F(\boldsymbol1)\longrightarrow\kk Q_F/J^2.
\]
For the affine types every localisation factor maps to the identity
and both vertex products map to $e_i$. For $III$ and $IV$, the
relations \eqref{eq:III-algebra} and \eqref{eq:IV-algebra} lie in
$J^2$ at these parameters. Thus every arrow survives in
$H^0(\mathcal B_F^{\boldsymbol1})$. If an edge joins $i$ and $j$,
arrows in the two directions have induced degrees $s_j-s_i$ and
$s_i-s_j$. Both are nonnegative, so $s_i=s_j$. The quivers are
connected, hence all shifts agree. For $II$ there is only one vertex.
A common shift does not change endomorphism degrees. The induced
grading therefore agrees with that of \Cref{sec:algebras} for every
bulk class.

For arbitrary $\boldsymbol q$, the algebra
$\mathcal B_F^{\boldsymbol q}$ is nonpositively graded, whereas its
cohomology is nonnegative by the surgery comparison and
\Cref{prop:nonnegative-wrapping}. Consequently
$H^k(\mathcal B_F^{\boldsymbol q})=0$ for $k\ne0$, and the canonical
projection
\[
 \mathcal B_F^{\boldsymbol q}\longrightarrow
 H^0(\mathcal B_F^{\boldsymbol q})=A_F(\boldsymbol q)
\]
is a dg algebra quasi-isomorphism. Hence $A_F(\boldsymbol q)$ is a
minimal model for the cocore endomorphism algebra. Restricting to the
corners gives \Cref{eq:bulk-hw}.
\end{proof}

\subsection{Generation in the bulk-deformed category}
\label{subsec:generation}

\Cref{prop:bulk-cocores} identifies the split-closed subcategory
generated by the bulk-deformed cocores with $\Perf(A_F(\boldsymbol q))$.
The corresponding statement for all bulk-deformed branes requires a
generation argument. We give the coefficient extension of the
geometric proof of \cite[Theorem~1.1 and Remark~1.3]{CDRGG}, which uses only the
ingredients of \Cref{lem:coefficient-principle}, together with a
single invertible triangle map at each step of the construction
of the augmentation.

\begin{proposition}\label{prop:bulk-generation}
For every Kodaira type $F$ and every bulk class $\mathfrak{b}$, the
cocores $\widehat D_i$ generate
$\mathcal W_{\mathfrak{b}}(M)$, with the bulk-deformed brane
convention of \Cref{subsec:bulk-data}.
\end{proposition}
\begin{proof}
We use linear wrapping throughout this argument, so the cocycle is
transported by continuation rather than kept fixed under a Liouville
rescaling. All branes, isotopies, and cobordisms below carry the
corresponding transported boundary trivialisations, and all
identities between curve counts are used in their weighted form
(\Cref{lem:coefficient-principle}). Let $(L,E)$ be a brane with a
finite-rank local system. Make $L$ transverse to the skeleton and take
a Hamiltonian translate $D_i$ of a cocore at each intersection $a_i$.
Shift the grading of each $D_i$ so that its short Legendrian chord
$a_i$ has degree zero.
Equip $D_i$ with the constant local system $V_i=E_{a_i}$, and write
$\mathfrak D_i=(D_i,V_i)$. Transport the cocycle trivialisations
along these isotopies. The preparations in
\cite[Section~9.1 and Lemma~9.4]{CDRGG} use only the geometry,
actions, and regularity of the curves and thus are unaffected by
weights. They order the cocores so that the quotient of the
Legendrian algebra by order-reversing chords has generators
$a_i,b_{ij}^m,c_{ij}^m$ and an action filtration.
The chords $a_i$ and $b_{ij}^m$ run from $D_i$ to $L$, while
$c_{ij}^m$ runs from $D_i$ to $D_j$, for $i<j$.
The distinguished triangle contributing to $db_{ij}^m$ has negative
boundary word $a_jc_{ij}^m$.

We seek an augmentation $\varepsilon$ assigning to each chord a map
between its endpoint fibres. Disk words are evaluated by composing
these maps with parallel transport along the boundary arcs.
We use the triangular differential of \cite[Lemma~9.5]{CDRGG},
proceeding by decreasing $i$ and then increasing action. Set
$\varepsilon(a_i)=\operatorname{id}_{V_i}$ and
$\varepsilon(b_{ij}^m)=0$. Since $da_i=0$, their augmentation
equations reduce to
\[
 0=\varepsilon(db_{ij}^m)
   =u_{ij}^m P_{ij}^m\varepsilon(c_{ij}^m)+\varepsilon(R_{ij}^m),
\]
where $R_{ij}^m$ and $dc_{ij}^m$ involve only generators treated
earlier in this order. The coefficient $u_{ij}^m\in\kk^\times$ is
the signed weight of the unique triangle of \cite[Lemma~9.4]{CDRGG}.
Here $P_{ij}^m$ is parallel transport in $E$ along its boundary arc
from $a_j$ to the endpoint of $b_{ij}^m$ on $L$.
The triangular structure follows from endpoints and action
inequalities, which bulk weights do not change.

Since $P_{ij}^m$ is invertible, this determines
\[
 \varepsilon(c_{ij}^m)=-(u_{ij}^m)^{-1}(P_{ij}^m)^{-1}
 \varepsilon(R_{ij}^m)\in\operatorname{Hom}(V_i,V_j).
\]
Applying $\varepsilon$ to $d^2b_{ij}^m=0$, the previously verified
augmentation equations give $u_{ij}^m P_{ij}^m\varepsilon(dc_{ij}^m)=0$,
hence $\varepsilon(dc_{ij}^m)=0$. The construction respects degrees:
the $a_i$ have degree zero, the differential is homogeneous, and bulk
weights and parallel transport have degree zero. Setting all order-reversing
chords to zero completes the augmentation.

Perform surgery at the $a_i$, using the identity maps to glue the
local systems and choosing bulk trivialisations so that the local
scalar coefficients are $1$. These data extend across the surgery
handles, and the weighted cobordism maps transfer the
augmentation. The comparison of \cite[Proposition~8.16]{CDRGG}
applies with bulk weights: \Cref{lem:coefficient-principle} supplies
the required identities, and the unique local triangles of
\cite[Lemma~4.14 and Corollary~4.18]{CDRGG} contribute invertible
coefficients. Thus, naturally for every test brane $T$,
\[
 \CW_{\mathfrak{b}}(T,\overline L)
 \simeq\operatorname{hom}_{\mathcal W_{\mathfrak{b}}}(T,\mathfrak L),
\]
where $\overline L$ is the surgered immersion with its glued local system and transferred
augmentation and $\mathfrak L$ is a finite twisted complex on
$\mathfrak D_1,\ldots,\mathfrak D_k,(L,E)$, with $(L,E)$ its final term.

Since $\overline L$ misses the skeleton, its wrapped morphisms vanish
by the displacement and action argument of
\cite[Proposition~7.6]{CDRGG}, which is unaffected by bulk weights.
Hence $\mathfrak L$ is zero, expressing $(L,E)$ as a twisted complex of
the $\mathfrak D_i$. Each $\mathfrak D_i$ is a direct sum of
$\dim V_i$ copies of $D_i$, and continuation identifies $D_i$ with
the corresponding grading shift of its original cocore. All resulting
twisted complexes are finite.
\end{proof}

\begin{proof}[Proof of \Cref{thm:main}]
By \Cref{thm:neighbourhood-surgery}, the transverse holomorphic disks
complete to the cocores $\widehat D_i$.
\Cref{prop:bulk-cocores} computes their bulk-deformed wrapped cohomology
and gives the minimal model $A_F(\boldsymbol q)$ for their endomorphism
algebra. Since the cocores generate by \Cref{prop:bulk-generation}, the
Yoneda functor gives
\[
 \mathcal W_{\mathfrak b}(M)\simeq\Perf\bigl(A_F(\boldsymbol q)\bigr).
\]
\end{proof}

\section{Mirror symmetry and crepant resolutions}
\label{sec:star-consequences}

We prove \Cref{cor:unipotent-mirror,cor:orbifold-mirror} using
\Cref{thm:main} and \Cref{cor:affine-formality}.
We also prove \Cref{thm:higher-orbifold} and establish the algebraic side
of the Hilbert-scheme conjecture in \Cref{prop:hilbert-mirror-model}.

\subsection{The homological consequence of formality}

Kaplan-Schedler show that exactness of the bimodule complex of
Crawley-Boevey and Shaw implies the $2$-Calabi-Yau property, and establish
this exactness for quivers containing a cycle
\cite[Corollaries~3.19 and~3.20]{KS}. For affine $D$ and $E$ trees,
\Cref{cor:affine-formality} supplies the required exactness, giving
the following consequence.

\begin{lemma}\label{lem:star-cy}
Let $Q$ be an affine tree of type $D$ or $E$ and $A=\Lambda^{\boldsymbol q}(Q)$.
For any field and every nonzero parameter tuple, $A$ has global
dimension at most two and is bimodule $2$-Calabi-Yau. In
particular, for every finite-dimensional left or right $A$-module
$V$,
\begin{equation}\label{eq:star-finite-duality}
 \operatorname{RHom}_A(V,A)
 \simeq\operatorname{Hom}_{\kk}(V,\kk)[-2],
\end{equation}
with the corresponding opposite module structure.
\end{lemma}

\begin{proof}
Consider the bimodule complex of Crawley-Boevey and Shaw
\cite[Lemma~3.1]{CBShaw}, written using the relation generators
for $A_v-q_vB_v$:
\[
 P_2\xrightarrow{d_2}P_1\xrightarrow{d_1}P_0
 \xrightarrow{\mu}A\longrightarrow0,
\]
where
\[
 P_2=P_0=\bigoplus_{v\in Q_0}Ae_v\otimes_{\kk}e_vA,
 \qquad
 P_1=\bigoplus_{a\in\overline Q_1}
 Ae_{t(a)}\otimes_{\kk}e_{s(a)}A.
\]
Let $\eta_v$, $\eta_a$ and $\theta_v$ denote the standard generators
of $P_0$, $P_1$ and $P_2$, respectively. The maps are
\[
 \mu(\eta_v)=e_v,\qquad
 d_1(\eta_a)=a\eta_{s(a)}-\eta_{t(a)}a,
\]
and
\[
\begin{aligned}
 d_2(\theta_v)
 ={}&\sum_{t(a)=v}
 A_{v,<a}(\eta_a a^*+a\eta_{a^*})A_{v,>a}\\
 &-q_v\sum_{s(a)=v}
 B_{v,<a}(\eta_{a^*}a+a^*\eta_a)B_{v,>a}.
\end{aligned}
\]
Here the sums run over $a\in Q_1$, and $A_{v,<a}$ and $A_{v,>a}$
are the products before and after the factor indexed by $a$ in $A_v$,
and similarly for $B_v$; empty products are $e_v$.
Let $\partial$ denote formal differentiation in the arrows.
Since $A_v-q_vB_v=0$ in $A$,
\[
 d_2(\theta_vB_v^{-1})
 =\partial(A_v-q_vB_v)B_v^{-1}
 =\partial(A_vB_v^{-1}-q_ve_v).
\]
The last expression is the differential $P_2\to P_1$ in
\cite[Lemma~3.1]{CBShaw}. Thus the change of generators
$\theta_v\mapsto\theta_vB_v^{-1}$, with $P_0$ and $P_1$ unchanged,
identifies the displayed complex with theirs.
The complex is exact except possibly at $P_2$
\cite[Lemma~3.1]{CBShaw}. It remains to prove that $d_2$ is injective.

Let $\mathcal B=\mathcal B_Q^{\boldsymbol q}$ be the dg algebra of
\Cref{def:plumbing-dga}. By \Cref{lem:factors-invertible},
$H^0(\mathcal B)=A$, and \Cref{cor:affine-formality} makes the
projection $\mathcal B\to A$ a quasi-isomorphism. Thus
$A\otimes_{\mathcal B}^{\mathbb L}A\simeq A$.
Derived base change of the standard diagonal resolution of the
semifree dg algebra $\mathcal B$, constructed using universal
derivations, gives precisely the displayed complex: the degree-zero
arrows contribute $P_1$, and the degree-$-1$ generators $\tau_v$
contribute $P_2$, with the differential $d_2$ obtained by differentiating
$d\tau_v=A_v-q_vB_v$. Hence this complex is a projective bimodule
resolution of $A$.

Its self-duality \cite[Theorem~3.17 and Corollary~3.19]{KS} gives
\[
 \operatorname{RHom}_{A\otimes A^{\mathrm{op}}}
 (A,A\otimes A^{\mathrm{op}})\simeq A[-2].
\]
As a resolution of right $A$-modules, the complex splits; tensoring
with any left $A$-module therefore gives a projective resolution of
length at most two. The right-module argument is identical.
For a finite-dimensional left $A$-module $V$, put
$V^\vee=\operatorname{Hom}_{\kk}(V,\kk)$.
Applying $\operatorname{Hom}_A(-,A)$ to this resolution of $V$ gives
\[
 \operatorname{RHom}_A(V,A)
 \simeq V^\vee\otimes_A^{\mathbb L}
 \operatorname{RHom}_{A\otimes A^{\mathrm{op}}}
 (A,A\otimes A^{\mathrm{op}})\simeq V^\vee[-2],
\]
which proves \Cref{eq:star-finite-duality}. The right-module case
follows by replacing $A$ with $A^{\mathrm{op}}$.
\end{proof}

\subsection{Noncommutative crepant resolutions}
\label{subsec:surface}

For a normal Gorenstein domain $R$, a noncommutative crepant
resolution (NCCR) is an algebra $\operatorname{End}_R(M)$, where $M$
is a finitely generated reflexive $R$-module and
$\operatorname{End}_R(M)$ is maximal Cohen-Macaulay over $R$ and has
finite global dimension
\cite[Definition~4.1 and Lemma~4.2]{VdBNCCR}. We construct such
resolutions for the surface algebras arising above.

The following lemma identifies $A$ with an endomorphism ring over a
corner algebra, providing the starting point for the NCCR construction
in \Cref{prop:affine-de-nccr}.

\begin{lemma}\label{lem:finite-quotient-corner}
Let $\kk$ be any field, let $Q$ be an affine Dynkin quiver of type $D$
or $E$, and set $A=\Lambda^{\boldsymbol q}(Q)$ for
$\boldsymbol q\in(\kk^\times)^{Q_0}$. Let $f$ be a sum of vertex
idempotents such that $A/AfA$ is finite-dimensional, and put $T=fAf$.
Then $Af$ and $fA$ are finitely generated as right and left
$T$-modules, respectively, and the natural maps give isomorphisms
\[
 A\cong\operatorname{End}_{T^{\mathrm{op}}}(Af),\qquad
 A^{\mathrm{op}}\cong\operatorname{End}_T(fA),\qquad
 Z(A)\cong Z(T).
\]
\end{lemma}

\begin{proof}
Put $B=A/AfA$ and $I=AfA$.
By \Cref{lem:star-cy}, $\operatorname{Ext}_A^i(B,A)=0$ for $i=0,1$
on both sides. Since $A$ is finitely generated as a $\kk$-algebra and $B$ is
finite-dimensional, $B$ is finitely presented as both a left and a right
$A$-module. Thus $I$ is finitely
generated on both sides; expressing its generators as sums $afb$
shows that $Af$ and $fA$ are finite over $T$.
The two endomorphism isomorphisms follow from
\cite[Proposition~2.9]{BuchweitzMorita}, applied to $A$ and
$A^{\mathrm{op}}$. The centre identification follows from
\cite[Corollary~5.6]{BuchweitzMorita}.
\end{proof}

\begin{lemma}\label{lem:classical-corner-domain}
Let $Q$ be an affine Dynkin quiver of type $D$ or $E$, let $\delta$
be its minimal positive imaginary root, and let $e$ be an extending
vertex idempotent. Over any field $\kk$, the corner
$e\Lambda^{\boldsymbol q}(Q)e$ is a commutative domain whenever
$\boldsymbol q\in(\kk^\times)^{Q_0}$ satisfies $\boldsymbol q^\delta=1$.
\end{lemma}

\begin{proof}
Put $A=\Lambda^{\boldsymbol q}(Q)$ and $R=eAe$.
We will use \Cref{lem:finite-quotient-corner} to compare the centres
of several corners $fAf$. Its finiteness hypothesis holds because
deleting the vertices selected by $f$ leaves finite Dynkin diagrams
\cite[Theorem~3.1.1]{ShawThesis}.

For $Q=\widetilde D_{n+4}$, use the notation of
\Cref{prop:d-algebra} and Fig.~\ref{fig:affine-d-quivers} and take $e=e_1$ at
$\ell_1$. Deleting the four leaves leaves the chain
$c_0-\cdots-c_n$, a single vertex when $n=0$.
We will prove that $R$ embeds in a Laurent polynomial ring.
We first work with the corner $B=fAf$, where
$f=\sum_{j=0}^n e_{c_j}$: the four leaves are represented there by
loops satisfying quadratic relations. These relations yield a central
element $h$, and a normal-form calculation proves that $A$, hence $R$,
is $\kk[h]$-torsion-free. After a quadratic extension $L/\kk(h)$,
the arrows between consecutive chain vertices become invertible,
reducing the calculation to
\[
 C=e_{c_0}(A\otimes_{\kk[h]}L)e_{c_0}
   \cong\operatorname{Mat}_2(L[w^{\pm1}]).
\]
The extending leaf corresponds to an idempotent $p\in C$, giving
\[
 R\lhook\joinrel\longrightarrow R\otimes_{\kk[h]}L
   \cong pCp\cong L[w^{\pm1}].
\]
Thus $R$ is a commutative domain.

\begin{enumerate}
\item We describe the chain corner and construct $h$.
In a vertex corner $e_vBe_v$, we write $1$ for its identity $e_v$.
The algebra $B$ is generated by $b_j,b_j^*$ and the four loops
$U_i=1+u_i$, where $u_i=a_ia_i^*$,
with
\[
 U_i^2=(1+\lambda_i^{-1})U_i-\lambda_i^{-1}
\]
and the vertex product relations of \Cref{prop:d-algebra}.
Let $\widetilde B$ be the algebra defined by these relations and let
$M$ be a left $\widetilde B$-module, with $M_{c_j}=e_{c_j}M$.
For each leaf $\ell_i$ adjacent to $c_j$, assign the space
$M_{\ell_i}=\operatorname{im}(u_i|_{M_{c_j}})$ and the maps
\[
 a_i:M_{\ell_i}\hookrightarrow M_{c_j},\qquad
 a_i^*:M_{c_j}\longrightarrow M_{\ell_i},\quad m\longmapsto u_i m.
\]
Since $u_i^2=(\lambda_i^{-1}-1)u_i$, these satisfy
$a_i a_i^*=u_i$ and
$a_i^*a_i=(\lambda_i^{-1}-1)\operatorname{id}_{M_{\ell_i}}$.
Thus $M$ extends to an $A$-module. Taking the faithful left regular
module $M=\widetilde B$ shows that the natural surjection
$\widetilde B\to B$ is injective, so there are no further relations.

We seek an operator commuting with $b_j,b_j^*$.
The vertex relations and the identity
\[
 (1+b_jb_j^*)^{-1}b_j=b_j(1+b_j^*b_j)^{-1}
\]
lead to the following choice.
Put $s_0=1$ and $s_j=\prod_{i=1}^j\nu_i$ for $1\leq j\leq n$, and set
\[
 T_{c_0}=\nu_0(U_1U_2)^{-1},\qquad
 T_{c_j}=s_{j-1}^{-1}(1+b_jb_j^*)^{-1}\quad(1\leq j\leq n).
\]
A short computation shows that $T=\sum_jT_{c_j}$ commutes with $b_j,b_j^*$.
Set $U=U_1$ and $V=U_3$. Eliminating
$U_2=\nu_0U^{-1}T_{c_0}^{-1}$ and
$U_4=s_nV^{-1}T_{c_n}$ gives the relations
\[
 b_j^*b_j=(s_{j-1}T_{c_{j-1}})^{-1}-e_{c_{j-1}},\qquad
 b_jb_j^*=(s_{j-1}T_{c_j})^{-1}-e_{c_j},
\]
together with the relations at $c_0$ and $c_n$ below.
Since $\boldsymbol q^\delta=1$,
\[
 D_0:=\nu_0^2\lambda_1\lambda_2
       =\frac{1}{s_n^2\lambda_3\lambda_4}.
\]
Put $c_L=\nu_0(1+\lambda_2)$ and
$c_R=(1+\lambda_4)/(s_n\lambda_3\lambda_4)$.
The relations at $c_0$ and $c_n$ are
\[
\begin{aligned}
 U^2&=(1+\lambda_1^{-1})U-\lambda_1^{-1}e_{c_0},\qquad
 V^2=(1+\lambda_3^{-1})V-\lambda_3^{-1}e_{c_n},\\
 UT_{c_0}&=D_0T_{c_0}^{-1}U+c_Le_{c_0}
                 -(1+\lambda_1^{-1})D_0T_{c_0}^{-1},\\
 VT_{c_n}&=D_0T_{c_n}^{-1}V+(1+\lambda_3^{-1})T_{c_n}-c_Re_{c_n}.
\end{aligned}
\]
Thus $h=T+D_0/T$ is central in $B$ and extends to $Z(A)$ by
\Cref{lem:finite-quotient-corner}.

\item We find a $\kk[h]$-basis of $B$ to prove that localisation in $h$
is injective. Since $h$ is central, we can work over $\kk[h]$ and use
\[
 T^2=hT-D_0f,\qquad T^{-1}=\frac{hf-T}{D_0}.
\]
Let $\mathcal P$ be the composable paths in $b_j,b_j^*,U,V$,
including the vertex idempotents, containing none of
$U^2,V^2,b_jb_j^*,b_j^*b_j$ as subpaths.
Replace $T^{-1}$ by the displayed expression, move $T$ to the left,
and reduce $T^2$ and these four pairs. Each reduction shortens a word
or moves $T$ left, so the reductions terminate.

For $W=U$ at $v=c_0$, or $W=V$ at $v=c_n$, the relations have the form
\[
 W^2=aW-be_v,\qquad WT=(h-T)W+(aT+c)e_v,
 \qquad a,b\in\kk,\quad c\in\kk[h].
\]
The overlaps $W^2T,WT^2,W^3,T^3,b_jb_j^*b_j,b_j^*b_jb_j^*$,
and those involving $T$ and the arrows $b_j,b_j^*$, all resolve
using these relations and the commutation of $T$ with $b_j,b_j^*$.

The Diamond Lemma \cite[Theorem~1.2]{BergmanDiamond} therefore gives
\[
 B=\bigoplus_{\gamma\in\mathcal P}\bigl(\kk[h]\gamma\oplus\kk[h]T\gamma\bigr)
\]
as a $\kk[h]$-module.
Let $I=AfA$. Since $A/I$ is finite-dimensional,
\Cref{lem:star-cy} implies that both annihilators of $I$ vanish.
If $g(h)x=0$ for $0\ne g\in\kk[h]$ and $x\in A$, freeness of $B$
gives $IxI=0$, hence $xI=0$ and $x=0$.
Thus $A$, and in particular $R$, is $\kk[h]$-torsion-free.

\item We compute the chain corner after adjoining an eigenvalue of $T$.
Since $T^2-hT+D_0f=0$, adjoin to $\kk(h)$ an element $\zeta$
satisfying $\zeta^2-h\zeta+D_0=0$ and set
\[
 L=\kk(h)(\zeta)=\kk(\zeta),\qquad
 \zeta_+=\zeta,\quad\zeta_-=D_0/\zeta.
\]
The equality of fields follows from $h=\zeta+D_0/\zeta$.
The eigenvalues $\zeta_\pm$ are distinct in any characteristic.
The factors $(s_{j-1}T)^{-1}-1$ have nonzero eigenvalues
$(s_{j-1}\zeta_\pm)^{-1}-1$, so $b_j,b_j^*$ become invertible
between their vertex idempotents.
The corner $C=e_{c_0}(B\otimes_{\kk[h]}L)e_{c_0}$ is generated by
$T_{c_0}$, $U$ and $(b_n\cdots b_1)^{-1}V(b_n\cdots b_1)$.
For the following matrix calculation, denote these generators by
$T,U,V$, respectively. The relations at $c_0$ and $c_n$ give the diagonal entries
of $U,V$ as
\[
 u_\pm=\frac{c_L-(1+\lambda_1^{-1})\zeta_\mp}
                   {\zeta_\pm-\zeta_\mp},\qquad
 v_\pm=\frac{(1+\lambda_3^{-1})\zeta_\pm-c_R}
                   {\zeta_\pm-\zeta_\mp}.
\]
The quadratic relations give the off-diagonal products
$\rho_U=u_+u_--\lambda_1^{-1}$ and
$\rho_V=v_+v_--\lambda_3^{-1}$.
They are nonzero: over the denominator $h^2-4D_0$, their numerators
have leading terms $-\lambda_1^{-1}h^2$ and $-\lambda_3^{-1}h^2$.
If $p_\pm$ are the spectral idempotents of $T$, then
$p_+Up_-$ and $\rho_U^{-1}p_-Up_+$ give matrix units. In these units,
\[
 T=\begin{pmatrix}\zeta_+&0\\0&\zeta_-\end{pmatrix},\qquad
 U=\begin{pmatrix}u_+&1\\\rho_U&u_-\end{pmatrix},\qquad
 V=\begin{pmatrix}v_+&w\\\rho_V/w&v_-\end{pmatrix}.
\]
There is no further relation between $U$ and $V$, so $w$ is a free
invertible variable. These matrices and, conversely, the spectral
idempotents and off-diagonal entries give inverse maps
\[
 C\cong\operatorname{Mat}_2(L[w^{\pm1}]).
\]

\item We recover the extending corner as a rank-one matrix corner.
Put
\[
 X=\frac{h-\nu_0(\lambda_1+\lambda_2)}
           {\nu_0\lambda_1\lambda_2}\in\kk[h]\setminus\{0\}.
\]
Expanding $he_{c_0}=\nu_0((U_1U_2)^{-1}+\lambda_1\lambda_2U_1U_2)$
and using $u_1a_1=(\lambda_1^{-1}-1)a_1$, we obtain
\[
 ha_1=\nu_0(\lambda_1+\lambda_2)a_1
       +\nu_0\lambda_1\lambda_2u_1u_2a_1.
\]
Since $u_1=a_1a_1^*$, the definition of $X$ gives
$a_1(a_1^*u_2a_1-Xe)=0$.
Every path from $\ell_1$ to the chain starts with $a_1$, so
$I(a_1^*u_2a_1-Xe)=0$. Since $I$ has zero annihilators,
$a_1^*u_2a_1=Xe$.
Over $L$, the maps $a_1$ and $X^{-1}a_1^*u_2$ identify the leaf with the
idempotent $p=u_1u_2/X\in\operatorname{Mat}_2(L)$.
Its image equals that of $U-1$: $p(U-1)=U-1$ and
$\operatorname{im}p\subseteq\operatorname{im}(U-1)$.
This image has dimension one, since $U-1$ has determinant zero
and a nonzero off-diagonal entry. Therefore
\[
 R\lhook\joinrel\longrightarrow R\otimes_{\kk[h]}L
   \cong pCp\cong L[w^{\pm1}],
\]
which proves that $R$ is a commutative domain.
\end{enumerate}

For types $E$, put $H=e_0Ae_0$, where $e_0$ is the central vertex
idempotent, and let $(d_1,d_2,d_3)=(3,3,3),(2,4,4),(2,3,6)$ for
$E_6,E_7,E_8$, with $\ell=d_3$. We first prove $R\cong Z(H)$, then
embed $\operatorname{gr}Z(H)$ in a polynomial ring. We repeat the
matrix argument of \cite[Theorem~6.5 and Proposition~6.6]{EOR},
checking its presentation and identities integrally. It suffices to
work after a field extension.

\begin{enumerate}
\item We prove that $R$ is commutative. Over the Laurent polynomial
ring $S=\ZZ[q_i^{\pm1}]/(\boldsymbol q^\delta-1)$ (since $\delta_e=1$),
the tree dg algebra of \Cref{def:plumbing-dga} is a bounded-above
complex of free $S$-modules, with residue-field cohomology only in
degree zero by \Cref{cor:affine-formality}. The residue-field
criterion for flat dimension \cite[\S4, (4.1)]{CIMRigidity} makes
its degree-zero cohomology, the universal algebra $A_S$ by
\Cref{lem:factors-invertible}, $S$-flat. Thus $eA_Se$ embeds in its
generic fibre. Embed the fraction field of $S$ in $\C$. At this
generic parameter, $A$ is Azumaya by
\cite[Corollary~11.5(iii)]{EOR}, and $e$ has rank $\delta_e=1$ in
every irreducible representation (see its proof). Hence the generic
corner is commutative, so all commutators in $eA_Se$ vanish.
Specialisation gives commutativity of $R$, and
\Cref{lem:finite-quotient-corner} gives $R\cong Z(A)\cong Z(H)$.

\item We normalise and filter $H$. For the roots $\alpha_{ja}$ of
\Cref{eq:star-central-corner}, choose $\kappa_j$ with
$\kappa_j^{d_j}=(-1)^{d_j-1}\prod_a\alpha_{ja}$ and
$\prod_j\kappa_j=q_0$. This is possible after a field extension:
$\boldsymbol q^\delta=1$ gives $(\prod_j\kappa_j)^\ell=q_0^\ell$,
so $\kappa_3$ can be adjusted. Then $T_j=\kappa_j^{-1}U_j$ satisfy
$T_1T_2T_3=1$ and $p_j(T_j)=0$, where $p_j$ is monic of degree $d_j$
with constant term $-1$. Let $S'$ be the polynomial ring over $\ZZ$
in the remaining coefficients of the $p_j$, regarded as independent
variables, and let $H_{S'}$ have these relations. Since the $T_j^{-1}$
are polynomials in $T_j$,
$H_{S'}$ is generated by $T_2$ and $c=T_3$. Filter it by the number
of letters $T_2$, and put $d=T_2c^{-1}$ for $E_6,E_7$, and
$d=T_2c^{-2}$ for $E_8$. Using $c,d$ also for their leading terms,
we claim that $\operatorname{gr}H_{S'}$ has defining relations
\[
 p_3(c)=0,\qquad dc^jd=0\ (2\leq j\leq\ell-2),\qquad
 dcd=cdc^{\ell-1}dc,\qquad dcdcd=0.
\]
First we check that these relations hold, writing $F_n$ for the
$n$th filtration term. For $E_6$, the cubic relations of $T_2$ and
$T_2c$ give $T_2^3\in F_2$ and
$T_2cT_2c-c^{-1}T_2^2\in F_1$.
For $E_7$, write $(T_2c)^2-\eta T_2c-1=0$. Then
$T_2cT_2=\eta T_2+c^{-1}$ and $T_2^{-1}=cT_2c-\eta c$; the
quartic relation gives $T_2^2-T_2^{-2}\in F_1$, hence
$T_2^2-cT_2c^2T_2c\in F_1$.
For $E_8$, write $T_2^3-f_1T_2^2+f_2T_2-1=0$ and
$(T_2c)^2-\eta T_2c-1=0$. Then
$T_2cT_2=\eta T_2+c^{-1}$ and
$T_2^2=cT_2c+f_1T_2-\eta c-f_2$; expanding $T_2^2cT_2$ in two ways gives
\[
 T_2c^2T_2\equiv c^{-1}T_2c^{-1}
   +\eta(cT_2+T_2c)+f_2T_2\pmod{F_0}.
\]
Thus $T_2^2,T_2cT_2,T_2c^2T_2\in F_1$, giving $dc^jd=0$.
Expanding $T_2^2c^2T_2$ in two ways gives
$cT_2c^3T_2-T_2c^{-1}T_2c^{-1}\in F_1$, giving $dcd=cdc^5dc$.
For $\ell=4,6$, multiplying $dcd=cdc^{\ell-1}dc$ by $cd$ gives
$dcdcd=0$ since $dc^2d=0$.

\item We prove that these are defining relations, with the
following $S'$-basis in degree $n$:
\[
 c^i\ (n=0);\qquad c^id^nc^j,\quad c^id^rcd^{n-r}c^j\ (1\leq r<n)
 \ (n\geq1);\qquad 0\leq i,j<\ell.
\]
Using only the displayed relations, induction on $m\geq1$ gives
\[
 dcd^mcd=0,\qquad
 d^mc^{\ell-1}dc=c^{-1}dcd^m+p_3'(0)\sum_{u=2}^md^ucd^{m+1-u}.
\]
For the second identity, expand $c^{-1}$ using $p_3(c)=0$; the
intermediate terms vanish by $dc^jd=0$ and the first identity.
These identities make the span of the proposed basis closed under
right multiplication by $c,d$, so the words span modulo the
displayed relations. Over the fraction field of $S'$,
\cite[Proposition~6.2]{EOR} gives
$\dim\operatorname{gr}_0H=\ell$ and
$\dim\operatorname{gr}_nH=n\ell^2$ for $n\geq1$, exactly the numbers
of these words. They are therefore independent over that field,
hence over $S'$, proving the presentation and basis.
Their lifts of degree at most $n$ form a basis of $F_nH_{S'}$, so forming
$\operatorname{gr}$ commutes with specialisation. Thus the same
presentation and basis hold over any field.

\item We embed $\operatorname{gr}H$ in a matrix algebra. Let $M$
be the companion matrix of $p_3$, with $Me_i=e_{i+1}$ for $i<\ell$,
and put $P=E_{11}$. Let $x,t$ be independent variables.
The matrix $M$ realizes the relation $p_3(c)=0$; the variable $t$
records degree, while the dependence on $x$ will distinguish the
images of the basis words. Set
\[
 D=(1+x)P-xMPM^{-1},\qquad c\longmapsto M,\qquad d\longmapsto tD
 \quad\text{in }\operatorname{Mat}_\ell(\kk[x,t]).
\]
Since $PM^jP=0$ for $1\leq j<\ell$ and $PM^\ell P=P$, multiplication gives
$DM^jD=0$ for $2\leq j\leq\ell-2$,
$DMD=MDM^{\ell-1}DM=-x(1+x)MP$ and $DMDMD=0$.
Thus the defining relations hold, and
\[
 c^id^rcd^{n-r}c^j\longmapsto
 (-1)^rt^nx^r(1+x)^{n-r}M^{i+1}PM^j.
\]
The matrices $M^iPM^j$, $0\leq i,j<\ell$, form a basis of
$\operatorname{Mat}_\ell(\kk)$, since $e_1$ and its dual row are
cyclic for $M$. In degree zero, injectivity follows from the
minimal polynomial $p_3$ of $M$. In degree $n\geq1$, evaluation at
$x=0$ sends the images of $c^id^nc^j$ to $t^nM^iPM^j$ and kills
the mixed words. Those mixed images are independent because the
polynomials $x^r(1+x)^{n-r}$ have distinct lowest powers of $x$.
Thus the map is injective. Its degree-two image contains
$t^2x(1+x)\operatorname{Mat}_\ell(\kk)$, so central elements map to
scalar matrices. The induced filtration on $Z(H)$ therefore gives
\[
 \operatorname{gr}Z(H)\lhook\joinrel\longrightarrow
 Z(\operatorname{gr}H)\lhook\joinrel\longrightarrow\kk[x,t].
\]
Consequently $\operatorname{gr}Z(H)$ is a domain, hence so is
$Z(H)\cong R$.
\end{enumerate}
\end{proof}

The following theorem extends the NCCR and Satake assertions of
\cite[Conjecture~1.4]{KS} from $\boldsymbol q=\boldsymbol1$ to all
classical parameters. Its homological input is \Cref{lem:star-cy},
which follows from \Cref{cor:affine-formality}.

\begin{theorem}\label{prop:affine-de-nccr}
Let $\kk$ be any field, let $Q$ be an affine Dynkin quiver of type $D$
or $E$, and let $\delta$ be its minimal positive imaginary root.
Suppose $\boldsymbol q\in(\kk^\times)^{Q_0}$ satisfies
$\boldsymbol q^\delta=\prod_iq_i^{\delta_i}=1$.
Set $A=\Lambda^{\boldsymbol q}(Q)$, let $e$ be an extending vertex
idempotent, and put $R=eAe$. Then $R$ is a normal Gorenstein affine
surface domain,
\[
 A\cong\operatorname{End}_R(Ae)
\]
is a noncommutative crepant resolution of $R$, and the Satake map
$Z(A)\longrightarrow R$, $z\longmapsto ez$, is an isomorphism.
\end{theorem}

\begin{proof}
By \Cref{lem:classical-corner-domain}, $R$ is a commutative domain.
\Cref{lem:finite-quotient-corner} gives $R\cong Z(A)$ and finite
$R$-modules $Ae,eA$. The multiplication map
$Ae\otimes_R eA\to AeA$ and the finite-dimensional quotient
$A/AeA$ show that $A$ is finite over $R$.
The Artin-Tate lemma \cite[\S3.7, Theorem~7.1]{SmithNAG} makes
$R$ affine Noetherian, so $A$ is Noetherian and its simple modules
are finite-dimensional. By \Cref{lem:star-cy}, every simple left $A$-module
$V$ satisfies $\operatorname{Ext}^2_A(V,A)\ne0$ and
$\operatorname{pd}_A V=2$, and the same holds on the right.
Hence $A$ is homologically homogeneous,
and \cite[Theorem~5.3]{BrownHajarnavis} decomposes it into finitely
many prime rings. Since $Z(A)=R$ is a domain, there is only one
factor, so $A$ is prime.
By \cite[Theorems~2.5 and~6.1]{BrownHajarnavis},
$R$ is a normal surface and $A$ is maximal Cohen-Macaulay over $R$.
By \cite[Theorem~4.1(i)]{PresslandInternal}, $R$ is Gorenstein.
By \Cref{lem:finite-quotient-corner},
$A\cong\operatorname{End}_R(Ae)$, and the finite $R$-module $Ae$
is reflexive as a direct summand of $A$. This is an NCCR of $R$.
\end{proof}

We now specialise to $\boldsymbol q=\boldsymbol1$.

Let $Q$ be $\widetilde D_n$ ($n\ge4$), $\widetilde E_6$, $\widetilde E_7$
or $\widetilde E_8$, put $A=\Lambda^{\boldsymbol1}(Q)$, and choose the
idempotent $e$ at an extending vertex.
Shaw's corner calculation \cite[Theorem~4.1.1]{ShawThesis}\footnote{In type
$D$, the term $-p_{n-4}(X)XZ$ in \cite[Theorem~4.1.1]{ShawThesis}
and \cite[Theorem~6.4]{KS} should be $+p_{n-4}(X)XZ$.
The sign error occurs in the passage from the second to the third line
of the display for $T$ in the proof of
\cite[Lemma~4.7.5, p.~95]{ShawThesis}, where the term
$a^*BABF^kBNCMa$ should carry a minus sign.
For example, in $\widetilde D_5$ over $\mathbb Q$, take the representation
\[
\begin{gathered}
\begin{tikzpicture}[
  >={Stealth[length=3pt,width=2.5pt]},
  every node/.style={font=\footnotesize},
  v/.style={inner sep=1pt,outer sep=1pt},
  arr/.style={->,line width=.35pt,shorten >=1.5pt,shorten <=1.5pt}]
 \node[v] (la) at (-2.6,.95) {$\mathbb Q$};
 \node[v] (lb) at (-2.6,-.95) {$\mathbb Q$};
 \node[v] (l) at (-1.1,0) {$\mathbb Q^2$};
 \node[v] (r) at (1.1,0) {$\mathbb Q^2$};
 \node[v] (rc) at (2.6,.95) {$\mathbb Q$};
 \node[v] (rd) at (2.6,-.95) {$\mathbb Q$};
 \foreach \u/\v/\lab in {la/l/a,lb/l/b,rc/r/c,rd/r/d,r/l/f}{
   \draw[arr] (\u) to[bend left=12] node[auto,inner sep=1pt] {$\lab$} (\v);
   \draw[arr] (\v) to[bend left=12] node[auto,inner sep=1pt] {$\lab^*$} (\u);
 }
\end{tikzpicture}\\[2pt]
 a=c=\binom10,\quad b=\binom01,\quad d=\binom11,\qquad
 a^*=(0\;1),\quad b^*=(1\;0),\quad c^*=(0\;-1),\quad d^*=(-1\;1),\\[2pt]
 f=I,\qquad f^*=\begin{pmatrix}0&-1\\-1&1\end{pmatrix}.
\end{gathered}
\]
Put $A=aa^*$, $B=bb^*$, $C=cc^*$, $D=dd^*$ and $F=ff^*=f^*f$.
The vertex relations hold since
$A^2=B^2=C^2=D^2=0$ and
\[
 (I+A)(I+B)(I+F)=I,\qquad (I+C)(I+D)=I+F.
\]
At the source of $a$, the corner
generators $X=a^*Ba$, $Y=a^*fCf^*a$ and $Z=a^*BfCf^*a$
act by $(X,Y,Z)=(1,0,1)$. Since $p_0=0$ and $p_1=-X$, the printed
relation $Z^2+X^2Z-XY^2-XYZ$ evaluates to $2$, whereas the corrected
relation evaluates to $0$.}, also
recorded in \cite[Theorem~6.4]{KS}, identifies $R=eAe$ with
$\kk[x,y,z]/(f_Q)$ for the following polynomials:
\[
\begin{array}{c|c|c}
 Q&f_Q&\text{weights of }(x,y,z)\\ \hline
 \widetilde D_n&xyz+xy^2-p_{n-5}x^2y-p_{n-4}xz-z^2&(2,n-2,n-1)\\
 \widetilde E_6&xyz+z^2+x^2z+y^3&(3,4,6)\\
 \widetilde E_7&xyz+z^2+y^3+x^3y&(4,6,9)\\
 \widetilde E_8&xyz+z^2+y^3+x^5&(6,10,15).
\end{array}
\]
Here $p_i=p_i(x)$ is defined by $p_{-1}=-1$, $p_0=0$ and
$p_{i+1}=x(p_{i-1}+p_i)$. The weights in the table are those of
the du Val initial forms.

\begin{lemma}\label{lem:shaw-du-val}
Over any field, $S_Q$ has precisely one geometric singular point,
the origin, of du Val type $D_n$, $E_6$, $E_7$ or $E_8$, respectively.
\end{lemma}

\begin{proof}
Work over $\overline\kk$. We first determine the singular locus.
In types $E$, write $g_Q=f_Q-xyz$. The weight of $xyz$ exceeds
that of $g_Q$ by one, so the weighted Euler identity forces $xyz=0$
at a singular point of $f_Q=0$. Substitution in the equation and
its three partial derivatives leaves only the origin, in every
characteristic.

In type $D$, put $m=n-4$. The recurrence gives
\begin{equation}\label{eq:shaw-cassini}
 p_m^2-xp_mp_{m-1}-xp_{m-1}^2=(-x)^{m+1}.
\end{equation}
If $\operatorname{char}\kk=2$, a singular point with $x\ne0$
satisfies $y=p_m$ and $z=xp_{m-1}$, by the partial derivatives
in $z$ and $y$. Substitution and \Cref{eq:shaw-cassini} give
$f_Q=x^{m+2}\ne0$, a contradiction. If $x=0$, the equations
$f_Q=\partial_xf_Q=0$ give $z=y=0$.

If $\operatorname{char}\kk\ne2$, set
$w=z-x(y-p_m)/2$. Then
\[
 f_Q=-w^2+\frac{x}{4}E,\qquad
 E=(x+4)y^2-2x(p_m+2p_{m-1})y+xp_m^2.
\]
By \Cref{eq:shaw-cassini},
\[
 \operatorname{disc}_y(E)=16(-1)^m x^{m+2}.
\]
A singular point with $x\ne0$ would satisfy $w=E=\partial_yE=0$,
contradicting this identity. If $x=0$, the equations
$f_Q=\partial_xf_Q=0$ again give $z=y=0$.

We now identify the singularity at the origin. In type $D$, suppose
first that $\operatorname{char}\kk\ne2$. In the completed local ring, set
\[
 W=\sqrt{x+4}\left(z-\frac{x(y-p_m)}2\right),\qquad
 V=\frac{(x+4)y-x(p_m+2p_{m-1})}{2}.
\]
The square root is the formal power series with constant term $2$.
By \Cref{eq:shaw-cassini},
\[
 (x+4)f_Q=-W^2+xV^2+(-1)^{m+1}x^{n-1}.
\]
Since $x+4$ is a unit and the change of coordinates is invertible,
this is the du Val form of type $D_n$.

In types $E$, first suppose that the characteristic is different
from $2,3$ for $E_6,E_7$, and from $2,3,5$ for $E_8$.
The weighted initial form $g_Q$ is obtained by omitting $xyz$.
These are the corresponding du Val forms; for $E_6$, completing
the square gives $z'^2-x^4/4+y^3$.
The weighted Tjurina algebra
\[
 \kk[x,y,z]/(g_Q,\partial_xg_Q,\partial_yg_Q,\partial_zg_Q)
\]
has top nonzero weight $10<12$, $16<18$, or $28<30$,
respectively, strictly below the weight of $g_Q$.
At each successive weight, the error lies in the ideal generated
by $g_Q$ and its partial derivatives. Multiplication by a unit and a
coordinate change remove that weighted piece; the corrections have
strictly higher weight than the terms they modify, so iteration
converges in the completed local ring. Hence the origin has the asserted
du Val type.

In the remaining characteristics, we compare with Artin's normal
forms, listed in \cite[Section~3, pp.~15--17]{ArtinRDP}.
For type $D$ in characteristic two, put $m=n-4$ and substitute
$y=Y+p_m$, $z=Z+xp_{m-1}$. \Cref{eq:shaw-cassini} gives
\[
 f_Q=Z^2+xYZ+xY^2+x^{n-2}.
\]
For $n=2r$, set $Z=W+x^{r-1}$; for $n=2r+1$, set
$Y=V+x^{r-1}$. The resulting equations are the listed $D_n$
forms after interchanging the first two coordinates.
In characteristic two, the three $E$ equations are already the
listed forms after interchanging $x,y$. In characteristic three,
completing the square and interchanging $x,y$ gives the listed
$E_7$ and $E_8$ forms, up to rescaling coordinates by nonzero constants.

For $E_6$ in characteristic three, completing the square and then
setting $x=X+y$ gives
\[
 z'^2+y^3(1-y)-X^4-X^2y^2.
\]
The substitution
\[
 (y,X,z')=\left(\frac{u}{1+u},\frac{v}{1+u},
                    \frac{w}{(1+u)^2}\right)
\]
turns this, up to a unit, into $w^2+u^3-v^4-u^2v^2$.
Rescaling $u,w$ gives the listed $E_6$ form.
For $E_8$ in characteristic five, completing the square and setting
$y=U+3x^2$ gives
\[
 z'^2+U^3+x^5(1+x)+3Ux^4.
\]
The substitution
\[
 (x,U,z')=\left(\frac{v}{1-v},\frac{u}{(1-v)^2},
                    \frac{w}{(1-v)^3}\right)
\]
gives $w^2+u^3+v^5+3uv^4$ up to a unit. Rescaling coordinates
gives the listed $E_8$ form. All these changes are invertible in
the completed local ring, so the asserted geometric types follow.
\end{proof}

\begin{remark}
For the four stars, the generalised DAHA $H=e_0Ae_0$ at the unipotent
point also has centre $R$: it is the endomorphism ring of the nonzero
reflexive $R$-module $e_0Ae$, so normality gives $Z(H)=R$.
However, $e_0$ is not full, since it annihilates every arm vertex-simple
module.
\end{remark}

The following type $A$ model supplies the local charts for types
$III$ and $IV$.

\begin{lemma}\label{lem:cyclic-additive-order}
Over any field, the ordinary preprojective algebra
$\Pi(\widetilde A_{n-1})$, for $n=2,3$, is a noncommutative crepant
resolution of
\[
 R_n=\kk[Y,Z,t]/(YZ-t^n).
\]
For each vertex idempotent $e$, restriction and left multiplication give
isomorphisms
\[
\begin{gathered}
 Z\bigl(\Pi(\widetilde A_{n-1})\bigr)
 \xrightarrow{\sim}e\Pi(\widetilde A_{n-1})e\cong R_n,\\
 \Pi(\widetilde A_{n-1})\xrightarrow{\sim}
 \operatorname{End}_{R_n}\bigl(\Pi(\widetilde A_{n-1})e\bigr).
\end{gathered}
\]
\end{lemma}

For $n=2,3$, give $S=\kk[r,s]$ the
$\ZZ/n$-grading with $\deg r=1$, $\deg s=-1$, and write $S_a$
for its graded pieces. Label the arrows of the doubled cycle by
$a_i:i\to i+1$ and $a_i^*:i+1\to i$, with indices in $\ZZ/n$.
The identification
\[
 \Pi(\widetilde A_{n-1})\cong
 \{(m_{ji})\in\operatorname{Mat}_n(S):m_{ji}\in S_{j-i}\}
\]
sends
\[
 e_i\longmapsto E_{ii},\qquad
 a_i\longmapsto rE_{i+1,i},\qquad
 a_i^*\longmapsto sE_{i,i+1}.
\]
Its centre is $S_0=\kk[r^n,s^n,rs]$, with
$Y=r^n$, $Z=s^n$ and $t=rs$.
After extending to an algebraic closure and completing at the origin,
this is the quotient construction for the group scheme $\mu_n$ acting
with weights $(1,-1)$ in \cite[Example~1.1]{VdBNCCR} in characteristic
zero and \cite[Theorems~1.5 and~1.7]{LiedtkeYasuda} in positive
characteristic. Away from the origin, the algebra is locally a matrix
algebra. The assertions descend to $\kk$ by faithful flatness.

For $n=2,3$, we also use the doubled cycle with a central parameter $t$
and relations
\[
 a_j^*a_j=(t-\lambda_j)e_j,\qquad
 a_ja_j^*=(t-\lambda_j)e_{j+1},
\]
where $\lambda_1,\ldots,\lambda_n\in\kk$ and indices are read cyclically.
Cancelling adjacent reversed arrows gives a $\kk[t]$-basis of
idempotents and paths following either orientation: the overlapping
reductions on paths of length three agree. Thus each vertex
corner, and the centre, is
\[
 \kk[t,Y,Z]/\Bigl(YZ-\prod_{j=1}^n(t-\lambda_j)\Bigr),
\]
where $Y,Z$ are the two oriented cycles; the algebra is finite over
this ring. It is again an NCCR. Indeed, near $t=c$, invert the
factors with $\lambda_j\ne c$ and contract the corresponding
invertible arrow pairs, retaining the chosen vertex. If $c$ occurs
$m$ times, the resulting full corner is a localisation of
$\Pi(\widetilde A_{m-1})$
with parameter $t-c$; for $m=1$ this is $\kk[Y,Z]$.
The deleted vertices give repeated vertex-projective summands, so
\Cref{lem:cyclic-additive-order} also gives reflexivity and the
endomorphism presentation for the original algebra. Away from all
roots, it is a matrix algebra over a smooth ring. These local
descriptions give global dimension at most two and the NCCR
assertion, and show that a root of
multiplicity $m\ge2$ gives a split $A_{m-1}$ singularity.

\begin{proposition}\label{prop:nontransverse-centres}
For $F=II,III,IV$ and classical parameters $\boldsymbol q$, let
$A=A_F(\boldsymbol q)$ and let $R_F(\boldsymbol q)$ be the coordinate
ring of the corresponding cubic following \Cref{cor:unipotent-mirror}.
Over any field, this ring is normal and Gorenstein, and the
Satake and left-multiplication maps give
\[
 Z(A)\xrightarrow{\sim}e_1Ae_1=R_F(\boldsymbol q),\qquad
 A\cong\operatorname{End}_{R_F(\boldsymbol q)}(Ae_1).
\]
The algebra $A$ is a noncommutative crepant resolution. Its centre has precisely
the singularities described in the introduction. For type $II$,
$A=R_{II}$ is smooth.
\end{proposition}

\begin{proof}
For $II$, the presentation in \Cref{subsec:nontransverse} gives
$A=R_{II}$. A singular point would satisfy $xy=yz=-1$ and $xz=0$,
which is impossible over any field.

For $III$ and $IV$, we use two localisations: one is a deformed
doubled cycle as above, and the other is a matrix algebra over a
smooth surface.

For $III$, the arrows are $a,c:1\to2$ and $b,d:2\to1$.
Put $q=q_1=q_2^{-1}$, and write $r_1,r_2$ for the two relations
in \Cref{eq:III-algebra}. The identities
\[
\begin{aligned}
 q(abc-cba)&=r_2(a+c+cba)-(a+c+abc)r_1,\\
 dab-bad&=dr_2-r_1d-d(abc-cba)d
\end{aligned}
\]
give $abc=cba$ and $dab=bad$, so $u=ba+ab$ is central.
Put $v=dc$, $w=bc$ in the corner $e_1Ae_1$, and set
$\rho(u)=u(u+1-q)$. We compute this corner and the centre
after inverting $1+u$ and $\rho(u)$.

After inverting $1+u$, replace $c$ by $c'=-a-(1+u)c$.
The relations become
\[
 ba=ue_1,\quad ab=ue_2,\qquad
 dc'=(u+1-q)e_1,\quad c'd=(u+1-q)e_2.
\]
The inverse substitution $c=-(c'+a)/(1+u)$ identifies
$A[(1+u)^{-1}]$ with the doubled-cycle algebra having paired
products $u$ and $u+1-q$, localised at $1+u$.
The two cyclic paths at vertex $1$ are
\[
 Y=da=q-1-u-(1+u)v,\qquad Z=bc'=-u-(1+u)w.
\]
The path calculation above shows that the $e_1$ corner is generated
by $u,Y,Z,(1+u)^{-1}$ with relation $YZ=\rho(u)$, and restriction
$z\mapsto e_1ze_1$ identifies the centre with this corner.
Since
\[
 YZ-\rho(u)=(1+u)\bigl(uv+(u+1-q)w+(1+u)vw\bigr),
\]
both are isomorphic to $R[(1+u)^{-1}]$, where
\begin{equation}\label{eq:III-centre-chart}
 R=\kk[u,v,w]/\bigl(uv+(u+1-q)w+(1+u)vw\bigr).
\end{equation}

After inverting $\rho(u)$, the arrow $a$ has inverse $a^{-1}=u^{-1}b$.
Put $C=e_1A[\rho(u)^{-1}]e_1$ and retain $Y=da$.
The matrix units $e_1,e_2,a,a^{-1}$ give
$A[\rho(u)^{-1}]\cong\operatorname{Mat}_2(C)$, with arrow images
\[
 a\longmapsto E_{21},\qquad b\longmapsto uE_{12},\qquad
 c\longmapsto(w/u)E_{21},\qquad d\longmapsto YE_{12}.
\]
Thus $C$ is generated over $\kk[u,\rho(u)^{-1}]$ by $Y,w$ with relations
\[
 YZ=ZY=\rho(u),\qquad Z=-u-(1+u)w.
\]
Since $Y^{-1}=\rho(u)^{-1}Z$ commutes with $w$, so does $Y$.
The inverse substitutions $v=Yw/u$ and $w=uv/Y$ give
\[
 C\cong\kk[u,\rho(u)^{-1},v,Y^{-1}]\cong R[\rho(u)^{-1}],
 \qquad Y=q-1-u-(1+u)v.
\]
For the latter isomorphism, $YZ=\rho(u)$ and $Yw=uv$ also hold in
$R[\rho(u)^{-1}]$. Hence this chart is a matrix algebra over a
smooth surface.
The change of variables $(x,y,z)=(-1-u,-1-v,1+w)$ turns
\Cref{eq:III-centre-chart} into $xyz+x+y-qz+1+q=0$.

For $IV$, put $\alpha=q_1$, $\beta=q_1q_2=q_3^{-1}$ and
$h=c+bd$. The defining relations become
\begin{equation}\label{eq:IV-centre-relations}
\begin{split}
 r_1&=ebda-eha+eb+da+(1-\alpha)e_1=0,\\
 r_2&=adfh+aeh+fh-\beta ad+(\alpha-\beta)e_2=0,\\
 r_3&=hf+be+(1-\beta)e_3=0.
\end{split}
\end{equation}
Indeed, the braid matrix in \Cref{eq:rainbow-differential} is
\[
 P=\begin{pmatrix}
 e_1&d&e+df\\
 a&e_2+ad&ae+f+adf\\
 b&h&e_3+be+hf
 \end{pmatrix}.
\]
The original relations say that $P^{-1}=Q\mathsf L\mathsf U$,
where $Q=\operatorname{diag}(q_ie_i)$ and $\mathsf L,\mathsf U$
are lower and upper unitriangular. Equivalently,
$P=\mathsf U^{-1}\mathsf L^{-1}Q^{-1}$.
Eliminating from the bottom right gives pivots $\beta e_3$ and
$(\alpha/\beta)e_2$ precisely when $r_3=r_2=0$.
The remaining pivot has inverse $(P^{-1})_{11}=\alpha e_1+r_1$,
so it is $\alpha^{-1}e_1$ precisely when $r_1=0$.

Set $L_1=e_1+eb$, $L_2=\beta e_2-fh$, $L_3=e_3+be$, and put
$d'=L_1d-eh$, $d''=dL_2-eh$. The identity
\[
 \alpha(d'-d'')=d'r_2+r_1d''+er_3h
\]
gives $d'=d''$. Put $L=L_1+L_2+L_3$. The paired products are
\[
\begin{gathered}
 d'a=(\alpha-L)e_1,\qquad ad'=(\alpha-L)e_2,\\
 hf=(\beta-L)e_3,\qquad fh=(\beta-L)e_2,\\
 e(-b)=(1-L)e_1,\qquad(-b)e=(1-L)e_3.
\end{gathered}
\]
By associativity, these equalities and $L_1d=dL_2$ show that $L$
is central. After inverting $L$, the substitution
$d=L^{-1}(d'+eh)$ identifies $A[L^{-1}]$ with the deformed doubled
triangle having oriented cycle $a,h,e$ and the displayed paired
products. Put
\[
 \rho(L)=(1-L)(\alpha-L)(\beta-L),\qquad
 p=da,\qquad v=(\beta-L)e_1-dfb.
\]
In the $e_1$ corner,
\[
 Y=eha=Lp-(\alpha-L),\qquad
 Z=-d'fb=Lv-(\beta-L),\qquad YZ=\rho(L).
\]
The doubled-cycle calculation identifies the centre
of $A[L^{-1}]$ with its $e_1$ corner, and both are isomorphic to the
$L$-localisation of
\begin{equation}\label{eq:IV-centre-chart}
 R=\kk[L,p,v]/\bigl(Lpv-(\beta-L)p-(\alpha-L)v
                         +(\alpha-L)(\beta-L)\bigr).
\end{equation}
On $D(\rho(L))$ the arrows $a,h,e$ are invertible between vertices.
Using $a$ and $e$ to choose matrix units, the original arrows have images
\[
\begin{aligned}
 a&\longmapsto E_{21},& b&\longmapsto(L-1)E_{31},&
 c&\longmapsto(p+L-\alpha)E_{32},\\
 d&\longmapsto pE_{12},& e&\longmapsto E_{13},&
 f&\longmapsto((\beta-L)/Y)E_{23},
\end{aligned}
\]
where $Y=Lp-(\alpha-L)$. Every arrow is expressed in terms of
the central element $L$ and $p$, so the corner is commutative.
These substitutions and their inverse give
\[
 A[\rho(L)^{-1}]\cong\operatorname{Mat}_3
 \kk[L,\rho(L)^{-1},p,Y^{-1}]
 \cong\operatorname{Mat}_3 R[\rho(L)^{-1}],
\]
where $v=(\beta-L)(p-\alpha+L)/Y$.
Substituting $(x,y,z)=(-L,-1-p,1+v)$ in
\Cref{eq:IV-centre-chart} gives the stated cubic.

The two localisations cover, since $\rho(-1)=q\ne0$ for $III$
and $\rho(0)=\alpha\beta\ne0$ for $IV$.
The natural map from $A$ to the product
of these two localisations is injective. The corner identifications
agree on the overlap, as do the unique central lifts of their generators.
Thus restriction induces an isomorphism
\[
 Z(A)\xrightarrow{\sim}e_1Ae_1\cong R,\qquad z\longmapsto e_1ze_1.
\]
Both charts are integral and their overlap is nonempty.
The second chart is smooth. On the first, singularities correspond
to repeated roots of $\rho$: $0,q-1$ for $III$, and
$1,\alpha,\beta$ for $IV$. Their multiplicities give exactly the
stated split $A_1$ and $A_2$ singularities. All three coordinate
rings are hypersurfaces with at most isolated singularities, hence
are normal by Serre's criterion and Gorenstein.

Left multiplication gives an isomorphism
$A\to\operatorname{End}_R(Ae_1)$ on both charts, hence globally.
The chart descriptions also give finiteness of $A$ over $R$,
reflexivity of $Ae_1$, the maximal Cohen-Macaulay condition for $A$,
and global dimension at most two. Thus $A$ is a noncommutative
crepant resolution of $R$.
\end{proof}

\begin{proof}[Proof of \Cref{cor:unipotent-mirror}]
Let $A$ be the algebra in the statement and put $R=Z(A)$ and $M=Ae$,
where $e$ is an extending vertex idempotent for affine $D,E$ and
$e=e_1$ for $II,III,IV$. By
\Cref{prop:affine-de-nccr,prop:nontransverse-centres},
$R$ is a normal Gorenstein surface domain and
$A=\operatorname{End}_R(M)$ is an NCCR, with the same assertions
after any field extension.
Applying \cite[Corollary~3.3]{DITV} to geometric henselian local rings
shows that the geometric singularities are rational double points.
By \cite[proof of Proposition~2.6]{IyamaWemyss} and
\cite[Theorem~4.3]{GustavsenIle}, the module blowup
$\widetilde S=\operatorname{Bl}_M(\operatorname{Spec}R)$ is
geometrically the minimal resolution. Compatibility with flat base
change \cite[Corollary~2.6(ii)]{GustavsenIle} gives a projective
crepant minimal resolution over $\kk$ with smooth source.

Van den Bergh's tilting construction \cite[Corollary~3.2.11]{VdBFlops}
and Morita equivalence of surface NCCRs
\cite[Theorem~1.5(1)]{IyamaWemyss} give
$\Perf(\widetilde S)\simeq\Perf(A)$.
Combining this with \Cref{thm:main} proves the corollary.
\end{proof}

\subsection{Hecke algebras and tensor powers}\label{sec:hecke-tensor}

For our inward-oriented arms, the vertex relations give
$\alpha_{ja}=\prod_{p=1}^{a}q_{(j,p)}^{-1}$, so we label the roots
at the group-algebra point by $\zeta_j^{-a}$.

We consider the specialisation of Etingof's formal orbifold Hecke
algebra for $E/G$ \cite[Definition~3.3 and Example~3.12]{EtingofOrbifold}
in which each defining polynomial has constant term $-1$.
We compare it with the completed central corner of our multiplicative
preprojective algebra. Put
\[
 R_\tau=\C[[\tau_{ja}:1\le j\le s,\ 1\le a<d_j]],
 \qquad \tau_{j0}=-\sum_{a=1}^{d_j-1}\tau_{ja}.
\]
The Hecke algebra is the $(\boldsymbol\tau)$-adic completion of
\[
 \frac{R_\tau\langle\gamma_1^{\pm1},\ldots,\gamma_s^{\pm1}\rangle}
 {\left(\prod_{a=0}^{d_j-1}
 (\gamma_j-\zeta_j^{-a}e^{\tau_{ja}})\ (1\le j\le s),
 \quad\gamma_1\cdots\gamma_s-1\right)}.
\]
The condition $\sum_a\tau_{ja}=0$ fixes the constant term of each
Hecke polynomial to $-1$. These root coordinates are related to the
coefficient parameters $\boldsymbol\hbar$ of
\cite[Definition~6.2]{HKTY} by the formal change of variables in
\cite[Lemma~6.3]{HKTY}.

Rescale the generators and set
\begin{equation}\label{eq:formal-parameter-dictionary}
 U_j=e^{-\tau_{j0}}\gamma_j,\qquad
 q_{(j,p)}(\boldsymbol\tau)
 =\zeta_j e^{\tau_{j,p-1}-\tau_{jp}},\qquad
 q(\boldsymbol\tau)=e^{-\sum_j\tau_{j0}}.
\end{equation}
The roots for $U_j$ are then
$\alpha_{ja}=\zeta_j^{-a}e^{\tau_{ja}-\tau_{j0}}$.
By \eqref{eq:star-central-corner}, the substitution
$\gamma_j=e^{\tau_{j0}}U_j$ identifies the completed Hecke algebra
above with
\[
 \widehat{e_0\Lambda_{R_\tau}^{\boldsymbol q(\boldsymbol\tau)}(Q)e_0}.
\]
The completion is with respect to $(\boldsymbol\tau)$.

For arbitrary nonzero multiplicative preprojective parameters, the
roots on each arm satisfy
\[
 \prod_{a=0}^{d_j-1}\alpha_{ja}
 =\prod_{p=1}^{d_j-1}q_{(j,p)}^{-(d_j-p)}.
\]
After extending the coefficient field if necessary, choose
\[
 \kappa_j^{d_j}=(-1)^{d_j-1}\prod_a\alpha_{ja},\qquad
 \widetilde q=\frac{q}{\prod_j\kappa_j}.
\]
Since $\sum_j(1-1/d_j)=2$, raising these equations to the powers
$N/d_j$ cancels the signs and gives the choice-independent identity
\begin{equation}\label{eq:normalised-parameter}
 \widetilde q^N
   =q^N\prod_{j=1}^s\prod_{p=1}^{d_j-1}
          q_{(j,p)}^{N(d_j-p)/d_j}.
\end{equation}
These exponents are the component multiplicities of $F$, so
\[
 \widetilde q^{\,N}=\mathfrak b([F]).
\]

For the formal family above, we may take
$\kappa_j=e^{-\tau_{j0}}$, giving $\widetilde q=1$.
Thus the family lies on the normalised $\widetilde q=1$ locus,
while its central parameter $q$ varies with the roots.

Keeping $q_{(j,p)}=\zeta_j$ and allowing arbitrary
$q\in\kk^\times$, \cite[Sections~2 and~4.1]{EOR} extend
\Cref{eq:star-group-corner} to
\begin{equation}\label{eq:quantum-corner}
 e_0\Lambda^{\boldsymbol q}(Q)e_0
 \cong\mathcal A_{q^N}\rtimes\mu_N,
\end{equation}
with the cyclic action described there. For $N=2$, take
\[
 X=q^{-1}U_1U_2,\qquad Y=q^{-1}U_2U_3.
\]
Then
\[
 YX=q^2XY,\qquad U_4XU_4=X^{-1},\qquad U_4YU_4=Y^{-1}.
\]
The factors $q^{-1}$ correct the generator normalization in
\cite[Section~4.1]{EOR}.
The presentation isomorphisms are defined over $\ZZ[q^{\pm1}]$, so they apply
over $\kk$. Fullness of the central idempotent follows from
\Cref{lem:star-fullness}; \Cref{thm:main} then gives
\Cref{eq:intro-quantum-mirror}.

The corresponding higher-rank generalised double affine Hecke
algebra $H_n(\boldsymbol u,\tau)$ has parameters
$u_{ja},\tau\in\kk^\times$, $1\le j\le s$, $1\le a\le d_j$.
Its invertible generators $V_1,\ldots,V_s,S_1,\ldots,S_{n-1}$
satisfy the integral presentation
\cite[Definition~3.2.1]{EGO}
\begin{equation}\label{eq:higher-hecke}
\begin{gathered}
 (V_1\cdots V_s)J_n=1,\qquad
 J_n=S_1\cdots S_{n-2}S_{n-1}^2S_{n-2}\cdots S_1,\\
 S_aS_{a+1}S_a=S_{a+1}S_aS_{a+1}\quad(1\le a<n-1),\\
 [S_a,S_b]=0\quad(|a-b|>1),\qquad
 [V_j,S_a]=0\quad(2\le a<n),\\
 [V_j,S_1V_jS_1]=0,\qquad
 [V_i,S_1^{-1}V_jS_1]=0\quad(i<j),\\
 \prod_{a=1}^{d_j}(V_j-u_{ja})=0\quad(1\le j\le s),\qquad
 (S_a-\tau)(S_a+\tau^{-1})=0\quad(1\le a<n).
\end{gathered}
\end{equation}
Here brackets are commutators. For $n=2$, $J_2=S_1^2$; for
$n=1$, omit all $S_a$ and their relations and set $J_1=1$.
The parameter dictionary is
\[
 u_{ja}=\alpha_{j,a-1}\quad(j<s),\qquad
 u_{sa}=q^{-1}\alpha_{s,a-1}.
\]
Thus $V_j=U_j$ for $j<s$ and $V_s=q^{-1}U_s$ identify
$H_1(\boldsymbol u)$ with $e_0\Lambda^{\boldsymbol q}(Q)e_0$.

The quantum parameter in \cite[Section~5.1]{EGO} is
$\prod_{j,a}u_{ja}^{-N/d_j}=\widetilde q^N$.

With $\Lambda_n^{\boldsymbol q}$ as in the introduction, put
\[
 \mathcal B_{Q,n}^{\boldsymbol q}
   =(\mathcal B_Q^{\boldsymbol q})^{\otimes n}\rtimes S_n,
 \qquad p_n=e_0^{\otimes n}.
\]
The dg permutation action uses Koszul signs.

\begin{proposition}\label{prop:tensor-wreath}
Over any field and at every multiplicative parameter,
\[
 \mathcal B_{Q,n}^{\boldsymbol q}\simeq\Lambda_n^{\boldsymbol q},
 \qquad p_n\Lambda_n^{\boldsymbol q}p_n
       \cong H_n(\boldsymbol u,1).
\]
If no nonempty consecutive product of parameters on a leg is
one, $p_n$ is full. At the unipotent point it is not full.
\end{proposition}
\begin{proof}
Tensor the quasi-isomorphism of \Cref{cor:affine-formality}
$n$ times over the field. It commutes with permutations with
Koszul signs. Taking a skew group algebra preserves this
quasi-isomorphism: its underlying complex is a finite direct
sum, so no exactness of invariants is needed. Since $p_n$ is
permutation-invariant, its corner is
$(e_0\Lambda^{\boldsymbol q}(Q)e_0)^{\otimes n}\rtimes S_n$.

At $\tau=1$, the standard identification
\cite[Remark~3.2.2]{EGO} gives
\[
 H_n(\boldsymbol u,1)\cong H_1(\boldsymbol u)^{\otimes n}\rtimes S_n.
\]
It sends $V_j$ to the first tensor factor and $S_a$ to the
adjacent transposition; its inverse places the rank-one
generators in each tensor position by conjugation. These maps
are defined over any field and require no division by $n!$.
Together with the rank-one identification above, this proves
the corner assertion.

Fullness follows by tensoring an expression
$1=\sum a_\ell e_0b_\ell$ supplied by \Cref{lem:star-fullness}.
At the unipotent point, a vertex-simple module supported on an
arm has a nonzero $n$th tensor power with permutation action.
The resulting $\Lambda_n^{\boldsymbol1}$-module is annihilated
by $p_n$, which therefore cannot be full.
\end{proof}

For general interaction parameter $\tau$, an ordinary multiplicative
analogue with corner $H_n(\boldsymbol u,\tau)$ was proposed in
\cite[Introduction]{EGO}. Constructing an interacting dg model and
computing the corresponding wrapped category require further work.

At the group-algebra point let $H_n^{\mathrm{grp}}$ be the
algebra defined by \Cref{eq:higher-hecke} with its polynomial
relations replaced by $V_j^{d_j}=1$ and $S_a^2=1$. This definition needs no roots of unity. The same identification gives
\[
 H_n^{\mathrm{grp}}\cong
 \kk[(\Gamma\rtimes G)^n\rtimes S_n]
 \cong\kk[\Gamma^n]\rtimes G_n.
\]

\subsection{Equivariant mirror symmetry}\label{sec:equivariant}

Let $N\in\{2,3,4,6\}$ and let $E=\C/\Lambda$ admit multiplication
by $\zeta_N=e^{2\pi i/N}$. Set
\[
 X_N=\Res((E\times\C_z)/\mu_N),\qquad
 (w,z)\longmapsto(\zeta_Nw,\zeta_N^{-1}z).
\]
The invariant function $z^N$ induces a proper map $t:X_N\to\C$,
and $dw\wedge dz$ extends to a nonvanishing holomorphic two-form
$\Omega$ on $X_N$. The elliptic curve is arbitrary for $N=2$,
has $j=1728$ for $N=4$, and has $j=0$ for $N=3,6$.
The stabilisers and central fibres are listed below.
\begin{equation}
\begin{array}{c|c|c|c}
 N&F&\text{stabiliser orders}&\text{singularities before resolution}\\ \hline
 2&I_0^*&(2,2,2,2)&4A_1\\
 3&IV^*&(3,3,3)&3A_2\\
 4&III^*&(2,4,4)&A_1+2A_3\\
 6&II^*&(2,3,6)&A_1+A_2+A_5.
\end{array}
\label{eq:geo-quotient-fibres}
\end{equation}
The stabilisers are computed from the kernels of $1-\zeta_N^k$.
The central component $E/\mu_N$ has multiplicity $N$, and an
order-$d$ stabiliser contributes an arm with multiplicities
\[
 \frac Nd,\frac{2N}d,\ldots,\frac{(d-1)N}d
\]
from the tip towards the centre, as shown in \Cref{fig:quotients}.

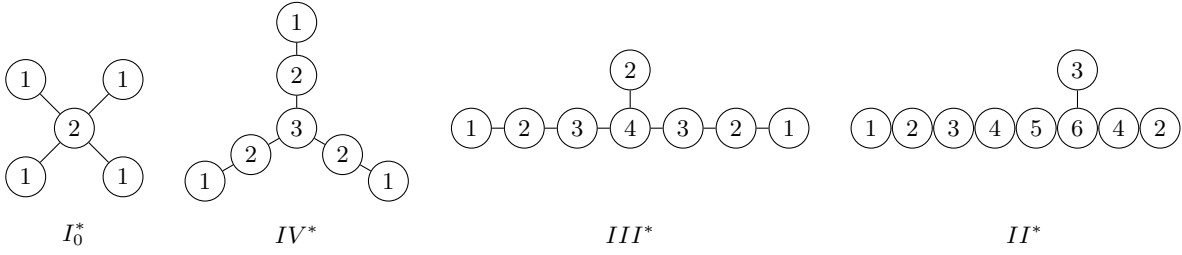
\begin{figure}[ht]
\centering
\begin{tikzpicture}[scale=.7,every node/.style={font=\small},
 v/.style={circle,draw,fill=white,inner sep=2pt,minimum size=15pt}]
\begin{scope}
\node[v](c)at(0,0){$2$};
\foreach \a/\b in {45/a,135/b,225/d,315/e}
 {\node[v](\b)at(\a:1.3){$1$};\draw(c)--(\b);}
\node at(0,-2){$I_0^*$};
\end{scope}
\begin{scope}[xshift=4.2cm]
\node[v](c)at(0,0){$3$};
\foreach \a/\b/\d in {90/a/b,210/d/e,330/f/g}
 {\node[v](\b)at(\a:1){$2$};\node[v](\d)at(\a:2){$1$};
 \draw(c)--(\b)--(\d);}
\node at(0,-2){$IV^*$};
\end{scope}
\begin{scope}[xshift=10.5cm]
\node[v](c)at(0,0){$4$};
\foreach \x/\m/\n in {-3/1/a,-2/2/b,-1/3/d,1/3/e,2/2/f,3/1/g}
 {\node[v](\n)at(\x,0){$\m$};}
\node[v](h)at(0,1.1){$2$};
\draw(a)--(b)--(d)--(c)--(e)--(f)--(g);
\draw(c)--(h);\node at(0,-2){$III^*$};
\end{scope}
\begin{scope}[xshift=18.95cm]
\node[v](c)at(0,0){$6$};
\foreach \x/\m/\n in {-5/1/a,-4/2/b,-3/3/d,-2/4/e,-1/5/f,1/4/g,2/2/h}
 {\node[v](\n)at(\x*.78,0){$\m$};}
\node[v](i)at(0,1.1){$3$};
\draw(a)--(b)--(d)--(e)--(f)--(c)--(g)--(h);
\draw(c)--(i);\node at(-1,-2){$II^*$};
\end{scope}
\end{tikzpicture}
\caption{The numbers are component multiplicities.}
\label{fig:quotients}
\end{figure}

\begin{proposition}\label{prop:quotients}
The surface $X_N$, with $\omega=\Ree\Omega$, admits a
complete Liouville structure exact symplectomorphic to the completed
affine plumbing of type $\widetilde D_4$, $\widetilde E_6$,
$\widetilde E_7$, $\widetilde E_8$ for $N=2,3,4,6$, respectively.
\end{proposition}

\begin{proof}
The real dilation of $z$ lifts to a complete field $Z$ with
$\mathcal L_Z\Omega=\Omega$. Its primitive
$\lambda=\iota_Z\Ree\Omega$ vanishes on $F_{\mathrm{red}}$, and
$Z\log|t|=N>0$ away from $F$. Properness of $t$ identifies $X_N$
with the completion of $(W_\epsilon,\lambda)$. The convex path
joining $\lambda$ to the primitive of
\Cref{prop:neighbourhood-weinstein} has outward-pointing Liouville
fields along $\partial W_\epsilon$, so their completions are exact
symplectomorphic by \cite[Lemma~11.6 and Proposition~11.8]{CE}.
Now apply \Cref{thm:neighbourhood-surgery}(1).
\end{proof}

Let $\mathcal C$ be a split-closed pretriangulated category with
a coherent action of a finite group $G$ and a split-generator $F$.
If $|G|\in\kk^\times$, then
$\operatorname{Ind}F=\bigoplus_{g\in G}gF$ split-generates
$\mathcal C^G$. Indeed, induction expresses
$\operatorname{Ind}(\operatorname{Forget}E)$ in terms of
$\operatorname{Ind}F$, and the coinduction unit divided by $|G|$
splits the counit onto $E$. In particular, if $G$ acts on an algebra $A$, then
\begin{equation}\label{eq:equivariant-perfect}
 \Perf(A)^G\simeq\Perf(A\rtimes G).
\end{equation}
Here the induced free right module has endomorphism algebra
$A\rtimes G$. For commutative $A$, equivariant descent also
identifies these categories with $\Perf([\operatorname{Spec}A/G])$.

Before the torus quotients of \Cref{cor:orbifold-mirror} and
\Cref{thm:higher-orbifold} we work out the simplest equivariant example,
the reflection of the circle.
Let $\operatorname{char}\kk\ne2$ and
$B=\kk\langle u,v\rangle/(u^2-1,v^2-1)$.
For $x=uv$, we have $uxu=x^{-1}$; conversely $v=ux$ recovers
the two-involution presentation. Hence
\begin{equation}\label{eq:involution-algebra}
 B\cong\kk[x^{\pm1}]\rtimes\mu_2,\qquad
 \Perf(B)\simeq\Perf([\mathbb G_{m,\kk}/\mu_2]).
\end{equation}
Its centre is $\kk[x+x^{-1}]$: commuting with $x$ kills the
coefficient of $u$, and commuting with $u$ imposes inversion
invariance. The coarse quotient is an affine line; the stack
also records the stabilisers at $x=\pm1$.

The symplectic model is the reflection
$(\theta,p)\mapsto(-\theta,-p)$ on $(T^*S^1,p\,d\theta)$.
Use the vertical cotangent grading, zero equivariant background,
and trivial $\operatorname{Pin}^+$ structures in the flat frames.
On both a fibre and the zero section choose the reflection lift
$\epsilon$ with $\epsilon^2=1$. A fibre $F$ over a nonfixed
point $x$ gives the equivariant object $F\oplus\sigma F$.
For the flat metric and Hamiltonian $p^2/2$, every relative path
class has a unique geodesic, of index zero. Thus its path-space
Floer model has cohomology in degree zero
\cite{AbouzaidFibre,AbouzaidLoops}.

The equivariant algebra has basis $(g,[\gamma])$, where
$\gamma:gx\to x$, and multiplication
\[
 (g,[\gamma])(h,[\delta])=(gh,[\gamma\circ(g\delta)]).
\]
For a reflection-fixed point $x_0$ and a path $\eta:x_0\to x$,
the paths $\eta\circ(g\eta)^{-1}$ split the extension by
$\pi_1(S^1)$. This identifies the path group with $\ZZ\rtimes\mu_2$
and its algebra with $B$. Since the cohomology lies in degree zero,
the algebra is formal.
Fibre generation and the preceding induction argument give
\[
 \mathcal W_{\mu_2}(T^*S^1)\simeq\Perf(B)
       \simeq\Perf([\mathbb G_{m,\kk}/\mu_2]).
\]

For the torus cotangent categories, $\mathcal W_{G_n}$ denotes the
category of coherent $G_n$-equivariant objects in
$\mathcal W(T^*\mathbb{T}_n)$. By coherent equivariant objects we mean
homotopy equivariant objects, in the sense discussed in
\cite[Section~6.1]{AurouxSmith}. We use the canonical cotangent grading
and the equivariant relative spin background $\pi^*T\mathbb{T}_n$,
where $\pi:T^*\mathbb{T}_n\to\mathbb{T}_n$ is the projection.
For $n=1$, this gives the conventions for $\mathcal W_G(T^*\mathbb{T})$.

\begin{proof}[Proof of \Cref{cor:orbifold-mirror} and \Cref{thm:higher-orbifold}]
Give $\mathbb{T}_n$ a product of $G$-invariant flat metrics. Each factor
of $G_n=G^n\rtimes S_n$ preserves orientation, and swapping two
real two-dimensional blocks has sign $(-1)^4=1$.
The cotangent-fibre equivalence \cite{AbouzaidFibre,AbouzaidLoops}
and its coherent naturality for orientation-preserving base
diffeomorphisms \cite[Theorems~1.0.12 and~6.0.1]{OhTanaka}
therefore apply to this action. These constructions allow the
coefficient field $\kk$ with the stated brane data
\cite[Choice~1.2.1]{OhTanaka}.

The chosen relative spin background allows the $G_n$-action on the
cotangent fibre to lift to its relative spin structure compatibly
with the group law
\cite[Example~3.3]{HKTY}. At the
fixed origin, put $V=T_0\mathbb{T}_n$. The relative spin
representation is $V^*\oplus V\cong V\oplus V$, and the diagonal
$SO(V)$-action lifts to $\operatorname{Spin}(V\oplus V)$:
its map on fundamental groups is twice the usual map and hence
zero modulo two. This gives coherent lifts of the $G_n$-action,
including the permutations.

Projection to components gives an equivariant quasi-isomorphism
\[
 C_{-*}(\Omega_0\mathbb{T}_n;\kk)\longrightarrow\kk[\Gamma^n],
 \qquad g(X^\gamma)=X^{g\gamma},
\]
because each based-loop component is contractible. Since $N^n n!$
is invertible,
\Cref{eq:equivariant-perfect} applies: the induced fibre generates
and has endomorphism algebra $\kk[\Gamma^n]\rtimes G_n$.
Equivariant descent proves the torus-stack equivalence of
\Cref{thm:higher-orbifold}.

If the $\zeta_j$ belong to $\kk$, the group corner of
\Cref{eq:star-group-corner} is full by \Cref{lem:star-fullness}.
\Cref{prop:tensor-wreath} makes $p_n$ full with corner
$H_n^{\mathrm{grp}}$, proving the remaining equivalence.
For $n=1$, \Cref{thm:main} at $\mathfrak b_{\mathrm{grp}}$
also identifies the category of the smooth surface, which proves
\Cref{cor:orbifold-mirror}.
\end{proof}

\subsection{Hilbert schemes}\label{sec:hilbert-schemes}

Recall that $Y_n=\operatorname{Hilb}^n(M_Q)$, and that
\[
 h_n:Y_n\longrightarrow\operatorname{Sym}^n(M_Q)
       \longrightarrow\operatorname{Sym}^n(\C)
\]
is induced by $f:M_Q\to\C$. Let $\Omega_n$ be the induced
holomorphic symplectic form and $V_n$ the generator of positive
real cotangent scaling.

\begin{proposition}\label{prop:hilbert-liouville}
The form $\lambda_n=\iota_{V_n}\Ree\Omega_n$ makes $Y_n$ a
finite-type Liouville manifold of real dimension $4n$, with
skeleton the underlying set of $h_n^{-1}(0)$.
\end{proposition}
\begin{proof}
The Hilbert scheme of a smooth holomorphic symplectic surface
is smooth and holomorphic symplectic by the local construction in
\cite[Section~6, Proposition~5]{BeauvilleHilbert}. Cotangent scaling lifts
to it and has weight one on $\Omega_n$. Cartan's formula gives
$d\lambda_n=\Ree\Omega_n$, with complete Liouville vector field
$V_n$ induced by positive real scaling.
Hilbert-Chow is projective and $f$ is proper, so $h_n$ is
proper. If $b_1,\ldots,b_n$ are elementary symmetric coordinates
on its base, their weights are $N,2N,\ldots,nN$. Hence
\[
 \rho_n=\sum_{j=1}^n|b_j\circ h_n|^2,\qquad
 V_n\rho_n=2N\sum_{j=1}^nj|b_j\circ h_n|^2>0
       \quad\text{off }h_n^{-1}(0).
\]
Every positive sublevel is a compact Liouville domain. Outside
the zero fibre, each flow line crosses every positive level
exactly once; its forward trajectory escapes every compact set.
The zero fibre is compact and invariant. Thus the completion is
$Y_n$ and the skeleton is precisely that fibre.
\end{proof}

The Hilbert-Chow map introduces exceptional curves where points
collide. In \Cref{conj:hilbert-mirror}, we assign weight $-1$ to a
collision line and retain weight $1$ on classes inherited from
$M_Q$.

Near a configuration with one double point and all other points
distinct, Hilbert-Chow is locally an $A_1$ resolution times a smooth
factor. The normalised Hecke generator $\tau^{-1}S_a$ has roots
$1,-\tau^{-2}$. Comparing these with the local $A_1$ leg roots
$1,t_{\mathrm{coll}}^{-1}$ gives
\[
 t_{\mathrm{coll}}=-\tau^2.
\]
Thus the tensor-product case $\tau=1$ suggests collision weight
$-1$, whereas collision weight $1$ would give $\tau^2=-1$.
The same sign occurs at zero orbifold parameter in the
Hilbert-Chow crepant resolution correspondence
\cite[Theorem~3.11]{BryanGraber}.

Recall that the tautological rank-$n$ bundle $\mathcal V_n$ on $Y_n$
has fibre $H^0(Z,\mathcal O_Z)$ at a length-$n$ subscheme $Z$.
The collision class of \Cref{eq:collision-bulk} is
\[
 \mathfrak b_{\mathrm{coll}}(\gamma)
 =(-1)^{\langle c_1(\det\mathcal V_n),\gamma\rangle},
 \qquad \gamma\in H_2(Y_n;\ZZ).
\]
Its values can be computed directly. For $n\ge2$, fix $n-2$ distinct
points and a further point $x$. The length-two subschemes supported
at $x$ form a collision line $\ell=\PP(T_xM_Q)$. Along this line,
$\mathcal V_n$ is the sum of $n-2$ trivial summands and the rank-two
bundle fitting into the evaluation sequence
\[
 0\longrightarrow\mathcal O_\ell(1)
 \longrightarrow\mathcal V_2|_\ell
 \longrightarrow\mathcal O_\ell\longrightarrow0.
\]
Hence $\deg(\det\mathcal V_n|_\ell)=1$ and
$\mathfrak b_{\mathrm{coll}}([\ell])=-1$.

To represent a class coming from $H_2(M_Q;\ZZ)$, choose $n-1$ distinct
fixed points disjoint from a surface representing the given two-cycle
in $M_Q$, and let the remaining point move along that surface.
The tautological bundle is trivial along this cycle: a basis is given
by the functions that are $1$ at one labelled point and $0$ at the
others. Thus $c_1(\det\mathcal V_n)$ evaluates to zero, and the
collision weight is $1$.

For \Cref{conj:hilbert-mirror}, we use the relative spin background
\[
 \beta_n=c_1(\det\mathcal V_n)\bmod2\in H^2(Y_n;\ZZ/2).
\]
An oriented brane $L$ carries a grading, a local system, and a spin
structure on $TL\oplus(\det\mathcal V_n)_{\R}|_L$; the obstruction
condition is $w_2(TL)=\beta_n|_L$. For the grading, choose a compatible
almost hyperk\"ahler structure extending $\Omega_n$ and use the complex
volume form of the rotated almost complex structure compatible with
$\Ree\Omega_n$. Where comparison with background zero is defined,
the change of orientations contributes
$(-1)^{\langle\beta_n,\gamma\rangle}
=\mathfrak b_{\mathrm{coll}}(\gamma)$ for a compatibly capped class
$\gamma$ \cite[Section~2.2 and Lemma~A.5]{AbouzaidLoops}.
We therefore insert no additional scalar copy of this sign.

This background admits the compact Lagrangians expected from the
surface mirror. For a core sphere $C\cong\PP^1$, put
$L=\operatorname{Hilb}^n(C)\cong\PP^n$, which lies in the compact core.
The form $\Omega_n$ vanishes
on the dense locus of distinct points of $C$, hence on all of $L$,
so $L$ is Lagrangian. It is exact because $H^1(L;\R)=0$.
The universal divisor on $\PP^1\times\PP^n$ has bidegree $(n,1)$;
pushing forward its defining sequence gives
$\det\mathcal V_n|_L\cong\mathcal O(1-n)$. Thus, for the mod-two
hyperplane class $h$,
\[
 \beta_n|_L=(n-1)h=(n+1)h=w_2(T\PP^n).
\]
Consequently $L$ is relatively spin for $\beta_n$, including even $n$,
when it is not spin for background zero.

There is also a check on endomorphisms. When $n!\in\kk^\times$,
the $S_n$-equivariant external tensor power of the mirror of $C$ has
endomorphism cohomology
\[
 \bigl(H^*(S^2;\kk)^{\otimes n}\bigr)^{S_n}
 \cong\kk[z]/(z^{n+1}),\qquad |z|=2.
\]
Indeed, if $x_i$ is the degree-two class in the $i$th factor, then
$z=x_1+\cdots+x_n$ satisfies $z^j=j!e_j(x_1,\ldots,x_n)$, where
$e_j$ is the $j$th elementary symmetric polynomial. The classes
$e_0,\ldots,e_n$ form a basis of the invariant subspace.
This agrees with the self-Floer cohomology $H^*(\PP^n;\kk)$ of $L$.

We now identify the algebraic category on the right-hand side of
\Cref{conj:hilbert-mirror}.

\begin{proposition}\label{prop:hilbert-mirror-model}
Under the field hypotheses of \Cref{conj:hilbert-mirror},
\[
 \Perf(\operatorname{Hilb}^n\widetilde S_Q)
       \simeq\Perf(\Lambda_n^{\boldsymbol1}).
\]
\end{proposition}
\begin{proof}
Put $Y=\widetilde S_Q$ and $A=\Lambda^{\boldsymbol1}(Q)$.
The tilting construction and Morita equivalence used in the proof
of \Cref{cor:unipotent-mirror} give a tilting bundle $\mathcal T$
on $Y$ with endomorphism algebra $A$. Its external power
$\mathcal T^{\boxtimes n}$ is a tilting generator on $Y^n$,
with endomorphism algebra $A^{\otimes n}$ and the natural
permutation linearisation. Since $n!\in\kk^\times$, the
equivariant-generation argument above gives
\[
 \Perf([Y^n/S_n])
 \simeq\Perf(A^{\otimes n}\rtimes S_n)
 =\Perf(\Lambda_n^{\boldsymbol1}).
\]
The surface $Y$ is smooth and quasiprojective, being projective
over the affine surface $S_Q$, and $\operatorname{char}\kk$ is
zero or greater than $n$. The Hilbert-scheme derived McKay
correspondence \cite[Theorem~3.11 and Corollary~3.12]{Groechenig}
therefore gives
\[
 \Perf(\operatorname{Hilb}^nY)\simeq\Perf([Y^n/S_n]).
\]
The universal-kernel construction is defined over $\kk$, and the
equivalence can be checked after faithfully flat extension to
$\overline\kk$.
Here smoothness and tameness allow passage to perfect categories.
Combining the equivalences proves the proposition.
\end{proof}

This identifies the algebraic mirror. A Floer comparison across
the Hilbert-Chow exceptional locus with the specified collision
background remains the assertion of \Cref{conj:hilbert-mirror}.
\printbibliography
\enlargethispage{3\baselineskip}

\vspace{\fill} 
\noindent
\begin{tabular}{@{}l}
  \textit{Department of Mathematics,} \\
  \textit{Imperial College,} \\
  \textit{London, UK}
\end{tabular}

\end{document}